\documentclass[generic]{imsart}

\RequirePackage[OT1]{fontenc}
\RequirePackage{amsthm, amsmath, amsfonts, amssymb}
\RequirePackage{amsthm, amsmath, amssymb, amsfonts, graphicx, color}
\RequirePackage[square]{natbib}
\RequirePackage[colorlinks,citecolor=blue,urlcolor=blue]{hyperref}
\RequirePackage{hypernat}

\RequirePackage{xr} 
\RequirePackage{graphicx, color}

\startlocaldefs

\newcommand{\MyHess}{{\nabla^2 \mathstrut}}

\def\RR{\mathbb{R}}

\DeclareMathOperator{\var}{Var}
\DeclareMathOperator{\cov}{Cov}
\DeclareMathOperator{\Hess}{Hess}
\DeclareMathOperator{\Sym}{Sym}
\DeclareMathOperator{\leb}{Leb}

\newcounter{assumI} \renewcommand\theassumI {\text{I.\arabic{assumI}}}
	
\newcounter{assumD} \renewcommand\theassumD {\text{D.\arabic{assumD}}}

\newtheorem{theorem}{Theorem}[section]

\newtheorem{corollary}[theorem]{Corollary}

\newtheorem{definition}[theorem]{Definition}
\newtheorem{example}[theorem]{Example}

\newtheorem{lemma}[theorem]{Lemma}

\newtheorem{proposition}[theorem]{Proposition}
\newtheorem{remark}[theorem]{Remark}

\endlocaldefs

\begin{document}

\begin{frontmatter}
\title{Log-Concavity and Strong Log-Concavity: \\ a review} 
\runtitle{Log-Concavity and Strong Log-Concavity}

\begin{aug}
    \author{\fnms{Adrien} \snm{Saumard}\thanksref{t1}\ead[label=e1]{asaumard@gmail.com}} 
	\and 
    \author{\fnms{Jon A.} \snm{Wellner}\thanksref{t2}\ead[label=e2]{jaw@stat.washington.edu}}
	
	\thankstext{t1}{Research supported in part by NI-AID grant 2R01 AI29168-04, a PIMS post-doctoral
                               fellowship and post-doctoral Fondecyt Grant 3140600}
	\thankstext{t2}{Research supported in part by NSF Grant DMS-1104832, NI-AID
                                grant 2R01 AI291968-04, and the Alexander von Humboldt Foundation}
	    
	\runauthor{Adrien Saumard and Jon A. Wellner}

	\affiliation{University of Washington}

	\address{Departamento de Estad\'{i}stica, CIMFAV\\Universidad de Valpara\'{i}so,  Chile\\
	     %Department of Statistics, Box 354322\\University of Washington\\Seattle, WA  98195-4322\\
	\printead{e1}}
	\address{Department of Statistics, Box 354322\\University of Washington\\Seattle, WA  98195-4322\\
	\printead{e2}}     
\end{aug}

\begin{abstract}
We review and formulate results concerning log-concavity and strong-log-concavity in both discrete 
and continuous settings.  
We show how preservation of log-concavity and strongly log-concavity 
on $\RR$ under convolution follows from a fundamental monotonicity result of Efron (1969).
%We put forward that a monotonicity property for log-concave measures may be used
%to prove stability under convolution for (strongly) log-concave measures. 
We provide a new proof of Efron's
theorem using the recent asymmetric Brascamp-Lieb inequality due to Otto and Menz (2013).
Along the way we review connections between log-concavity and other areas of mathematics 
and statistics, including concentration of measure, log-Sobolev inequalities, convex geometry,
MCMC algorithms, Laplace approximations, and machine learning.
\end{abstract}



\end{frontmatter}

\tableofcontents
\vspace{1cm}

%\tableofcontents
%\vspace{1cm}

%----------------------------------------------------------
\section{Introduction: log-concavity}
\label{sec:intro}

%----------------------------------------------------------

Log-concave distributions and various properties related to log-concavity
play an increasingly important role in probability, statistics, optimization
theory, econometrics and other areas of applied mathematics. In view of
these developments, the basic properties and facts concerning log-concavity
deserve to be more widely known in both the probability and statistics
communities. Our goal in this survey is to review and summarize the basic
preservation properties which make the classes of log-concave densities,
measures, and functions so important and useful. In particular we review
preservation of log-concavity and ``strong log-concavity'' (to be defined
carefully in section~\ref{sec:basicDefinitions}) under marginalization,
convolution, formation of products, and limits in distribution. The
corresponding notions for discrete distributions (log-concavity and ultra
log-concavity) are also reviewed in section~\ref{Sec:DiscreteLogConc}.

A second goal is to acquaint our readers with a useful monotonicity theorem
for log-concave distributions on $\mathbb{R}$ due to \cite{MR0171335}, 
%Efron
and to briefly discuss connections with recent progress concerning
``asymmetric'' Brascamp-Lieb inequalities. Efron's theorem is reviewed in
Section~\ref{subsec:EfronMonoThm}, and further applications are given in the
rest of Section~\ref{sec:EfronsTheoremOneDimension}.

There have been several reviews of developments connected to log-concavity
in the mathematics literature, most notably \cite{MR588074} and \cite%
{MR1898210}. We are not aware of any comprehensive review of log-concavity
in the statistics literature, although there have been some review type
papers in econometrics, in particular \cite{An:98} %  An (1998)
and \cite{Bagnoli:95}. %Bagnoli and Bergstrom (2005)
Given the pace of recent advances, it seems that a review from a statistical
perspective is warranted.

Several books deal with various aspects of log-concavity: the classic books
by \cite{MR552278} (see also \cite{MR2759813}) and \cite{MR954608} both
cover aspects of log-concavity theory, but from the perspective of
majorization in the first case, and a perspective dominated by unimodality
in the second case. Neither treats the important notion of strong
log-concavity. The recent book by \cite{MR2814377} perhaps comes closest to
our current perspective with interesting previously unpublished material
from the papers of Brascamp and Lieb in the 1970's and a proof of the
Brascamp and Lieb result to the effect that strong log-concavity is
preserved by marginalization. Unfortunately Simon does not connect with
recent terminology and other developments in this regard and focuses on
convexity theory more broadly. \cite{MR1964483} % Villani (2003)
 (chapter 6) gives a nice treatment of the Brunn-Minkowski inequality and
related results for log-concave distributions and densities with interesting
connections to optimal transportation theory. His chapter 9 also gives a
nice treatment of the connections between log-Sobolev inequalities and
strong log-concavity, albeit with somewhat different terminology. \cite%
{MR1849347} % Ledoux (2001) 
is, of course, a prime source for material on log-Sobolev inequalities and
strong log concavity. The nice book on stochastic programming by \cite%
{MR1375234} % Prekopa (1995)
has its chapter 4 devoted to log-concavity and $s-$concavity, but has no
treatment of strong log-concavity or inequalities related to log-concavity
and strong log-concavity. % Boucheron Lugosi and Massart (2013)
%\cite{MR1964483},   % Villani (2003)
%\cite{MR2459454},   % Villani (2009)
%\cite{MR954608},     % Dharmadhikari, Sudhakar and Joag-Dev, Kumar (1988)
%\cite{MR1849347},    % Ledoux (2001)
%\cite{An:98},             %  An (1998)
%\cite{Bagnoli:95},    %Bagnoli and Bergstrom (2005)
%Marshall, Olkin, and Arnold 2011
%Marshall and Olkin 1979 .
%\cite{MR2814377} ,  % Simon   
%Important review papers:  
%\cite{MR588074},  %DasGupta (1980)
%\cite{MR1898210},   %Gardner 
In this review we will give proofs some key results in the body of the review, 
while proofs of supporting results are postponed to Section~\ref{sec:Proofs} (Appendix B).

%----------------------------------------------------------
\section{Log-concavity and strong log-concavity: definitions and basic results}
\label{sec:basicDefinitions} 
%----------------------------------------------------------

We begin with some basic definitions of
log-concave densities and measures on $\mathbb{R}^d$.

\begin{definition}
(0-d): A density function $p$ with respect to Lebesgue measure $\lambda$ on $%
(\mathbb{R}^d, \mathcal{B}^d)$ is log-concave if $p = e^{-\varphi}$ where $%
\varphi$ is a convex function from $\mathbb{R}^d$ to $(-\infty, \infty]$.
Equivalently, $p$ is log-concave if $p = \exp (\tilde{\varphi} )$ where $%
\tilde{\varphi} = - \varphi $ is a concave function from $\mathbb{R}^d$ to $%
[-\infty, \infty)$.
\end{definition}

We will usually adopt the convention that $p$ is lower semi-continuous and $%
\varphi = - \log p$ is upper semi-continuous. Thus $\{x \in \mathbb{R}^d : \
p(x) > t \}$ is open, while $\{ x \in \mathbb{R}^d : \ \varphi(x) \le t \}$
is closed. We will also say that a non-negative and integrable function $f $ from 
$\mathbb{R}^d$ to $[0,\infty)$ is log-concave if $f = e^{-\varphi}$ where $\varphi $
is convex even though $f$ may not be a density; that is 
$\int_{\mathbb{R}^d} f d \lambda \in (0,\infty)$.

Many common densities are log-concave; in particular all Gaussian densities 
\begin{eqnarray*}
p_{\mu, \Sigma} (x) = (2 \pi | \Sigma |)^{-d/2} \exp \left ( - \frac{1}{2}
(x - \mu)^T \Sigma^{-1} (x - \mu ) \right )
\end{eqnarray*}
with $\mu \in \mathbb{R}^d$ and $\Sigma$ positive definite are log-concave,
and 
\begin{eqnarray*}
p_C(x) = 1_C (x) / \lambda (C)
\end{eqnarray*}
is log-concave for any non-empty, open and bounded convex subset $C \subset 
\mathbb{R}^d$. % with $0 < \lambda (C) < \infty$.
With $C$ open, $p$ is lower semi-continuous in agreement with our convention
noted above; of course taking $C$ closed leads to upper semi-continuity of $%
p $.

In the case $d=1$, log-concave functions and densities are related to
several other important classes. The following definition goes back to the
work of P\'olya and Schoenberg.

\begin{definition}
\label{PolyaFrequencyOfOrderK} Let $p$ be a function on $\mathbb{R}$ (or
some subset of $\mathbb{R}$), and let $x_1 < \cdots < x_k$, $y_1 < \cdots <
y_k$. Then $p$ is said to be a P\'olya frequency function of order $k$ (or $%
p \in PF_k$) if $\mbox{det} ( p(x_i - y_j ) ) \ge 0$ for all such choices of
the $x$'s and $y$'s. If $p$ is $PF_k$ for every $k$, then $p \in PF_{\infty}$%
, the class of P\'olya frequency functions of order $\infty$.
\end{definition}

A connecting link to P\'olya frequency functions and to the notion of
monotone likelihood ratios, which is of some importance in statistics, is
given by the following proposition:

\begin{proposition}
\label{LogConcaveEqualsPF2} $\phantom{blab}$\newline
(a) The class of log-concave functions on $\mathbb{R}$ coincides with the
class of P\'olya frequency functions of order $2$.\newline
(b) A density function $p$ on $\mathbb{R}$ is log-concave if and only if the
translation family $\{ p( \cdot - \theta ) : \ \theta \in \mathbb{R} \}$ has
monotone likelihood ratio: i.e. for every $\theta_1 < \theta_2 $ the ratio $%
p(x-\theta_2)/p(x-\theta_1)$ is a monotone nondecreasing function of $x$.
\end{proposition}

\begin{proof}  
See Section~\ref{sec:Proofs}.
\end{proof}

\smallskip

\begin{definition}
(0-m): A probability measure $P$ on $(\mathbb{R}^d, \mathcal{B}^d)$ is
log-concave if for all non-empty sets $A,B \in \mathcal{B}^d$ and for all $0
< \theta < 1$ we have 
\begin{eqnarray*}
P( \theta A + (1-\theta) B) \ge P(A)^{\theta} P(B)^{1-\theta} .
\end{eqnarray*}
\end{definition}

It is well-known that log-concave measures have sub-exponential tails, see 
\cite{MR733944} and Section \ref{section_reg} below. To accommodate
densities having tails heavier than exponential, the classes of $s-$concave
densities and measures are of interest.

\begin{definition}
\label{defn:Sconcave} (s-d): A density function $p$ with respect to Lebesgue
measure $\lambda $ on an convex set $C \subset \mathbb{R}^d$ is $s-$concave
if 
\begin{eqnarray*}
p( \theta x + (1-\theta) y) \ge M_s (p(x), p(y); \theta )
\end{eqnarray*}
where the generalized mean $M_s (u,v; \theta)$ is defined for $u,v\ge 0$ by 
\begin{eqnarray*}
M_s (u,v; \theta) \equiv \left \{ 
\begin{array}{ll}
(\theta u^s + (1-\theta) v^s )^{1/s}, & s \not= 0, \\ 
u^{\theta} v^{1-\theta}, & s = 0, \\ 
\mbox{min} \{ u, v \}, & s = -\infty, \\ 
\mbox{max} \{ u,v\} , & s = + \infty.%
\end{array}
\right .
\end{eqnarray*}
\end{definition}

\begin{definition}
(s-m): A probability measure $P $ on $(\mathbb{R}^d, \mathcal{B}^d)$ is $s-$%
concave if for all non-empty sets $A,B $ in $\mathcal{B}^d$ and for all $%
\theta \in (0,1)$, 
\begin{eqnarray*}
P( \theta A + (1-\theta ) B) \ge M_s (P(A), P(B); \theta )
\end{eqnarray*}
where $M_s (u,v; \theta)$ is as defined above.
\end{definition}

These classes of measures and densities were studied by \cite{MR0404557} 
%Prekopa (1973)
in the case $s=0$ and for all $s \in \mathbb{R}$ by \cite{MR0450480}, 
%Brascamp and Lieb (1975), 
\cite{MR0404559}, \cite{MR0388475}, %Borell (1975)
and \cite{MR0428540}. %Rinott (1976). MR0428540
The main results concerning these classes are nicely summarized by \cite%
{MR954608}; see especially sections 2.3-2.8 (pages 46-66) and section 3.3
(pages 84-99). In particular we will review some of the key results for
these classes in the next section. For bounds on densities of $s-$%
concave distributions on $\mathbb{R}$ see \cite{Doss-Wellner:2013}; 
% Doss & W (2013);
for probability tail bounds for $s-$concave measures on $\mathbb{R}^d$, see 
\cite{MR2510011}. For moment bounds and concentration inequalities for $s-$%
concave distributions with $s<0$ see \cite{MR3005719} 
%Adamczak et al (2012)
and \cite{Guedon:12}, section 3.

A key theorem connecting probability measures to densities is as follows:

\begin{theorem}
\label{PrekopaRinott} Suppose that $P$ is a probability measure on $(\mathbb{%
R}^d, \mathcal{B}^d)$ such that the affine hull of $\mbox{supp} (P)$ has
dimension $d$. Then $P$ is a log-concave measure if and only if it has a
log-concave density function $p$ on $\mathbb{R}^d$; that is $p = e^{\varphi}$
with $\varphi$ concave satisfies 
\begin{equation*}
P(A) = \int_A p d \lambda \ \ \ \mbox{for} \ \ \ A \in \mathcal{B}^d .
\end{equation*}
\end{theorem}

For the correspondence between $s-$concave measures and $t-$concave
densities, see \cite{MR0404559}, \cite{MR0450480} section 3, \cite{MR0428540}%
, %Borell 
and \cite{MR954608}. %Dharmadkiri & Joag-dev section 3.3 .
\medskip

One of our main goals here is to review and summarize what is known
concerning the (smaller) classes of (what we call) \textsl{strongly
log-concave} densities. This terminology is not completely standard. Other
terms used for the same or essentially the same notion include:

\begin{itemize}
\item \textsl{Log-concave perturbation of Gaussian}; \cite{MR1964483}, 
% Villani 
\cite{MR1800860}. %Caffarelli
pages 290-291.

\item \textsl{Gaussian weighted log-concave}; \cite{MR0450480} pages 379,
381.

\item \textsl{Uniformly convex potential}: 
\cite{MR1800062}, abstract and page 1034, \cite{MR2895086}, Section 7. %Bobkov and Ledoux

\item \textsl{Strongly convex potential}: \cite{MR1800860}. 
%Cafarelli (2000)
%\item
%{\sl Semi log-concave probability measure}:   \cite{CattiauxGuillin:13}.    %Cattiaux and Guillin (2013)  
\end{itemize}

In the case of real-valued discrete variables the comparable notion is
called \textsl{ultra log-concavity}; see e.g. \cite{MR1462561}, % Liggett
\cite{MR3030616}, and \cite{MR2327839}. We will re-visit the notion of ultra
log-concavity in Section~\ref{Sec:DiscreteLogConc}. \smallskip

Our choice of terminology is motivated in part by the following definition
from convexity theory: following \cite{MR1491362}, page 565, we say that a
proper convex function $h:\mathbb{R}^{d}\rightarrow \overline{\mathbb{R}}$
is \textsl{strongly convex} if there exists a positive number $c$ such that 
\begin{equation*}
h(\theta x+(1-\theta )y)\leq \theta h(x)+(1-\theta )h(y)-\frac{1}{2}c\theta
(1-\theta )\Vert x-y\Vert ^{2}
\end{equation*}
for all $x,y\in \mathbb{R}^{d}$and $\theta \in (0,1)$. It is easily seen
that this is equivalent to convexity of $h(x)-(1/2)c\Vert x\Vert ^{2}$ (see 
\cite{MR1491362}, Exercise12.59, page 565).

Thus our first definition of strong log-concavity of a density function $p$
on $\mathbb{R}^d$ is as follows:

\begin{definition}
\label{strongLogConcaveDefn1} For any $\sigma ^{2}>0$ define the class of
strongly log-concave densities with variance parameter $\sigma ^{2}$, or $%
SLC_{1}(\sigma ^{2},d)$ to be the collection of density functions $p$ of the
form %Now $f$ is strongly log-concave if and only 
\begin{equation*}
p(x)=g(x)\phi _{\sigma ^{2}I}(x)
\end{equation*}%
for some log-concave function $g$ where, for a positive definite matrix $%
\Sigma $ and $\mu \in \mathbb{R}^{d}$, $\phi _{\Sigma }(\cdot -\mu )$
denotes the $N_{d}(\mu ,\Sigma )$ density given by 
\begin{equation}
\phi _{\Sigma }(x-\mu )=(2\pi |\Sigma |)^{-d/2}\exp \left( -\frac{1}{2}%
(x-\mu )^{T}\Sigma ^{-1}(x-\mu )\right) .  \label{def_Gaussian}
\end{equation}

If a random vector $X$ has a density $p$ of this form, then we also say that 
$X$ is strongly log-concave.
\end{definition}

Note that this agrees with the definition of strong convexity given above
since, 
\begin{equation*}
h(x)\equiv -\log p(x)=-\log g(x)+d\log (\sigma \sqrt{2\pi })+\frac{|x|^{2}}{%
2\sigma ^{2}},
\end{equation*}%
so that 
\begin{equation*}
-\log p(x)-\frac{|x|^{2}}{2\sigma ^{2}}=-\log g(x)+d\log (\sigma \sqrt{2\pi }%
)
\end{equation*}%
is convex; i.e. $-\log p(x)$ is strongly convex with $c=1/\sigma ^{2}$.
Notice however
that if $p \in SLC_{1}(\sigma ^{2},d)$ then larger values of $\sigma^2$ corresp to smaller 
values of $c = 1/\sigma^2$, and hence $p$ becomes {\sl less} strongly log-concave 
as $\sigma^2$ increases.  Thus in our definition of strong log-concavity the 
coefficient $\sigma^2$ measures the ``flatness'' of the convex potential
%the larger is $\sigma ^{2}$ the lower is $c=1/\sigma ^{2}$
%and by consequence, the less strongly log-concave is $p$. 
%In our definition of strong log-concavity,
%the coefficient $\sigma ^{2}$ thus measures the flatness of the convex potential.

It will be useful to relax this definition in two directions: by allowing
the Gaussian distribution to have a non-singular covariance matrix $\Sigma$
other than the identity matrix and perhaps a non-zero mean vector $\mu$.
Thus our second definition is as follows.

\begin{definition}
\label{strongLogConcaveDefn2} Let $\Sigma $ be a $d\times d$ positive
definite matrix and let $\mu \in \mathbb{R}^{d}$. We say that a random
vector $X$ and its density function $p$ are strongly log-concave and write $%
p\in SLC_{2}(\mu ,\Sigma ,d)$ if 
\begin{equation*}
p(x)=g(x)\phi _{\Sigma }(x-\mu )\ \ \ \mbox{for}\ \ \ x\in \mathbb{R}^{d}
\end{equation*}%
for some log-concave function $g$ where $\phi _{\Sigma }(\cdot -\mu )$
denotes the $N_{d}(\mu ,\Sigma )$ density given by (\ref{def_Gaussian}). 
%\begin{eqnarray*}
%\phi_{\Sigma} (x- \mu) = (2 \pi | \Sigma |)^{-d/2} \exp \left ( - \frac{1}{2} (x - \mu)^T \Sigma^{-1} (x- \mu) \right ) .
%\end{eqnarray*}
\end{definition}

Note that $SLC_{2}(0,\sigma ^{2}I,d)=SLC_{1}(\sigma ^{2},d)$ as in
Definition~\ref{strongLogConcaveDefn1}. Furthermore, if $p\in SLC_{2}(\mu
,\Sigma ,d)$ with $\Sigma $ non-singular, then we can write 
\begin{eqnarray*}
p(x) &=&g(x)\frac{\phi _{\Sigma }(x-\mu )}{\phi _{\Sigma }(x)}\cdot \frac{%
\phi _{\Sigma }(x)}{\phi _{\sigma ^{2}I}(x)}\phi _{\sigma ^{2}I}(x) \\
&=&g(x)\exp (\mu ^{T}\Sigma ^{-1}x-(1/2)\mu ^{T}\Sigma ^{-1}\mu ^{T}) \\
&&\ \ \ \cdot \exp \left( -\frac{1}{2}x^{T}(\Sigma ^{-1}-\frac{1}{\sigma ^{2}%
}I)x\right) \cdot \phi _{\sigma ^{2}I}(x) \\
&\equiv &h(x)\phi _{\sigma ^{2}I}(x)\text{ ,}
\end{eqnarray*}%
where $\Sigma ^{-1}-I/\sigma ^{2}$ is positive definite if $1/\sigma ^{2}$
is smaller than$\ $the smallest eigenvalue of $\Sigma ^{-1}$. In this case, $%
h$ is log-concave, so $p\in SLC_{1}(\sigma ^{2},d)$.\medskip

\begin{example}
\label{NormalDensities} (Gaussian densities) If $X \sim p $ where $p$ is the 
$N_d(0, \Sigma)$ density with $\Sigma$ positive definite, then $X $ (and $p)$
is strongly log-concave $SLC_2 ( 0 , \Sigma, d)$ and hence also log-concave.
In particular for $d=1$, if $X \sim p$ where $p $ is the $N_1 (0, \sigma^2)$
density, then $X$ (and $p$) is $SLC_1 ( \sigma^2 ,1) = SLC_2 (0, \sigma^2 ,
1)$ and hence is also log-concave. Note that $\varphi_X^{\prime\prime} (x)
\equiv (-\log p)^{\prime \prime} (x) = 1/\sigma^2$ is constant in this
latter case.
\end{example}

\begin{example}
\label{LogisticDensity} (Logistic density) If $X \sim p$ where $p(x) =
e^{-x} /(1+e^{-x} )^2 = (1/4)/(\cosh(x/2))^2$, then $X $ (and $p$) is
log-concave and even \textsl{strictly log-concave} since $\varphi_X^{\prime
\prime} (x) = (-\log p)^{\prime \prime} (x) = 2p(x) > 0 $ for all $x \in 
\mathbb{R}$, but $X$ is \textsl{not} strongly log-concave.
\end{example}

\begin{example}
\label{BridgeDensity} (Bridge densities) If $X \sim p_\theta$ where, for $%
\theta \in (0,1)$, 
\begin{equation*}
p_{\theta} (x) = \frac{\sin (\pi\theta)}{2 \pi ( \cosh(\theta x) + \cos (\pi
\theta))},
\end{equation*}
then $X$ (and $p_{\theta}$) is log-concave for $\theta \in (0,1/2]$, but
fails to be log-concave for $\theta \in (1/2,1)$. For $\theta \in (1/2,1)$, $%
\varphi_{\theta}^{\prime \prime} (x) = (-\log p_{\theta})^{\prime \prime}
(x) $ is bounded below, by some negative value depending on $\theta$, and
hence these densities are \textsl{semi-log-concave} in the terminology of 
\cite{CattiauxGuillin:13} who introduce this further generalization of
log-concave densities by allowing the constant in the definition of a class
of strongly log-concave densities to be negative as well as positive. This particular 
family of densities on $\mathbb{R}$ was introduced in the context of binary
mixed effects models by \cite{MR2024756}.
\end{example}

\begin{example}
\label{SubbotinDensity} (Subbotin density) If $X \sim p_r$ where $p_r(x) =
C_r \exp( - |x|^r/r )$ for $x \in \mathbb{R}$ and $r>0$ where $C_r = 1/[2
\Gamma (1/r) r^{1/r-1}]$, then $X$ (and $p_r$) is log-concave for all $r\ge1$%
. Note that this family includes the Laplace (or double exponential) density
for $r=1$ and the Gaussian (or standard normal) density for $r=2$. The only
member of this family that is strongly log-concave is $p_2$, the standard
Gaussian density, since $(-\log p)^{\prime \prime} (x) = (r-1) |x|^{r-2}$
for $x\not= 0$.
\end{example}

\begin{example}
\label{SupremumBrownianBridge} (Supremum of Brownian bridge) If $\mathbb{U}$
is a standard Brownian bridge process on $[0,1]$, Then $P(\sup_{0\leq t\leq
1}\mathbb{U}(t)>x)=\exp (-2x^{2})$ for $x>0$, so the density is $f(x)=4x\exp
(-2x^{2})1_{(0,\infty )}(x)$, which is strongly log concave since $(-\log
f)^{\prime \prime }(x)=4+x^{-2}\geq 4$. This is a special case of the
Weibull densities $f_{\beta }(x)=\beta x^{\beta -1}\exp (-x^{\beta })$ which
are log-concave if $\beta \geq 1$ and strongly log-concave for $\beta \geq 2$%
. For more about suprema of Gaussian processes, see Section \ref%
{ssec_suprema}\ below.
\end{example}

For further interesting examples, see \cite{MR954608} 
%Dharmadhikari & Joag-Dev
and \cite{MR1375234} . %Prekopa (1995)

There exist a priori many ways to strengthen the property of log-concavity.
An very interesting notion is for instance the log-concavity of order $p$.
This is a one-dimensional notion, and even if it can be easily stated for
one-dimensional measures on $\mathbb{R}^{d}$, see \cite{MR2857249}  %BobkovMadiman:11a}\
Section 4, we state it in its classical way on $\mathbb{R}$.

\begin{definition}
\label{def_LC_order_alpha}A random variable $\xi >0$ is said to have a
log-concave distribution of order $p\geq 1$, if it has a density of the form 
$f(x)=x^{p-1}g(x),$ $x>0$, where the function $g$ is log-concave on 
$\left(0, \infty \right) $.
\end{definition}

Notice that the notion of log-concavity of order $1$ coincides with the
notion of log-concavity for positive random variables. Furthermore, it is
easily seen that log-concave variables of order $p>1$ are more concentrated
than log-concave variables. Indeed, with the notations of 
Definition~\ref{def_LC_order_alpha} 
and setting moreover $f=\exp \left( -\varphi_{f}\right) $ 
and $g=\exp \left( -\varphi _{g}\right) $, assuming that $f$
is $\mathcal{C}^{2}$ we get,
\begin{equation*}
\Hess \varphi _{f}=
\Hess \varphi _{g}+\frac{p-1}{x^{2}}\text{ .}
\end{equation*}%
As a matter of fact, the exponent $p$ strengthens the Hessian of the
potential of $g$, which is already a log-concave density. Here are some
example of log-concave variables of order $p$.

\begin{example}
The Gamma distribution with $\alpha \geq 1$ degrees of freedom, which has
the density $f(x)=\Gamma \left( \alpha \right) ^{-1}x^{\alpha -1}e^{-x} 1_{(0,\infty)} (x) $ 
%\mathbf{1}\left\{ x>0\right\} $ 
is log-concave of order $\alpha $.
\end{example}

\begin{example}
The Beta distribution $B_{\alpha ,\beta }$ with parameters $\alpha \geq 1$
and $\beta \geq 1$ is log-concave of order $\alpha $. We recall that its
density $g$ is given by 
$g(x)=B\left( \alpha ,\beta \right) ^{-1}x^{\alpha-1}\left( 1-x\right) ^{\beta -1} 1_{(0,1)}(x)$.   
%\mathbf{1}\left\{ 0<x<1\right\} $.
\end{example}

\begin{example}
The Weibull density of parameter $\beta \geq 1$, given by $h_{\beta }\left(
x\right) =\beta x^{\beta -1}\exp \left( -x^{\beta }\right) 1_{(0, \infty)} (x) $
%\mathbf{1}\left\{x>0\right\} $ 
is log-concave of order $\beta $.
\end{example}

It is worth noticing that when $X$ is a log-concave vector in 
$\mathbb{R}^{d} $ with spherically invariant distribution, then the Euclidian norm of $X$, 
denoted $\left\Vert X\right\Vert $, follows a log-concave distribution
of order $d-1$ (this is easily seen by transforming to polar coordinates; see 
\cite{MR2083386} for instance). The notion of log-concavity of
order $p$ is also of interest when dealing with problems in greater
dimension. Indeed, a general way to reduce a problem defined by $d$
-dimensional integrals to a problem involving one-dimensional integrals is given by the
``localization lemma'' of \cite{MR1238906}; 
see also \cite{MR1608200}. We will not further review this notion and we refer to 
\cite{MR2083386}, \cite{MR2839024}\ and \cite{MR2857249}
for nice results related in particular to concentration of log-concave
variables of order $p$.

The following sets of equivalences for log-concavity and strong
log-concavity will be useful and important. To state these equivalences we
need the following definitions from \cite{MR2814377}, page 199. %Simon
First, a subset $A $ of $\RR^d$ is \emph{balanced} (\cite{MR2814377}) 
or \emph{centrally symmetric} (\cite{MR954608})  if $x \in A$ implies $-x \in A$.  

\begin{definition}
\label{ConvexlyLayered-EvenRadialMonotone} A nonnegative function $f$ on $%
\mathbb{R}^d$ is \emph{convexly layered} if $\{ x : \ f(x) > \alpha \}$ is a
balanced convex set for all $\alpha>0$. It is called \emph{even, radial
monotone} if (i) $f(-x) = f(x)$ and (ii) $f(rx) \ge f(x)$ for all $0 \le r
\le 1$ and all $x \in \mathbb{R}^d$.
\end{definition}

\begin{proposition}
\label{EquivalencesLogCon} (Equivalences for log-concavity). Let $p =
e^{-\varphi}$ be a density function with respect to Lebesgue measure $%
\lambda $ on $\mathbb{R}^d$; that is, $p \ge 0$ and $\int_{\mathbb{R}^d} p d
\lambda =1$. Suppose that $\varphi \in C^2 $. Then the following are
equivalent:\newline
(a) $\varphi = - \log p$ is convex; i.e. $p$ is log-concave.\newline
(b) $\nabla \varphi = - \nabla p / p : \mathbb{R}^d \rightarrow \mathbb{R}^d$
is monotone:\newline
\begin{equation*}
\langle \nabla \varphi (x_2) - \nabla \varphi (x_1) , x_2 - x_1 \rangle \ge
0 \ \ \ \mbox{for all} \ \ x_1, x_2 \in \mathbb{R}^d .
\end{equation*}
(c) $\nabla^2 \varphi = {\nabla^2 \mathstrut} (\varphi ) \ge 0$.\newline
(d) $J_a (x; p) = p(a+x) p(a-x) $ is convexly layered for each $a \in 
\mathbb{R}^d$.\newline
(e) $J_a (x; p)$ is even and radially monotone.\newline
(f) $p$ is mid-point log-concave: for all $x_1,x_2 \in \mathbb{R}^d$, 
\begin{equation*}
p\left ( \frac{1}{2} x_1 + \frac{1}{2} x_2\right ) \ge p(x_1)^{1/2}
p(x_2)^{1/2} .
\end{equation*}
\end{proposition}

The equivalence of (a), (d), (e), and (f) is proved by \cite{MR2814377},
page 199, without assuming that $p\in C^{2}$. The equivalence of (a), (b),
and (c) under the assumption $\varphi \in C^{2}$ is classical and
well-known. This set of equivalences generalizes naturally to handle $%
\varphi \notin C^{2}$, but $\varphi $ proper and upper semicontinuous so
that $p$ is lower semicontinuous; see Section \ref{section_ap} below for the
adequate tools of convex regularization.

In dimension 1, \cite{Bobkov:96a} proved the following further 
characterizations of log-concavity on $\RR$.

\begin{proposition}[\protect\cite{Bobkov:96a}]
\label{prop_bobkov}Let $\mu $ be a nonatomic probability measure with
distribution function $F=\mu \left( \left( -\infty ,x\right] \right) $, 
$x\in \mathbb{R}$. 
Set $a=\inf \left\{ x\in \mathbb{R}: \ F\left( x\right)>0\right\} $ and 
$b=\sup \left\{ x\in \mathbb{R}: \ F\left( x\right) <1\right\} $. 
Assume that $F$ strictly increases on $\left( a,b\right) $, 
and let $F^{-1}:(0,1)\rightarrow \left( a,b\right) $ denote the inverse of $F$
restricted to $\left( a,b\right) $. Then the following properties are
equivalent:\\
%\begin{description}
%\item[(a)] 
(a) \ $\mu $ is log-concave;\\
%\item[(b)] 
(b) for all $h>0$, the function 
$R_{h}\left( p\right) =F\left( F^{-1}\left( p\right) +h\right) $ is concave on $\left( a,b\right) $;\\
%\item[(c)] 
(c ) $\mu $ has a continuous, positive density $f$ on $\left(a,b\right) $ and, moreover, the function 
$I\left( p\right) =f\left( F^{-1}\left( p\right) \right) $ is concave on $(0,1)$.
%\end{description}
\end{proposition}

Properties (b) and (c ) 
%\textbf{(b)} and \textbf{(c)} 
of Proposition \ref{prop_bobkov}
were first used in \cite{Bobkov:96a} along the proofs of his description of
the extremal properties of half-planes for the isoperimetric problem for
log-concave product measures on $\mathbb{R}^{d}$. 
In \cite{MR2857249}
the concavity of the function 
$I\left( p\right) =f\left( F^{-1}\left(p\right) \right) $ defined in point (c ) %\textbf{(c)} 
of Proposition \ref{prop_bobkov}, 
plays a role in the proof of concentration and moment
inequalities for the following information quantity: $-\log f\left( X\right)$ 
where $X$ is a random vector with log-concave density $f$.  Recently, 
\cite{BobkovLedoux:14} used the concavity of $I$ to prove upper and lower
bounds on the variance of the order statistics associated to an i.i.d.
sample drawn from a log-concave measure on $\mathbb{R}$. The latter results
allow then the authors to prove refined bounds on some Kantorovich transport
distances between the empirical measure associated to the i.i.d. sample and
the log-concave measure on $\mathbb{R}$. For more facts about the function $I $ 
for general measures on $\mathbb{R}$ and in particular, its relationship
to isoperimetric profiles, see Appendix A.4-6 of \cite{BobkovLedoux:14}.

\begin{example}
If $\mu $ is the standard Gaussian measure on the real line, then $I$ is
symmetric around $1/2$ and there exist constants $0< c_{0} \le c_{1}< \infty$ such that 
\begin{equation*}
c_{0}t\sqrt{\log \left( 1/t\right) }\leq I\left( t\right) \leq c_{1}t\sqrt{%
\log \left( 1/t\right) }\text{ ,}
\end{equation*}%
for $t\in \left( 0,1/2\right] $ (see \cite{BobkovLedoux:14}\ p.73).
\end{example}

We turn now to similar characterizations of strong log-concavity.

\begin{proposition}
\label{EquivalencesStrongLogConTry3} (Equivalences for strong log-concavity, 
$SLC_{1}$). Let $p=e^{-\varphi }$ be a density function with respect to
Lebesgue measure $\lambda $ on $\mathbb{R}^{d}$; that is, $p\geq 0$ and $%
\int_{\mathbb{R}^{d}}pd\lambda =1$. Suppose that $\varphi \in C^{2}$. Then
the following are equivalent:\newline
(a) $p$ is strongly log-concave; $p\in SLC_{1}(\sigma ^{2},d)$.\newline
(b) $\rho (x)\equiv \nabla \varphi (x)-x/\sigma ^{2}:\mathbb{R}^{d}\rightarrow \mathbb{R}^{d}$ 
is monotone:\newline
\begin{equation*}
\langle \rho (x_{2})-\rho (x_{1}),x_{2}-x_{1}\rangle \geq 0\ \ \ 
\mbox{for
all}\ \ x_{1},x_{2}\in \mathbb{R}^{d}.
\end{equation*}%
(c) $\nabla \rho (x)=\nabla ^{2}\varphi -I/\sigma ^{2}\geq 0$.\newline
(d) For each $a\in \mathbb{R}^{d}$ the function 
\begin{equation*}
J_{a}^{\phi }(x;p)\equiv \frac{p(a+x)p(a-x)}{\phi _{\sigma ^{2}I/2}(x)}
\end{equation*}%
is convexly layered. \newline
(e) The function $J_{a}^{\phi }(x;p)$ in (d) is even and radially monotone
for all $a\in \mathbb{R}^{d}$. \newline
(f) For all $x,y\in \mathbb{R}^{d}$, 
\begin{equation*}
p\left( \frac{1}{2}x+\frac{1}{2}y\right) \geq p(x)^{1/2}p(y)^{1/2}\exp
\left( \frac{1}{8}|x-y|^{2}\right) .
\end{equation*}
\end{proposition}

%\begin{proposition} 
%\label{EquivalencesStrongLogConTry2}
%(Equivalences for strong log-concavity, second try).
%Let $p = e^{-\varphi}$ be a density function with respect to Lebesgue measure 
%$\lambda$ on $\RR^d$; that is, $p \ge 0$ and $\int_{\RR^d} p d \lambda =1$.
%Suppose that $\varphi \in C^2 $.  Then the following are equivalent:\\
%(a)   $p$ is strongly log-concave; $p \in SLC_1 (\sigma^2, d)$.\\
%(b)  $\rho (x) \equiv \nabla \varphi (x) - x/\sigma^2  : \RR^d \rightarrow \RR^d$ is monotone:\\
%$$
%\langle  \rho (x_2) - \rho (x_1) , x_2 - x_1 \rangle \ge 0  \ \ \ \mbox{for all} \ \ x_1, x_2 \in \RR^d .
%$$
%(c) $\nabla \rho (x) = \nabla^2 \varphi - I/\sigma^2  \ge 0$.\\
%(d) The function 
%\begin{eqnarray*}
%J_a (x; p) 
%& = & p(a+x) p(a-x) \\
%& = & g(a+x)g(a-x) \prod_{i=1}^d \frac{1}{\sigma} \phi \left (\frac{a_i + x_i }{\sigma} \right ) 
%              \prod_{i=1}^d \frac{1}{\sigma} \phi \left (\frac{a_i - x_i }{\sigma} \right ) \\
%& = & J_a (x; g) C(a, \sigma ) \exp (- | x |^2 /\sigma^2)
%\end{eqnarray*}
%where $C(a,\sigma) = (2 \pi \sigma^2)^{-d} \exp( - | a |^2/\sigma^2)$ 
%and $J_a (x; g)$ is convexly layered for each $a \in \RR^d$.\\
%(e) The function $J_a (x; g)$ in (d) is even and radially  monotone 
%for all $a \in \RR^d$.\\ 
%\end{proposition}

\begin{proof}
See Section~\ref{sec:Proofs}.
\end{proof}

We investigate the extension of Proposition \ref{prop_bobkov} concerning
log-concavity on $\mathbb{R}$,\ to the case of strong log-concavity. 
(The following result is apparently new.)  Recall that a function $h$ is
strongly concave on $\left( a,b\right) $ with parameter $c>0$ %we also say 
(or $c $-strongly concave), if for any $x,y\in \left( a,b\right) $, any 
$\theta \in \left( 0,1\right) $, 
\begin{equation*}
h(\theta x+(1-\theta )y)\geq \theta h(x)+(1-\theta )h(y)+\frac{1}{2}c\theta
(1-\theta )\Vert x-y\Vert ^{2}\text{ .}
\end{equation*}

\begin{proposition}
\label{prop_I_SLC}
Let $\mu $ be a nonatomic probability measure with
distribution function $F=\mu \left( \left( -\infty ,x\right] \right) $, 
$x\in \mathbb{R}$. Set 
$a=\inf \left\{ x\in \mathbb{R}: F\left( x\right) >0\right\} $ and 
$b=\sup \left\{ x\in \mathbb{R}: F\left( x\right) <1\right\} $, possibly infinite. 
Assume that $F$ strictly increases on 
$\left( a,b\right)$, and let $F^{-1}:(0,1)\rightarrow \left( a,b\right) $ denote
the inverse of $F$ restricted to $\left( a,b\right) $. 
Suppose that $X$ is a random variable with distribution $\mu $. 
Then the following properties hold:   

\begin{description}
\item[(i)] If $X\in SLC_{1}\left( c,1\right) $, $c>0$, then 
$I\left( p\right) =f\left( F^{-1}\left( p\right) \right) $ is 
$\left( c\left\Vert f\right\Vert _{\infty }\right) ^{-1}$-strongly concave and 
$\left( c^{-1} \sqrt{\var \left( X\right) }\right) $-strongly concave on $(0,1)$.

\item[(ii)] The converse of point \textbf{(i)} is false: there exists a
log-concave variable $X$ which is not strongly concave (for any parameter 
$c>0$) such that the associated $I$ function is strongly log-concave on 
$\left( 0,1\right) $.

\item[(iii)] There exist a strongly log-concave random variable 
$X\in SLC\left( c,1\right) $ and $h_{0}>0$ such that the function 
$R_{h_{0}}\left( p\right) =F\left( F^{-1}\left( p\right) +h_{0}\right) $ is concave but not
strongly concave on $\left( a,b\right) $.

\item[(iv)] There exists a log-concave random variable $X$ which is not
strongly log-concave (for any positive parameter), such that for all $h>0$,
the function 
$R_{h_{0}}\left( p\right) =F\left( F^{-1}\left( p\right) +h\right) $ is strongly concave on $\left( a,b\right) $.
\end{description}
\end{proposition}

From  \textbf{(i) }and \textbf{(ii)} in Proposition \ref{prop_I_SLC},
we see that the strong concavity of the function $I$ is a necessary but not
sufficient condition for the strong log-concavity of $X$. Points \textbf{%
(iii)} and \textbf{(iv)} state that no relations exist in general between
the strong log-concavity of $X$ and strong concavity of its associated
function $R_{h}$.

\begin{proof}  
See Section~\ref{sec:Proofs}.
\end{proof}

The following proposition gives a similar set of equivalences for our second
definition of strong log-concavity, Definition~\ref{strongLogConcaveDefn2}.
\smallskip

\begin{proposition}
\label{EquivalencesStrongLogConDef2Try2} (Equivalences for strong
log-concavity, $SLC_2$). Let $p = e^{-\varphi}$ be a density function with
respect to Lebesgue measure $\lambda$ on $\mathbb{R}^d$; that is, $p \ge 0$
and $\int_{\mathbb{R}^d} p d \lambda =1$. Suppose that $\varphi \in C^2 $.
Then the following are equivalent:\newline
(a) $p$ is strongly log-concave; $p \in SLC_2 (\mu, \Sigma, d)$ with $\Sigma
> 0$, $\mu \in \mathbb{R}^d$.\newline
(b) $\rho (x) \equiv \nabla \varphi (x) - \Sigma^{-1}(x - \mu) : \mathbb{R}%
^d \rightarrow \mathbb{R}^d$ is monotone:\newline
\begin{equation*}
\langle \rho (x_2) - \rho (x_1) , x_2 - x_1 \rangle \ge 0 \ \ \ 
\mbox{for
all} \ \ x_1, x_2 \in \mathbb{R}^d .
\end{equation*}
(c) $\nabla \rho (x) = \nabla^2 \varphi - \Sigma^{-1} \ge 0$.\newline
(d) For each $a \in \mathbb{R}^d$, the function 
% $J_a (x; p) = p(a+x) p(a-x) $ satisfies
\begin{eqnarray*}
J_a^{\phi} (x; p) = p(a+x)p(a-x)/ \phi_{\Sigma/2} (x)
\end{eqnarray*}
is convexly layered. (e) For each $a \in \mathbb{R}^d$ the function $%
J_a^{\phi} (x; p)$ in (d) is even and radially monotone.\newline
(f) For all $x , y \in \mathbb{R}^d$, 
\begin{eqnarray*}
p\left ( \frac{1}{2} x + \frac{1}{2} y\right ) \ge p(x)^{1/2} p(y)^{1/2}
\exp \left ( \frac{1}{8} (x-y)^T \Sigma^{-1} (x-y) \right ) .
\end{eqnarray*}
\end{proposition}

\begin{proof}
To prove Proposition~\ref{EquivalencesStrongLogConDef2Try2} it
suffices to note the log-concavity of $g(x)=p(x)/\phi _{\Sigma /2}(x)$ and
to apply Proposition~\ref{EquivalencesLogCon} (which holds as well for
log-concave functions). The claims then follow by straightforward calculations; 
see Section~\ref{sec:Proofs} for more details.
\end{proof}

%----------------------------------------------------------
\section{Log-concavity and strong log-concavity: preservation theorems}

%----------------------------------------------------------

Both log-concavity and strong log-concavity are preserved by a number of
operations. Our purpose in this section is to review these preservation
results and the methods used to prove such results, with primary emphasis
on: (a) affine transformations, (b) marginalization, (c) convolution. The
main tools used in the proofs will be: (i) the Brunn-Minkowski inequality;
(ii) the Brascamp-Lieb Poincar\'{e} type inequality; (iii) Pr\'{e}kopa's
theorem; (iv) Efron's monotonicity theorem.

\subsection{Preservation of log-concavity}

\label{subsec:PreservationLogConcave}

%\par\noindent
%{\bf Affine transformations:}    Suppose that $P$ is a log-concave measure on $(\RR^d, {\cal B}^d)$

\subsubsection{Preservation by affine transformations}

Suppose that $X$ has a log-concave distribution $P$ on $(\mathbb{R}^{d},%
\mathcal{B}^{d})$, and let $A$ be a non-zero real matrix of order $m\times d$%
. Then consider the distribution $Q$ of $Y=AX$ on $\mathbb{R}^{m}$.

\begin{proposition}
\label{prop:AffinePreservLogCon} (log-concavity is preserved by affine
transformations). The probability measure $Q $ on $\mathbb{R}^m$ defined by $%
Q(B) = P( AX \in B )$ for $B \in \mathcal{B}^m$ is a log-concave probability
measure. If $P$ is non-degenerate log-concave on $\mathbb{R}^d$ with density 
$p$ and $m=d$ with $A $ of rank $d$, then $Q$ is non-degenerate with
log-concave density $q$.
\end{proposition}

\begin{proof}
See \cite{MR954608}, Lemma 2.1, page 47.
\end{proof}

\subsubsection{Preservation by products}

Now let $P_1 $ and $P_2$ be log-concave probability measures on $(\mathbb{R}%
^{d_1} , \mathcal{B}^{d_1})$ and $(\mathbb{R}^{d_2} , \mathcal{B}^{d_2})$
respectively. Then we have the following preservation result for the product
measure $P_1 \times P_2$ on $(\mathbb{R}^{d_1} \times \mathbb{R}^{d_2} , 
\mathcal{B}^{d_1} \times \mathcal{B}^{d_2} )$:

\begin{proposition}
\label{prop:ProdMeasPreservLogConc} (log-concavity is preserved by products)
If $P_1$ and $P_2$ are log-concave probability measures then the product
measure $P_1 \times P_2$ is a log-concave probability measure.
\end{proposition}

\begin{proof}
See \cite{MR954608}, Theorem 2.7, page 50. A key fact used in this proof is
that if a probability measure $P$ on $(\mathbb{R}^d, \mathcal{B}^d)$ assigns
zero mass to every hyperplane in $\mathbb{R}^d$, then log-concavity of $P$
holds if and only if $P( \theta A + (1-\theta) B) \ge P(A)^{\theta}
P(B)^{1-\theta}$ for all rectangles $A,B$ with sides parallel to the
coordinate axes; see \cite{MR954608}, Theorem 2.6, page 49.
\end{proof}

\subsubsection{Preservation by marginalization}

Now suppose that $p$ is a log-concave density on $\mathbb{R}^{m+n}$ and
consider the marginal density $q(y) = \int_{\mathbb{R}^m} p(x,y) dx $. The
following result due to \cite{MR0404557} concerning preservation of
log-concavity was given a simple proof by \cite{MR0450480} (Corollary 3.5,
page 374). In fact they also proved the whole family of such results for $s-$%
concave densities.

\begin{theorem}
\label{LCpreservedByMargin} (log-concavity is preserved by marginalization;
Pr\'ekopa's theorem). Suppose that $p$ is log-concave on $\mathbb{R}^{m+n}$
and let $q(y) = \int_{\mathbb{R}^m} p(x,y) dx$. Then $q$ is log-concave.
\end{theorem}

This theorem is a center piece of the entire theory. It was proved
independently by a number of mathematicians at about the same time: these
include \cite{MR0404557}, building on \cite{MR0096173}, % Dinghas 1957 
\cite{MR0315079}, %Prekopa 1971 & 1973
\cite{Brascamp-Lieb:74}, %Brascamp and Lieb '74
\cite{Brascamp-Lieb:75}, %Brascamp and Lieb '75
\cite{MR0450480}, %Brascamp and Lieb '76,
\cite{MR0404559}, \cite{MR0388475}, % Borell,
and \cite{MR0428540}. % Rinott.
\cite{MR2814377}, page 310, %Simon 
gives a brief discussion of the history, including an unpublished proof of
Theorem~\ref{LCpreservedByMargin} given in \cite{Brascamp-Lieb:74}. Many of
the proofs (including the proofs in \cite{Brascamp-Lieb:75}, \cite{MR0404559}%
, and \cite{MR0428540}) are based fundamentally on the Brunn-Minkowski
inequality; see \cite{MR588074}, \cite{MR1898210}, and \cite{MR2167203} for
useful surveys.

We give two proofs here. The first proof is a \textsl{transportation}
argument from \cite{MR1991646}; the second is a proof from \cite%
{Brascamp-Lieb:74} which has recently appeared in \cite{MR2814377}.

\begin{proof}
(Via \textsl{transportation}). We can reduce to the case $n=1$ since it
suffices to show that the restriction of $q$ to a line is log-concave. Next
note that an inductive argument shows that the claimed log-concavity holds
for $m+1$ if it holds for $m$, and hence it suffices to prove the claim for $%
m=n=1$.

Since log-concavity is equivalent to mid-point log concavity (by the
equivalence of (a) and (e) in Proposition~\ref{EquivalencesLogCon}), we only
need to show that 
\begin{eqnarray}
q \left (\frac{u+v}{2} \right ) \ge q(u)^{1/2} q(v)^{1/2}
\label{MidPointLogConcavityofQ}
\end{eqnarray}
for all $u,v \in \mathbb{R}$. Now define 
\begin{eqnarray*}
f(x) = p(x,u), \ \ \ g(x) = p(x,v), \ \ \ h(x) = p(x, (u+v)/2) .
\end{eqnarray*}
Then (\ref{MidPointLogConcavityofQ}) can be rewritten as 
\begin{eqnarray*}
\int h(x) dx \ge \left ( \int f(x) dx) \right )^{1/2} \left ( \int g(x) dx
\right )^{1/2} .
\end{eqnarray*}
From log-concavity of $p$ we know that 
\begin{eqnarray}
h\left ( \frac{z+w}{2} \right ) = p \left ( \frac{z+w}{2} , \frac{u+v}{2}
\right ) \ge p(z,u)^{1/2} p(w,v)^{1/2} = f(z)^{1/2} g(w)^{1/2} .
\label{ConsequenceLogConcavityOfP}
\end{eqnarray}
By homogeneity we can arrange $f, g$, and $h$ so that $\int f(x) dx = \int
g(x) dx = 1$; if not, replace $f$ and $g$ with $\tilde{f}$ and $\tilde{g}$
defined by $\tilde{f} (x) = f(x) / \int f(x^{\prime})dx^{\prime}= f(x) /
q(u) $ and $\tilde{g} (x) = g(x) / \int g(x^{\prime})dx^{\prime}= g(x) /
q(v) $.

Now for the transportation part of the argument: let $Z$ be a real-valued
random variable with distribution function $K$ having smooth density $k$.
Then define maps $S$ and $T$ by $K(z) = F(S(z))$ and $K(z) = G(T(z))$ where $%
F$ and $G$ are the distribution functions corresponding to $f$ and $g$. Then 
\begin{eqnarray*}
k(z) = f(S(z))S^{\prime}(z) = g(T(z))T^{\prime}(z)
\end{eqnarray*}
where $S^{\prime}, T^{\prime}\ge0$ since the same is true for $k$, $f$, and $%
g$, and it follows that 
\begin{eqnarray*}
1 = \int k(z) dz & = & \int f(S(z))^{1/2} g(T(z))^{1/2}
(S^{\prime}(z))^{1/2} (T^{\prime}(z))^{1/2} dz \\
& \le & \int h\left ( \frac{S(z) + T(z)}{2} \right ) (S^{\prime}(z))^{1/2}
(T^{\prime}(z))^{1/2} dz \\
& \le & \int h\left ( \frac{S(z) + T(z)}{2} \right ) \cdot \frac{%
S^{\prime}(z) + T^{\prime}(z)}{2} dz \\
& = & \int h(x) dx
\end{eqnarray*}
by the inequality (\ref{ConsequenceLogConcavityOfP}) in the first inequality
and by the arithmetic - geometric mean inequality in the second inequality.
\end{proof}

\begin{proof}
(Via \textsl{symmetrization}). By the same induction argument as in the
first proof we can suppose that $m=1$. By an approximation argument we may
assume, without loss of generality that $p$ has compact support and is
bounded.

Now let $a \in \mathbb{R}^n$ and note that 
\begin{eqnarray*}
J_a (y; q) & = & q(a+y) q(a-y) \\
& = & \int \! \! \int p(x, a+y) p(z, a-y) dx dz \\
& = & 2 \int \! \! \int p(u+v, a+y) p(u-v, a-y) du dv \\
& = & 2 \int \! \! \int J_{u,a} (v,y; p ) du dv
\end{eqnarray*}
where, for $(u,a)$ fixed, the integrand is convexly layered by Proposition~%
\ref{EquivalencesLogCon} (d). Thus by the following Lemma~\ref%
{convexlyLayeredPreservedByMarginalization}, the integral over $v$ is an
even lower semi-continuous function of $y$ for each fixed $u,a$. Since this
class of functions is closed under integration over an indexing parameter
(such as $u$), the integration over $u$ also yields an even radially
monotone function, and by Fatou's lemma $J_a (y; g) $ is also lower
semicontinuous. It then follows from Proposition~\ref{EquivalencesLogCon}
again that $g$ is log-concave.
\end{proof}

\begin{lemma}
\label{convexlyLayeredPreservedByMarginalization} Let $f$ be a lower
semicontinuous convexly layered function on $\mathbb{R}^{n+1}$ written as $%
f(x,t)$, $x \in \mathbb{R}^n$, $t\in \mathbb{R}$. Suppose that $f$ is
bounded and has compact support. Let 
\begin{eqnarray*}
g(x) = \int_{\mathbb{R}} f(x,t) dt .
\end{eqnarray*}
Then $g$ is an even, radially monotone, lower semicontinuous function.
\end{lemma}

\begin{proof}
First note that sums and integrals of even radially monotone functions are
again even and radially monotone. By the wedding cake representation 
\begin{eqnarray*}
f(x) = \int_0^\infty 1\{ f(x) > t\} dt,
\end{eqnarray*}
it suffices to prove the result when $f$ is the indicator function of an
open balanced convex set $K$. Thus we define 
\begin{eqnarray*}
K(x) = \{ t \in \mathbb{R} : \ (x,t) \in K \}, \ \ \mbox{for} \ \ x \in 
\mathbb{R}^n .
\end{eqnarray*}
Thus $K(x) = (c(x), d(x))$, an open interval in $\mathbb{R}$ and we see that 
\begin{eqnarray*}
g(x) = d(x) - c(x).
\end{eqnarray*}
But convexity of $K$ implies that $c(x)$ is convex and $d(x)$ is concave,and
hence $g(x)$ is concave. Since $K$ is balanced, it follows that $c(-x) = -
d(x)$, or $d(-x) = - c(x)$, so $g$ is even. Since an even concave function
is even radially monotone, and lower semicontinuity of $g$ holds by Fatou's
lemma, the conclusion follows.
\end{proof}

%\subsection{Preservation by Weak Limits}
%\subsection{Preservation by Affine Transformations}

\subsubsection{Preservation under convolution}

Suppose that $X, Y$ are independent with log-concave distributions $P$ and $%
Q $ on $(\mathbb{R}^d, \mathcal{B}^d)$, and let $R$ denote the distribution
of $X+Y$. The following result asserts that $R$ is log-concave as a measure
on $\mathbb{R}^d$.

\begin{proposition}
\label{LCpreservedByConv} (log-concavity is preserved by convolution). Let $%
P $ and $Q$ be two log-concave distributions on $(\mathbb{R}^d , \mathcal{B}%
^d) $ and let $R$ be the convolution defined by $R(B) = \int_{\mathbb{R}^d}
P( B- y) dQ(y)$ for $B \in \mathcal{B}^d$. Then $R$ is log-concave.
\end{proposition}

\begin{proof}
It suffices to prove the proposition when $P$ and $Q$ are absolutely
continuous with densities $p$ and $q$ on $\mathbb{R}^d$. Now $h(x,y) =
p(x-y)q(y)$ is log-concave on $\mathbb{R}^{2d}$, and hence by Proposition~%
\ref{LCpreservedByMargin} it follows that 
\begin{eqnarray*}
r(y) = \int_{\mathbb{R}^d} h(x,y) dy = \int_{\mathbb{R}^d} p(x-y) q(y) dy
\end{eqnarray*}
is log-concave.
\end{proof}

Proposition~\ref{LCpreservedByConv} was proved when $d=1$ by \cite{MR0047732}
who used the $PF_2$ terminology of P\'olya frequency functions. In fact all
the P\'olya frequency classes $PF_k$, $k\ge 2$, are closed under convolution
as shown by \cite{MR0230102}; see \cite{MR2759813}, Lemma A.4 (page 758) and
Proposition B.1, page 763. The first proof of Proposition~\ref%
{LCpreservedByConv} when $d\ge 2$ is apparently due to \cite{MR0241584}.
While the proof given above using Pr\'ekopa's theorem is simple and quite
basic, there are at least two other proofs according as to whether we use:%
\newline
(a) the equivalence between log-concavity and monotonicity of the scores of $%
f$, or \newline
(b) the equivalence between log-concavity and non-negativity of the matrix
of second derivatives (or Hessian) of $-\log f$, assuming that the second
derivatives exist.

The proof in (a) relies on Efron's inequality when $d=1$, and was %first
noted by \cite{Wellner-2013} in parallel to the corresponding 
proof of ultra log-concavity in the discrete case given by \cite{MR2327839}; 
see Theorem~\ref{thm:UltraLogConvPreservByConv}.
We will return to this in Section~\ref{sec:EfronsTheoremOneDimension}. For $d>1$ this approach breaks down because
Efron's theorem does not extend to the multivariate setting without further
hypotheses. %relies on a (potential? or conjectured?) 
Possible generalizations of Efron's theorem will be discussed in Section~\ref%
{sec:StrongLogConcavityPreservMultivariateCase}. The proof in (b) relies on
a Poincar\'{e} type inequality of \cite{MR0450480}. %Brascamp and Lieb 
These three different methods are of some interest since they all have
analogues in the case of proving that strong log-concavity is preserved
under convolution.

It is also worth noting the following difference between the situation in
one dimension and the result for preservation of convolution in higher
dimensions: as we note following Theorems 29 and 33, \cite{Ibragimov:56} 
% Ibragimov 
and \cite{Keilson-Gerber:71} %Keilson proved 
showed that in the one-dimensional continuous and discrete settings
respectively that if $p\star q$ is unimodal for every unimodal $q$, then $p$
is log-concave. The analogue of this for $d>1$ is more complicated in part
because of the great variety of possible definitions of \textquotedblleft
unimodal\textquotedblright\ in this case; see \cite{MR954608} 
% Dhamadikari and Joag-Dev 
chapters 2 and 3 for a thorough discussion. In particular \cite{MR0074845} 
% Sherman 
provided the following counterexample when the notion of unimodality is
taken to be \textsl{centrally symmetric convex} unimodality; that is, the
sets $S_{c}(p)\equiv \{x\in \mathbb{R}^{d}:p(x)\geq c\}$ are symmetric and
convex for each $c\geq 0$. Let $p$ be the uniform density on $[-1,1]^{2}$
(so that $p(x)=(1/4)1_{[-1,1]^{2}}(x)$); then $p$ is log-concave. Let $q$ be
the density given by $1/12$ on $[-1,1]^{2}$ and $1/24$ on $([-1,1]\times
(1,5])\cup ([-1,1]\times \lbrack -5,-1))$. Thus $q$ is centrally symmetric
convex (and hence also \textsl{quasi-concave}, $q\in \mathcal{P}_{-\infty }$
as in Definition~\ref{defn:Sconcave}. But $h=p\star q$ is not centrally
symmetric convex (and also is not quasi-concave), since the sets $S_{c}(h)$
are not convex: see Figure~\ref{fig:ShermansExample}.

\begin{figure}[htb!]
\centering
\includegraphics[width=0.50\textwidth]{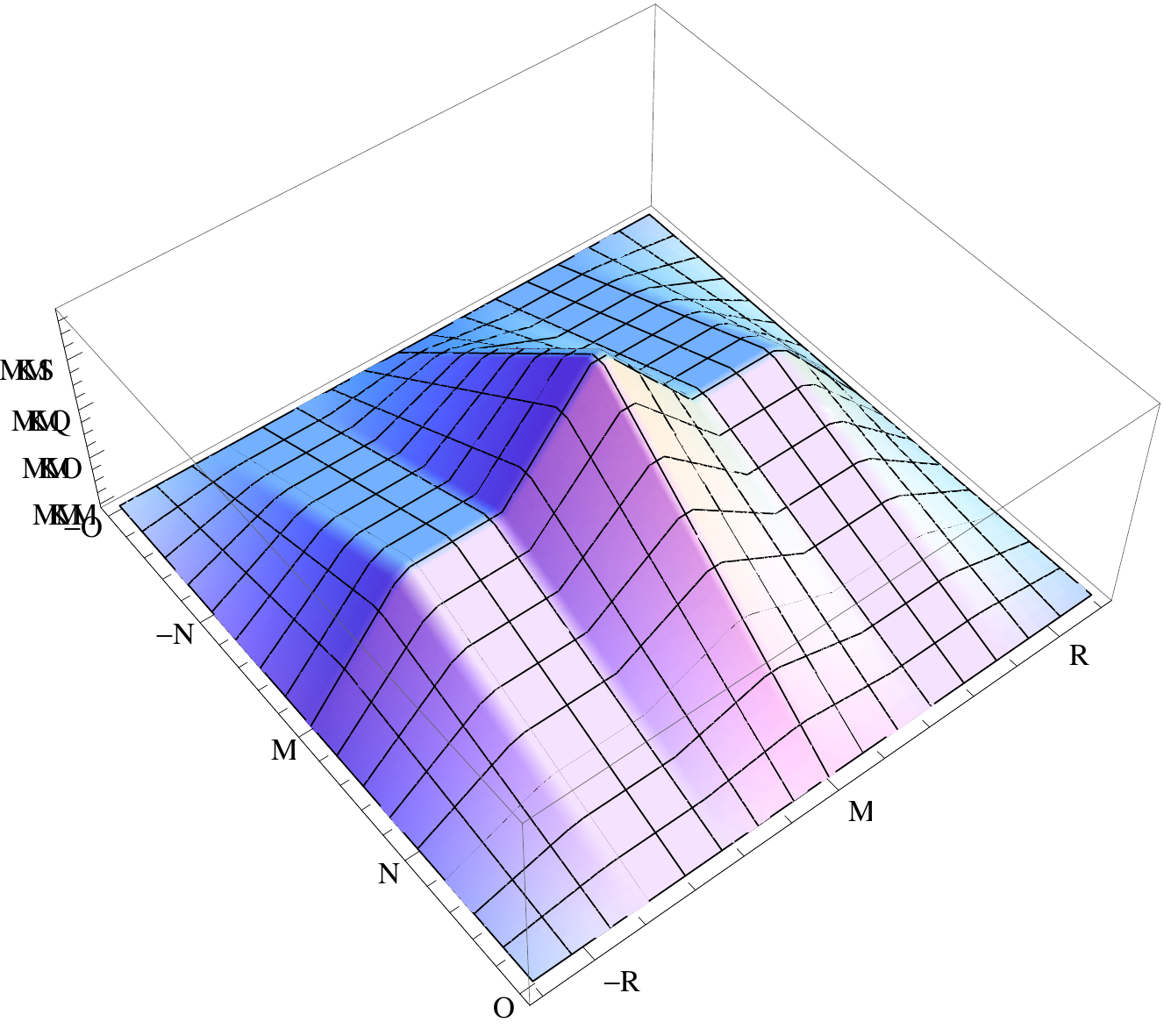} %AiryHadamardProdCheck3.pdf}
\caption{Sherman's example, $h=p\star q$}
\label{fig:ShermansExample}
% plot produced by ``Recent.11.21.11/ConvUnimodal-LogConc-try3.2d''
\end{figure}

\subsubsection{Preservation by (weak) limits}

Now we consider preservation of log-concavity under convergence in
distribution.

\begin{proposition}
\label{prop:logconcPreservConvD} (log-concavity is preserved under
convergence in distribution). Suppose that $\{ P_n \}$ is a sequence of
log-concave probability measures on $\mathbb{R}^d$, and suppose that $P_n
\rightarrow_d P_0$. Then $P_0$ is a log-concave probability measure.
\end{proposition}

\begin{proof}
See \cite{MR954608}, Theorem 2.10, page 53.
\end{proof}

Note that the limit measure in Proposition~\ref{prop:logconcPreservConvD}
might be concentrated on a proper subspace of $\mathbb{R}^d$. If we have a
sequence of log-concave densities $p_n$ which converge pointwise to a
density function $p_0$, then by Scheff\'e's theorem we have $p_n \rightarrow
p_0 $ in $L_1 (\lambda)$ and hence $d_{TV} (P_n , P_0) \rightarrow 0$. Since
convergence in total variation implies convergence in distribution we
conclude that $P_0$ is a log-concave measure where the affine hull of $%
\mbox{supp} (P_0)$ has dimension $d$ and hence $P_0$ is the measure
corresponding to $p_0$ which is necessarily log-concave by Theorem~\ref%
{PrekopaRinott}. 
%{\bf Q:}  Can the limit measure be concentrated on a proper subspace of $\RR^d$ (even 
%when the $P_n$'s all have $\mbox{supp}(P_n) = \RR^d$)? 

Recall that the class of normal distributions on $\mathbb{R}^d$ is closed
under all the operations discussed above: affine transformation, formation
of products, marginalization, convolution, and weak limits. Since the larger
class of log-concave distributions on $\mathbb{R}^d$ is also preserved under
these operations, the preservation results of this section suggest that the
class of log-concave distributions is a very natural nonparametric class
which can be viewed naturally as an enlargement of the class of all normal
distributions. This has stimulated much recent work on nonparametric
estimation for the class of log-concave distributions on $\mathbb{R}$ and $%
\mathbb{R}^d$: for example, see \cite{MR2546798}, \cite{MR2645484}, \cite%
{MR2758237}, \cite{MR2757433}, %??, ??, ?? , ?? 
\cite{MR2509075}, and \cite{hen+ast06}, and see Section~\ref%
{subsec:NonparametricStatistics} for further details.

\subsection{Preservation of strong log-concavity}

\label{subsec:PreservationStrongLogConcave}

Here is a theorem summarizing several preservation results for strong
log-concavity. Parts (a), (b), and (d) were obtained by \cite{hen+ast06}.

\begin{theorem}
\label{StrongLogConPreservOne} (Preservation of strong log-concavity)\newline
(a) (Linear transformations) Suppose that $X $ has density $p \in SLC_2
(0,\Sigma, d)$ and let $A$ be a $d\times d$ nonsingular matrix. Then $Y=AX$
has density $q \in SLC_2 (0, A \Sigma A^T, d)$ given by $q(y) = p(A^{-1} y) %
\mbox{det} (A^{-1} )$.\newline
(b) (Convolution) If $Z = X+Y$ where $X \sim p \in SLC_2 (0, \Sigma, d)$ and 
$Y \sim q \in SLC_2 (0, \Gamma , d)$ are independent, then $Z = X +Y \sim p
\star q \in SLC_2 (0 , \Sigma + \Gamma , d)$.\newline
(c) (Product distribution) If $X \sim p \in SLC_2 (0, \Sigma, m)$ and $Y
\sim q \in SLC_2 (0, \Gamma, n)$, then 
\begin{equation*}
(X,Y) \sim p\cdot q \in SLC_2 \left ( 0 , \left ( 
\begin{array}{cc}
\Sigma & 0 \\ 
0 & \Gamma%
\end{array}
\right ) , m+n \right ) .
\end{equation*}
(d) (Product function) If $p \in SLC_2 (0 , \Sigma, d)$ and $q \in SLC_2 (0,
\Gamma, d)$, then $h$ given by $h(x) = p(x) q(x)$ (which is typically not a
probability density function) satisfies $h \in SLC_2 (0, (\Sigma^{-1} +
\Gamma^{-1} )^{-1} )$.
\end{theorem}

Part (b) of Theorem~\ref{StrongLogConPreservOne} is closely related to the
following result which builds upon and strengthens Pr\'ekopa's Theorem~\ref%
{LCpreservedByMargin}. It is due to \cite{MR0450480} (Theorem 4.3, page
380); see also \cite{MR2814377}, Theorem 13.13, page 204.

\begin{theorem}
\label{StrongLogConPreservTwoTry2} (Preservation of strong log-concavity
under marginalization). Suppose that $p\in SLC_2 (0, \Sigma, m+n)$. Then the
marginal density $q$ on $\mathbb{R}^n$ given by 
\begin{eqnarray*}
q(x) = \int_{\mathbb{R}^m} p(x,y) dy
\end{eqnarray*}
is strongly log-concave: $q \in SLC_2 (0, \Sigma_{11}, m)$ where 
\begin{eqnarray}
\Sigma = \left ( 
\begin{array}{cc}
\Sigma_{11} & \Sigma_{12} \\ 
\Sigma_{21} & \Sigma_{22}%
\end{array}
\right ).  \label{PartitionOfCovMatrix}
\end{eqnarray}
%and 
%\begin{eqnarray}
%D \equiv C - B^T A^{-1} B .
%\label{DefnOfMatrixD}
%\end{eqnarray}
\end{theorem}

\begin{proof}
Since $p \in SLC_2 (0, \Sigma, m+n)$ we can write 
\begin{eqnarray*}
p(x,y) & = & g(x,y) \phi_{\Sigma} (x,y) \\
& = & g(x,y) \frac{1}{(2 \pi | \Sigma |)^{(m+n)/2}} \exp \left ( - \frac{1}{2%
} (x^T, y^T) \left ( 
\begin{array}{cc}
\Sigma_{11} & \Sigma_{12} \\ 
\Sigma_{21} & \Sigma_{22}%
\end{array}
\right )^{-1} \left ( 
\begin{array}{c}
x \\ 
y%
\end{array}
\right ) \right )
\end{eqnarray*}
where $g$ is log-concave. Now the Gaussian term in the last display can be
written as 
\begin{eqnarray*}
\lefteqn{\phi_{Y|X} (y|x) \cdot \phi_X (x) } \\
& = & \frac{1}{(2 \pi | \Sigma_{22\cdot 1} | )^{n/2}} \exp \left ( - \frac{1%
}{2} (y- \Sigma_{21} \Sigma_{11}^{-1} x)^T \Sigma_{22\cdot 1}^{-1} ( y -
\Sigma_{21} \Sigma_{11}^{-1}x ) \right ) \\
&& \ \ \ \cdot \frac{1}{(2 \pi | \Sigma_{11} | )^{m/2}} \exp \left ( - \frac{%
1}{2} x^T \Sigma_{11}^{-1} x \right )
\end{eqnarray*}
where $\Sigma_{22\cdot 1} \equiv \Sigma_{22} - \Sigma_{21} \Sigma_{11}^{-1}
\Sigma_{12}$, 
%is given by (\ref{DefnOfMatrixD}), and hence  %$D \equiv C - B^T A^{-1} B$, 
and hence 
\begin{eqnarray*}
q(x) 
& = & \int_{\mathbb{R}^n} g(x,y) \frac{1}{(2 \pi | \Sigma_{22\cdot 1} |)^{n/2}} 
               \exp \left ( - \frac{1}{2} (y- \Sigma_{21} \Sigma_{11}^{-1} x)^T
               \Sigma_{22\cdot 1}^{-1} ( y - \Sigma_{21} \Sigma_{11}^{-1}x ) \right ) dy \\
&& \ \ \ \cdot \frac{1}{(2 \pi | \Sigma_{11} | )^{m/2}} \exp \left ( - \frac{1}{2} x^T \Sigma_{11}^{-1} x \right ) \\
& = & \int_{\mathbb{R}^n} g(x, \tilde{y} + \Sigma_{21} \Sigma_{11}^{-1} x )
             \cdot \frac{1}{(2 \pi | \Sigma_{22\cdot 1} | )^{n/2}} 
             \exp \left ( - \frac{1}{2} \tilde{y}^T \Sigma_{22\cdot 1}^{-1} \tilde{y} \right ) d \tilde{y} \\
&& \ \ \ \cdot \frac{1}{(2 \pi | \Sigma_{11} | )^{m/2}} \exp ( - (1/2) x^T \Sigma_{11}^{-1} x ) \\
& \equiv & h(x) \phi_{\Sigma_{11}} (x)
\end{eqnarray*}
where 
\begin{eqnarray*}
h(x) \equiv \int_{\mathbb{R}^n} g(x, \tilde{y} + \Sigma_{21}
\Sigma_{11}^{-1} x ) \cdot \frac{1}{(2 \pi | \Sigma_{22\cdot 1} | )^{n/2}}
\exp \left ( - \frac{1}{2} \tilde{y}^T \Sigma_{22\cdot 1}^{-1} \tilde{y}
\right ) d \tilde{y}
\end{eqnarray*}
is log-concave: $g$ is log-concave, and hence $\tilde{g} (x, \tilde{y})
\equiv g(x, \tilde{y} + \Sigma_{21} \Sigma_{11}^{-1} x )$ is log-concave;
the product $\tilde{g} (x, \tilde{y}) \cdot \exp ( - (1/2) \tilde{y}^T
\Sigma_{22\cdot 1}^{-1} \tilde{y} )$ is (jointly) log-concave; and hence $h$
is log-concave by Pr\'ekopa's Theorem~\ref{LCpreservedByMargin}.
\end{proof}

\begin{proof}
(Theorem~\ref{StrongLogConPreservOne}): (a) The density $q$ is given by $%
q(y) = p(A^{-1} y) \det(A^{-1} )$ by a standard computation. Then since $p
\in SLC_2 ( 0 , \Sigma,d)$ we can write 
\begin{eqnarray*}
q(y) = g(A^{-1} y) \det (A^{-1} ) \phi_{\Sigma} (A^{-1} y) = g(A^{-1} y)
\phi_{A \Sigma A^T} (y)
\end{eqnarray*}
where $g(A^{-1} y)$ is log-concave by Proposition~\ref%
{prop:AffinePreservLogCon}. \newline
(b) If $p \in SLC_2 (0, \Sigma, d)$ and $q \in SLC_2 (0, \Gamma,d)$, then
the function 
\begin{eqnarray*}
h(z,x) = p(x) q(z-x)
\end{eqnarray*}
is strongly log-concave jointly in $x$ and $z$: since 
\begin{eqnarray*}
\lefteqn{x^T \Sigma^{-1} x + (z-x)^T \Gamma^{-1} (z-x) } \\
& = & z^T (\Sigma + \Gamma)^{-1} z + (x - Cz)^T ( \Sigma^{-1} + \Gamma^{-1}
) (x - Cz)
\end{eqnarray*}
where $C \equiv (\Sigma^{-1} + \Gamma^{-1} )^{-1} \Gamma^{-1}$, it follows
that 
\begin{eqnarray*}
h(z,x) & = & g_p (x) g_q (z-x) \phi_{\Sigma}(x) \phi_{\Gamma} (z-x) \\
& = & g(z,x) \phi_{\Sigma + \Gamma} (z) \cdot \phi_{\Sigma^{-1} +
\Gamma^{-1} } (x - Cz)
\end{eqnarray*}
is jointly log-concave. Hence it follows 
%from Pr\'ekopa's theorem (Theorem~\ref{LCpreservedByMargin})  
that 
\begin{eqnarray*}
p\star q (z) = \int_{\mathbb{R}^d} h(z,x) dx & = & \phi_{\Sigma + \Gamma}
(z) \int_{\mathbb{R}^d} g(z,x) \phi_{\Sigma^{-1} + \Gamma^{-1} } (x-Cz) dx \\
& \equiv & \phi_{\Sigma + \Gamma} (z) g_0 (z)
\end{eqnarray*}
where $g_0 (z)$ is log-concave by Pr\'ekopa's theorem, Theorem~\ref%
{LCpreservedByMargin}. \newline
(c ) This is easy since 
\begin{eqnarray*}
p(x) q(y) & = & g_p (x) g_q (y) \phi_{\Sigma} (x) \phi_{\Gamma} (y) = g(x,y)
\phi_{\tilde{\Sigma}} (x,y)
\end{eqnarray*}
where $\tilde{\Sigma} $ is the given $2d\times 2d$ block diagonal matrix and 
$g$ is jointly log-concave (by Proposition~\ref{prop:ProdMeasPreservLogConc}%
). \newline
(d) Note that 
\begin{eqnarray*}
p(x) q(x) = g_p (x) g_q (x) \phi_{\Sigma} (x) \cdot \phi_{\Gamma} (x) \equiv
g_0 (x) \phi_{( \Sigma^{-1} + \Gamma^{-1} )^{-1}} (x)
\end{eqnarray*}
where $g_0 $ is log-concave.
\end{proof}

%\par\noindent
%{\bf Jon:}  now give the proof of Theorem~\ref{StrongLogConPreservOne}.  

%----------------------------------------------------------

\section{Log-concavity and ultra-log-concavity for discrete distributions}
\label{Sec:DiscreteLogConc}

%----------------------------------------------------------

We now consider %begin with a discussion of 
log-concavity and ultra-log-concavity in the setting of discrete random
variables. Some of this material is from \cite{MR3030616} 
%JohnsonKontoyiannisMadiman:2013 
and \cite{MR2327839}.  %\johnson:2007

An integer-valued random variable $X$ with probability mass function 
$\{ p_x : \ x \in \mathbb{Z}\}$ is \textsl{log-concave} if 
\begin{eqnarray}
p_x^2 \ge p_{x+1} p_{x-1} \ \ \mbox{for all} \ x \in \mathbb{Z} .
\end{eqnarray}
If we define the \textsl{score function} $\varphi$ by $\varphi (x) \equiv
p_{x+1}/p_x$, then log-concavity of $\{ p_x\}$ is equivalent to $\varphi$
being decreasing (nonincreasing).

A stronger notion, analogous to strong log-concavity in the case of
continuous random variables, is that of \textsl{ultra-log-concavity}: for
any $\lambda > 0$ define $\mathbf{U L C} (\lambda)$ to be the class of
integer-valued random variables $X$ with mean $E X = \lambda $ such that the
probability mass function $p_x$ satisfies 
\begin{eqnarray}
x p_x^2 \ge (x+1) p_{x+1} p_{x-1} \ \ \ \mbox{for all} \ \ x\ge 1.
\label{ULCversion1}
\end{eqnarray}
Then the class of ultra log-concave random variables is $\mathbf{ULC} =
\cup_{\lambda>0} \mathbf{ULC} (\lambda )$. Note that (\ref{ULCversion1}) is
equivalent to log-concavity of $x \mapsto p_x / \pi_{\lambda,x} $ where $%
\pi_{\lambda ,x} = e^{-\lambda} \lambda^x / x!$ is the Poisson distribution
on $\mathbb{N}$, and hence ultra-log-concavity corresponds to $p$ being 
\textsl{log-concave relative to $\pi_{\lambda}$} (or $p \le_{{\small %
\mbox{lc}}} \pi_{\lambda}$) in the sense defined by \cite{MR799285}.
Equivalently, $p_x= h_x \pi_{\lambda, x}$ where $h$ is log-concave. When we
want to emphasize that the mass function $\{ p_x \}$ corresponds to $X$, we
also write $p_X (x) $ instead of $p_x$.

If we define the \textsl{relative score function} $\rho$ by 
\begin{eqnarray*}
\rho (x) & \equiv & \frac{(x+1)p_{x+1}}{\lambda p_x} -1 ,
\end{eqnarray*}
then $X \sim p \in \mathbf{U L C} (\lambda)$ if and only if $\rho $ is
decreasing (nonincreasing). Note that 
\begin{equation*}
\rho (x) = \frac{(x+1)\varphi (x)}{\lambda} -1 = \frac{(x+1)\varphi (x)}{%
\lambda} - \frac{(x+1)\pi_{\lambda, x+1}}{\lambda \pi_{\lambda , x}}.
\end{equation*}
% then log-concavity of $\{ p_x\}$ is equivalent
%to $\varphi$ being decreasing (non increasing).

\smallskip

Our main interest here is the preservation of log-concavity and
ultra-log-concavity under convolution. \smallskip

\begin{theorem}
\label{thm:UltraLogConvPreservByConv} %\noindent \textbf{Theorem 1.} 
(a) (\cite{Keilson-Gerber:71}) The class of log-concave distributions on $%
\mathbb{Z}$ is closed under convolution. If $U \sim p$ and $V\sim q$ are
independent and $p$ and $q$ are log-concave, then $U+V \sim p\star q $ is
log-concave.\newline
(b) (\cite{MR0494391}, \cite{MR1462561}) %Walkup,  Liggett,
The class of ultra-log-concave distributions on $\mathbb{Z}$ is closed under
convolution. More precisely, these classes are closed under convolution in
the following sense: if $U \in \mathbf{U L C} (\lambda)$ and $V \in \mathbf{%
U L C} (\mu)$ are independent, then $U+V \in \mathbf{U L C} (\lambda + \mu)$.
\end{theorem}

\smallskip

Actually, \cite{Keilson-Gerber:71} proved more: analogously to \cite%
{Ibragimov:56} %Ibragimov (1956)
they showed that $p$ is strongly unimodal (i.e. $X+Y \sim p \star q$ with $%
X,Y$ independent is unimodal for every unimodal $q$ on $\mathbb{Z}$) if and
only if $X \sim p$ is log-concave. Liggett's proof of (b) proceeds by direct
calculation; see also \cite{MR0494391}. % Walkup 1976 
For recent alternative proofs of this property of ultra-log-concave
distributions, see \cite{MR2482101} and \cite{MR2793607}. A relatively
simple proof is given by \cite{MR2327839} using results from \cite{MR2236061}
and \cite{MR0171335}, and that is the proof we will summarize here. See \cite%
{MR2929095} for an application of ultra log-concavity and Theorem~\ref%
{thm:UltraLogConvPreservByConv} to finding optimal constants in Khinchine
inequalities. \smallskip

Before proving Theorem~\ref{thm:UltraLogConvPreservByConv} we need the
following lemma giving the score and the relative score of a sum of
independent integer-valued random variables.

\begin{lemma}
\label{ConvolutionAndScoresDiscreteCase} If $X, Y$ are independent
non-negative integer-valued random variables with mass functions $p = p_X$
and $q=p_Y$ then:\newline
(a) $\varphi_{X+Y} (z) = E \{ \varphi_X (X) | X+Y = z \}$. \newline
(b) If, moreover, $X$ and $Y$ have means $\mu$ and $\nu$ respectively, then
with $\alpha = \mu / (\mu + \nu)$, 
\begin{eqnarray*}
\rho_{X+Y} (z) = E \{ \alpha \rho_X (X) + (1-\alpha) \rho_Y (Y) \big | X + Y
= z \} .
\end{eqnarray*}
\end{lemma}

\begin{proof}
For (a), note that with $F_z \equiv p_{X+Y}(z)$ we have 
\begin{eqnarray*}
\varphi_{X+Y} (z) & = & \frac{p_{X+Y} (z+1)}{p_{X+Y} (z)} = \sum_x \frac{p
(x) q (z+1-x)}{F_z} \\
& = & \sum_x \frac{p(x)}{p(x-1)} \cdot \frac{p(x-1) q(z+1-x)}{F_z} \\
& = & \sum_x \frac{p(x+1)}{p(x)} \cdot \frac{p(x) q(z-x)}{F_z} .
\end{eqnarray*}
To prove (b) we follow \cite{MR2236061}, page 471: using the same notation
as in (a), 
\begin{eqnarray*}
\rho_{X+Y} (z) & = & \frac{(z+1) p_{X+Y} (z+1)}{(\mu+\nu) p_{X+Y} (z)} -1 \\
& = & \sum_x \frac{(z+1)p(x) q(z+1-x)}{(\mu+\nu) F_z} -1 \\
& = & \sum_x \left \{ \frac{x p(x) q(z+1-x)}{(\mu+\nu) F_z} \ + \ \frac{%
(z-x+1)p(x) q(z+1-x)}{(\mu+\nu) F_z} \right \} -1 \\
& = & \alpha \left \{ \sum_x \frac{x p_X (x)}{\mu p(x-1)} \cdot \frac{p(x-1)
q(z-x+1)}{F_z} -1 \right \} \\
&& \ \ \ + \ (1-\alpha) \left \{ \sum_x \frac{z-x+1}{\nu} \frac{q(z-x+1)}{%
q(z-x)} \cdot \frac{ p(x) q(z-x)}{F_z} - 1 \right \} \\
& = & \sum_x \frac{p(x) q(z-x)}{F_z} \left \{ \alpha \rho_X (x) + (1-\alpha)
\rho_Y (z-x) \right \} .
\end{eqnarray*}
\end{proof}

\begin{proof}
~Theorem~\ref{thm:UltraLogConvPreservByConv}: (b) This follows from (b) of
Lemma~\ref{ConvolutionAndScoresDiscreteCase} and Theorem 1 of \cite%
{MR0171335}, upon noting Efron's remark 1, page 278, concerning the discrete
case of his theorem: for independent log-concave random variables $X$ and $Y$
and a measurable function $\Phi$ monotone (decreasing here) in each
argument, $E \{ \Phi (X,Y) | X+Y = z \}$ is a monotone decreasing function
of $z$: note that log-concavity of $X$ and $Y$ implies that 
\begin{equation*}
\Phi (x,y) = \frac{\mu}{\mu+ \nu} \rho_X (x) + \frac{\nu}{\mu + \nu} \rho_Y
(y)
\end{equation*}
is a monotone decreasing function of $x$ and $y$ (separately) by since the
relative scores $\rho_X$ and $\rho_Y$ are decreasing. Thus $\rho_{X+Y}$ is a
decreasing function of $z$, and hence $X+Y \in \mathbf{ULC} (\mu + \nu)$.

(a) Much as in part (b), this follows from (a) of Lemma~\ref%
{ConvolutionAndScoresDiscreteCase} and Theorem 1 of \cite{MR0171335}, upon
replacing the relative scores $\rho_X$ and $\rho_Y$ by scores $\varphi_X$
and $\varphi_Y$ and by taking $\Phi (x,y) = \varphi_X (x)$.
\end{proof}

For interesting results concerning the entropy of discrete random variables,
Bernoulli sums, log-concavity, and ultra-log-concavity, see \cite{MR3030616}%
, %  JohnsonKontoyiannisMadiman:2013 
\cite{MR1093412}, and \cite{MR2327839}. For recent results concerning
nonparametric estimation of a discrete log-concave distribution, see 
\cite{MR3091658} %Balabdaoui, Jankowski, Rufibach, and Pavlides
 and \cite{MR3174308}. %Balabdaoui (2013)  %\medskip
%\noindent \textbf{Fact 5.} 
It follows from \cite{MR1093412} that the hypergeometric distribution
(sampling without replacement count of \textquotedblleft
successes\textquotedblright ) is equal in distribution to a Bernoulli sum;
hence the hypergeometric distribution is ultra-log-concave. \medskip

%----------------------------------------------------------
\section{Regularity and approximations of log-concave functions}
\label{section_reg_and_ap}

%----------------------------------------------------------

\subsection{Regularity\label{section_reg}}

The regularity of a log-concave function $f=\exp \left( -\varphi \right) $
depends on the regularity of its convex potential $\varphi $. Consequently,
log-concave functions inherit the special regularity properties of convex
functions.

Any log-concave function is nonnegative. When the function $f$ is a
log-concave \textit{density} (with respect to the Lebesgue measure), which
means that $f$ integrates to $1$, then it is automatically bounded. More
precisely, it has exponentially decreasing tails and hence, it has finite $%
\Psi _{1}$ Orlicz norms; for example, see \cite{MR733944} and \cite%
{MR1849347}. The following lemma gives a pointwise estimate of the density.

\begin{theorem}[\protect\cite{MR2645484}, Lemma 1]
\label{theorem_exp_tails} Let $f$ be a log-concave density on $\mathbb{R}%
^{d} $. Then there exist $a_{f}=a>0$ and $b_{f}=b\in \mathbb{R}$ such that $%
f\left( x\right) \leq e^{-a\left\Vert x\right\Vert +b}$ for all $x\in 
\mathbb{R}^{d}$.
\end{theorem}

Similarly, strong log-concavity implies a finite $\Psi _{2}$ Orlicz norm
(see \cite{MR1849347} Theorem 2.15, page 36, \cite{MR1964483}, Theorem 9.9,
page 280), \cite{MR1742893}, and \cite{MR1682772}.

For other pointwise bounds on log-concave densities themselves, see \cite%
{MR773927}, %Devroye 1984
\cite{MR2546798} and \cite{MR2309621}.

As noticed in \cite{MR2645484}, Theorem \ref{theorem_exp_tails} implies that
if a random vector $X$ has density $f$, then the moment generating function
of $X$ is finite in an open neighborhood of the origin. Bounds can also be
obtained for the supremum of a log-concave density as well as for its values
on some special points in the case where $d=1$.

\begin{proposition}
\label{prop_pointwise_bounds}Let $X$ be a log-concave random variable, with
density $f$ on $\mathbb{R}$ and median $m$. Then 
\begin{eqnarray}
\frac{1}{12%
%TCIMACRO{\TeXButton{var}{\var}}%
%BeginExpansion
\var%
%EndExpansion
\left( X\right) } &\leq &f\left( m\right) ^{2}\leq \frac{1}{2%
%TCIMACRO{\TeXButton{var}{\var}}%
%BeginExpansion
\var%
%EndExpansion
\left( X\right) }\text{ ,}  \label{line1_prop_pointwise} \\
\frac{1}{12%
%TCIMACRO{\TeXButton{var}{\var}}%
%BeginExpansion
\var%
%EndExpansion
\left( X\right) } &\leq &\sup_{x\in \mathbb{R}}f\left( x\right) ^{2}\leq 
\frac{1}{%
%TCIMACRO{\TeXButton{var}{\var}}%
%BeginExpansion
\var%
%EndExpansion
\left( X\right) }\text{ ,}  \label{line2_prop_pointwise} \\
\frac{1}{3e^{2}%
%TCIMACRO{\TeXButton{var}{\var}}%
%BeginExpansion
\var%
%EndExpansion
\left( X\right) } &\leq &f\left( \mathbb{E}\left[ X\right] \right) ^{2}\leq 
\frac{1}{%
%TCIMACRO{\TeXButton{var}{\var}}%
%BeginExpansion
\var%
%EndExpansion
\left( X\right) }\text{ .}  \label{line3_prop_pointwise}
\end{eqnarray}
\end{proposition}

Proposition \ref{prop_pointwise_bounds} can be found in \cite%
{BobkovLedoux:14}, Proposition B.2. See references therein for historical
remarks concerning these inequalities. Proposition \ref%
{prop_pointwise_bounds} can also be seen as providing bounds for the
variance of a log-concave variable.  
See \cite{Kim-Samworth:14}, section 3.2, for some further results of this type.

Notice that combining (\ref{line1_prop_pointwise}) and (\ref%
{line2_prop_pointwise}) we obtain the inequality $\sup_{x\in \mathbb{R}%
}f\left( x\right) \leq 2\sqrt{3}f\left( m\right) $. In fact, the concavity
of the function $I$ defined in Proposition \ref{prop_bobkov} allows to prove
the stronger inequality $\sup_{x\in \mathbb{R}}f\left( x\right) \leq
2f\left( m\right) $. Indeed, with the notations of Proposition \ref%
{prop_bobkov}, we have $I\left( 1/2\right) =f\left( m\right) $ and for any $%
x\in \left( a,b\right) $, there exists $t\in \left( 0,1\right) $ such that $%
x=F^{-1}\left( t\right) $. 
Hence,%
\begin{eqnarray*}
\lefteqn{2f\left( m\right) =2I\left( \frac{1}{2}\right) 
        = 2I\left( \frac{t}{2}+\frac{1-t}{2}\right)} \\
& \geq & 2\left( \frac{1}{2}I\left( t\right) 
                  +\frac{1}{2}I\left( 1-t\right) \right) \geq I\left( t\right) =f\left( x\right) \text{ .}
\end{eqnarray*}

A classical result on continuity of convex functions is that any real-valued
convex function $\varphi $ defined on an open set $U\subset \mathbb{R}^{d}$
is locally Lipschitz and in particular, $\varphi $ is continuous on $U$. For
more on continuity of convex functions see Section 3.5 of \cite{MR2178902}.
Of course, any continuity of $\varphi $ (local or global) corresponds to the
same continuity of $f$.

For an expos\'{e} on differentiability of convex functions, see 
\cite{MR2178902} (in particular sections 3.8 and 3.11; see also \cite{MR1676726}
section 7). A deep result of \cite{Alexandrov:39} is the following (we
reproduce here Theorem 3.11.2 of \cite{MR2178902}).

\begin{theorem}[\protect\cite{Alexandrov:39}]
\label{theorem_Alexandrov}Every convex function $\varphi $ on $\mathbb{R}%
^{d} $ is twice differentiable almost everywhere in the following sense: $f$
is twice differentiable at a, with Alexandrov Hessian $\nabla ^{2}f\left(
a\right) $ in $\Sym^{+}\left( d,\mathbb{R}\right) $ (the space of real
symmetric $d\times d$ matrices), if $\nabla f\left( a\right) $ exists, and
if for every $\varepsilon >0$ there exists $\delta >0$ such that%
\begin{equation*}
\left\Vert x-a\right\Vert <\delta \text{ \ \ implies \ \ }\sup_{y\in
\partial f\left( x\right) }\left\Vert y-\nabla f\left( a\right) -\nabla
^{2}f\left( a\right) \left( x-a\right) \right\Vert \leq \varepsilon
\left\Vert x-a\right\Vert \text{ .}
\end{equation*}%
Here $\partial f\left( x\right) $ is the subgradient of $f$ at $x$ (see
Definition 8.3 in \cite{MR1491362}). Moreover, if $a$ is such a point, then%
\begin{equation*}
\lim_{h\rightarrow 0}\frac{f\left( a+h\right) -f\left( a\right)
-\left\langle \nabla f\left( a\right) ,h\right\rangle -\frac{1}{2}%
\left\langle \nabla ^{2}f\left( a\right) h,h\right\rangle }{\left\Vert
h\right\Vert ^{2}}=0\text{ .}
\end{equation*}
\end{theorem}

We immediately see by Theorem \ref{theorem_Alexandrov}, that since $\varphi $
is convex and $f=\exp \left( -\varphi \right) $, it follows that $f$ is
almost everywhere twice differentiable.  For further results in the direction of 
Alexandrov's theorem see \cite{MR0482164,MR585231}.

\subsection{Approximations\label{section_ap}}

Again, if one wants to approximate a non-smooth log-concave function $f=\exp
\left( -\varphi \right) $ by a sequence of smooth log-concave functions,
then convexity of the potential $\varphi $ can be used to advantage. For an
account about approximation of convex functions see \cite{MR2178902},
section 3.8.

On the one hand, if $\varphi \in L_{loc}^{1}\left( \mathbb{R}^{d}\right) $
the space of locally integrable functions, then the standard use of a
regularization kernel (i.e. a one-parameter family of functions associated
with a mollifier) to approximate $\varphi $ preserves the convexity as soon
as the mollifier is nonnegative. A classical result is that this gives in
particular approximations of $\varphi $ in $L^{p}$ spaces, $p\geq 1$, as
soon as $\varphi \in L^{p}\left( \mathbb{R}^{d}\right) $.

On the other hand, \textit{infimal convolution} (also called 
\textit{epi-addition}, see \cite{MR1491362}) is a nonlinear analogue of
mollification that gives a way to approximate a lower
semicontinuous proper convex function  from below (section 3.8, \cite{MR2178902}). 
More precisely, take two proper convex functions $f$ and $g$ from $\mathbb{R}^{d}$
to $\mathbb{R}\cup \left\{ \infty \right\} $, which means that the functions
are convex and finite for at least one point. The infimal convolution
between $f$ and $g$, possibly taking the value $-\infty $, is%
\begin{equation*}
\left( f\odot g\right) \left( x\right) =\inf_{y\in \mathbb{R}^{n}}\left\{
f\left( x-y\right) +g\left( y\right) \right\} \text{ .}
\end{equation*}%
Then, $f\odot g$ is a proper convex function as soon as 
$f\odot g\left( x\right) >-\infty $ for all $x\in \mathbb{R}^{d}$. Now, if $f$ is a lower
semicontinuous proper convex function on $\mathbb{R}^{d}$, the Moreau-Yosida
approximation $f_{\varepsilon }$ of $f$ is given by%
\begin{eqnarray*}
f_{\varepsilon }\left( x\right) &=&\left( f\odot \frac{1}{2\varepsilon }%
\left\Vert \cdot \right\Vert ^{2}\right) \left( x\right) \\
&=&\inf_{y\in \mathbb{R}^{n}}\left\{ f\left( y\right) +\frac{1}{2\varepsilon 
}\left\Vert x-y\right\Vert ^{2}\right\} \text{ .}
\end{eqnarray*}%
for any $x\in \mathbb{R}^{d}$ and $\varepsilon >0$. The following theorem
can be found in \cite{MR1676726} (Proposition 7.13), see also \cite%
{BarbuPrecupanu:86}, \cite{Brezis:73} or \cite{MR2178902}.

\begin{theorem}
\label{theorem_Moreau-Yosida}The Moreau-Yosida approximates 
$f_{\varepsilon } $ are $\mathcal{C}^{1,1}$ (i.e. differentiable with Lipschitz derivative)
convex functions on $\mathbb{R}^{d}$ and 
$f_{\varepsilon }\rightarrow f$ as $\varepsilon \rightarrow 0$. 
Moreover, 
$\partial f_{\varepsilon }=\left( \varepsilon I+\left( \partial f\right) ^{-1}\right) ^{-1}$ 
as set-valued maps.
\end{theorem}

An interesting consequence of Theorem \ref{theorem_Moreau-Yosida} is that if
two convex and proper lower semicontinuous functions agree on their
subgradients, then they are equal up to a constant (corollary 2.10 in \cite%
{Brezis:73}).

Approximation by a regularization kernel and Moreau-Yosida approximation
have different benefits. While a regularization kernel gives the most
differentiability, the Moreau-Yosida approximation provides an approximation
of a convex function from below (and so, a log-concave function from above).
It is thus possible to combine these two kinds of approximations and obtain
the advantages of both. For an example of such a combination in the context
of a (multivalued) stochastic differential equation and the study of the
so-called Kolmogorov operator, see \cite{MR2474491}.

When considering a log-concave random vector, the following simple
convolution by Gaussian vectors gives an approximation by %strongly
log-concave vectors that have $\mathcal{C}^{\infty }$ densities and finite
Fisher information matrices. In the context of Fisher information,
regularization by Gaussians was used for instance in \cite{SidneyStone:74}\
to study the Pitman estimator of a location parameter.

\begin{proposition}[convolution by Gaussians]
\label{Lemma_convol_Gauss_multidim}
Let $X$ be a random vector in 
$\mathbb{R}^{d}$ with density $p$ w.r.t. the Lebesgue measure and $G$ a $d$
-dimensional standard Gaussian variable, independent of $X$. Set 
$Z=X+\sigma G$, $\sigma >0$ and $p_{Z}=\exp \left( -\varphi _{Z}\right) $ the density of 
$Z$. Then:
\begin{description}
%\item 
\item[(i)] If $X$ is log-concave, then $Z$ is also log-concave.

\item[(ii)] If $X$ is strongly log-concave, 
$Z\in SLC_{1}\left( \tau^{2},d\right) $then $Z$ is also strongly log-concave; 
$Z\in SLC_{1}\left( \tau ^{2}+\sigma ^{2},d\right) $.

\item[(iii)] $Z$ has a positive density $p_{Z}$ on $\mathbb{R}^{d}$.
Furthermore, $\varphi _{Z}$ is $C^{\infty }$ on $\mathbb{R}^{d}$ and 
\begin{eqnarray}
\nabla \varphi _{Z}\left( z\right) 
&=&\sigma ^{-2}\mathbb{E}\left[ \sigma G\left\vert X+\sigma G=z\right. \right]  \notag \\
&=&\mathbb{E}\left[ \rho _{\sigma G}\left( \sigma G\right) \left\vert
          X+\sigma G=z\right. \right] \text{ ,}  \label{conv_gauss_score}
\end{eqnarray}%
where $\rho _{\sigma G}\left( x\right) =\sigma ^{-2}x$ is the score of 
$\sigma G$.

\item[(iv)] The Fisher information matrix for location 
$J(Z)=\mathbb{E}\left[ \nabla \varphi _{Z}\otimes \nabla \varphi _{Z}\left( Z\right) \right] $, is
finite and we have 
$J\left( Z\right) \leq J\left( \sigma G\right)  =\sigma ^{-4}I_{d}$ as symmetric matrices.
\end{description}
\end{proposition}

\begin{proof}  
See Section~\ref{sec:Proofs}.
\end{proof}

We now give a second approximation tool, that allows to approximate any log-concave density by
strongly log-concave densities.

\begin{proposition}
\label{prop_approx_SLC}Let $f$ be a log-concave density on $\mathbb{R}^{d}$.
Then for any $c>0$, the density%
\begin{equation*}
h_{c}\left( x\right) 
=\frac{f\left( x\right) e^{-c\left\Vert x\right\Vert ^{2}/2}}
         {\int_{\mathbb{R}^{d}}f\left( v\right) e^{-c\left\Vert v\right\Vert^{2}/2}dv},
         \text{ \ \ }x\in \mathbb{R}^{d},
\end{equation*}%
is $SLC_{1}\left( c^{-1},d\right) $ and $h_{c}\rightarrow f$ as 
$c\rightarrow 0$ in $L_{p}$, $p\in \left[ 1, \infty \right] $. More
precisely, there exists a constant $A_{f}>0$ depending only on $f$, such
that for any $\varepsilon >0$,%
\begin{equation*}
\sup \left\{ \sup_{x\in \mathbb{R}^{d}}\left\vert h_{c}\left( x\right)
-f\left( x\right) \right\vert ;\left( \int_{\mathbb{R}^{d}}\left\vert
h_{c}\left( x\right) -f\left( x\right) \right\vert ^{p}dx\right)
^{1/p}\right\} \leq A_{f}c^{1-\varepsilon }\text{ .}
\end{equation*}
\end{proposition}

\begin{proof}  
See Section~\ref{sec:Proofs}.
\end{proof}

Finally, by combining Proposition \ref{prop_approx_SLC}\ and 
\ref{Lemma_convol_Gauss_multidim}, we obtain the following approximation
lemma.

\begin{proposition}
\label{lemma_approx}For any log-concave density on $\mathbb{R}^{d}$, there
exists a sequence of strongly log-concave densities that are 
$\mathcal{C}^{\infty }$, have finite Fisher information matrices and that converge to $f$
in $L_{p}\left( \leb \right) $, $p\in \left[ 1,\infty \right] $.
\end{proposition}

\begin{proof}
Approximate $f$ by a strongly log-concave density $h$ as in 
Proposition \ref{prop_approx_SLC}. Then approximate $h$ by convolving with a Gaussian
density. In the two steps the approximations can be as tight as desired in 
$L_{p}$, for any $p\in \left[ 1, \infty \right] $. The fact that the
convolution with Gaussians for a (strongly) log-concave density (that thus
belongs to any $L_{p}\left( \leb \right) $, $p\in \left[ 1,\infty \right] $) gives approximations in 
$L_{p}$, $p\in \left[ 1, \infty \right] $ is a simple application of general
classical theorems about convolution in $L_{p}$ (see for instance \cite{MR924157}, p. 148).
\end{proof}

%----------------------------------------------------------
\section{Efron's theorem and more on preservation of log-concavity and
strong log-concavity under convolution in 1-dimension}
\label{sec:EfronsTheoremOneDimension}

%----------------------------------------------------------

Another way of proving that strong log-concavity is preserved by convolution
in the one-dimensional case is by use of a result of \cite{MR0171335}. This
has already been used by \cite{MR3030616} and \cite{MR2327839} to prove
preservation of ultra log-concavity under convolution (for discrete random
variables), and by \cite{Wellner-2013} to give a proof that strong
log-concavity is preserved by convolution in the one-dimensional continuous
setting. These proofs operate at the level of scores or relative scores and
hence rely on the equivalences between (a) and (b) in Propositions~\ref%
{EquivalencesLogCon} and \ref{EquivalencesStrongLogConTry3}. Our goal in
this section is to re-examine Efron's theorem, briefly revisit the results
of \cite{MR3030616} and \cite{Wellner-2013}, give alternative proofs using
second derivative methods via symmetrization arguments, and to provide a new
proof of Efron's theorem using some recent results concerning asymmetric
Brascamp-Lieb inequalities due to \cite{Menz-Otto:2013} and \cite%
{CarlenCordero-ErausquinLieb}.

\subsection{Efron's monotonicity theorem}

\label{subsec:EfronMonoThm}

The following monotonicity result is due to \cite{MR0171335}.

\begin{theorem}[Efron]
\label{EfronMonotonicityThm} Suppose that 
%Consider two functions $g:\mathbb{R\rightarrow R}$ and 
$\Phi :\mathbb{R} ^{m}\rightarrow \mathbb{R}$ where $\Phi $ is
coordinatewise non-decreasing and let 
\begin{equation*}
g(z)\equiv E\left\{ \Phi (X_1, \cdots , X_m)\bigg |\sum_{j=1}^m X_j
=z\right\},
\end{equation*}%
where $X_1, \ldots , X_m $ are independent and log-concave. Then $g$ is
non-decreasing.
\end{theorem}

\begin{remark}
As noted by \cite{MR0171335}, Theorem~\ref{EfronMonotonicityThm} continues
to hold for integer valued random variables which are log-concave in the
sense that $p_x \equiv P(X = x)$ for $x \in \mathbb{Z}$ satisfies $p_x^2 \ge
p_{x+1} p_{x-1}$ for all $x \in \mathbb{Z}$.
\end{remark}

In what follows, we will focus on Efron's theorem for $m=2$. As it is shown
in \cite{MR0171335}, the case of a pair of variables ($m=2$) indeed implies
the general case with $m \ge 2$. Let us recall the argument behind this
fact, which involves preservation of log-concavity under convolution. In
fact, stability under convolution for log-concave variables is not needed to
prove Efron's theorem for $m=2$ as will be seen from the new proof of
Efron's theorem given here in Section~\ref%
{subsec:AlternativePrfOfEfronViaAsymmBrascampLieb}, so it is consistent to
prove the preservation of log-concavity under convolution via Efron's
theorem for $m=2$.

\begin{proposition}
If Theorem \ref{EfronMonotonicityThm} is satisfied for $m=2$, then it is
satisfied for $m\geq 2$.
\end{proposition}

\begin{proof}
We proceed as in \cite{MR0171335} by induction on $m\geq 2$. Let $\left(
X_{1},\ldots ,X_{m}\right) $ be a $m-$tuple of log-concave variables, let $%
S=\sum_{i=1}^{m}X_{i}$ be their sum and set 
\begin{equation*}
\Lambda \left( t,u\right) =\mathbb{E}\left[ \Phi \left( X_{1},\ldots
,X_{m}\right) \left\vert \sum_{i=1}^{m-1}X_{i}=t\text{ },\text{ }%
X_{m}=u\right. \right] \text{ .}
\end{equation*}%
Then 
\begin{equation*}
\mathbb{E}\left[ \Phi \left( X_{1},\ldots ,X_{m}\right) \left\vert
S=s\right. \right] =\mathbb{E}\left[ \Lambda \left( T,X_{m}\right)
\left\vert T+X_{m}=s\right. \right] \text{ ,}
\end{equation*}%
where $T=\sum_{i=1}^{m-1}X_{i}$. Hence, by the induction hypothesis for
functions of two variables, it suffices to prove that $\Lambda $ is
coordinatewise non-decreasing. As $T$ is a log-concave variable (by
preservation of log-concavity by convolution), $\Lambda \left( t,u\right) $
is non-decreasing in $t$ by the induction hypothesis for functions of $m-1$
variables. Also $\Lambda \left( t,u\right) $ is non-decreasing in $u$ since $%
\Phi $ is non-decreasing in its last argument. %This concludes the proof.
\end{proof}

\cite{MR0171335} also gives the following corollaries of Theorem \ref%
{EfronMonotonicityThm}\ above.

\begin{corollary}
\label{CorOneOfEfron} Let $\{\Phi _{t}(x_{1},\ldots ,x_{m}):\ t\in T\}$ be a
family of measurable functions increasing in every argument for each fixed
value of $t$, and increasing in $t$ for each fixed value of $%
x_{1},x_{2},\ldots ,x_{m}$. Let $X_{1},\ldots ,X_{m}$ be independent and
log-concave and write $S\equiv \sum_{j=1}^{m}X_{j}$. Then 
\begin{equation*}
g(a,b)=E\left\{ \Phi _{a+b-S}(X_{1},\cdots ,X_{m})\bigg |a\leq S\leq
a+b\right\}
\end{equation*}%
is increasing in both $a$ and $b$.
\end{corollary}

\begin{corollary}
\label{CorTwoOfEfron} Suppose that the hypotheses of Theorem~\ref%
{EfronMonotonicityThm} hold and that $A=\{x=(x_{1},\ldots ,x_{m})\in \mathbb{%
R}^{m}:a_{j}\leq x_{j}\leq b_{j}\}$ with $-\infty \leq a_{j}<b_{j}\leq
\infty $ is a rectangle in $\mathbb{R}^{m}$. Then 
\begin{equation*}
g(z)\equiv E\left\{ \Phi (X_{1},\cdots ,X_{m})\bigg |%
\sum_{j=1}^{m}X_{j}=z,(X_{1},\ldots ,X_{m})\in A\right\}
\end{equation*}%
is a non-decreasing function of $z$.
\end{corollary}

The following section will give applications of Efron's theorem to
preservation of log-concavity and strong log-concavity in the case of
real-valued variables.

\subsection{First use of Efron's theorem: strong log-concavity is preserved
by convolution via scores}

\begin{theorem}
\label{LCandSLCpreservedByConvOneDim} (log-concavity and strong
log-concavity preserved by convolution via scores)\newline
(a) (\cite{MR0087249}) If $X $ and $Y$ are independent and log-concave with
densities $p$ and $q$ respectively, then $X+Y \sim p\star q$ is also
log-concave.\newline
(b) If $X \in SLC_1( \sigma^2, 1)$ 
%\mathbf{S} \mathbf{L} \mathbf{C}(\sigma^2)$ 
and $Y \in SLC_1 (\tau^2 , 1)$ 
%\mathbf{S} \mathbf{L} \mathbf{C} (\tau^2)$ are
are independent, then $X + Y \in SLC_1 ( \sigma^2 + \tau^2 , 1) $ 
%\mathbf{S} \mathbf{L} \mathbf{C} (\sigma^2 +\tau^2)$. 
\end{theorem}

Actually, \cite{MR0087249} proved more: he showed that $p$ is strongly
unimodal (i.e. $X+Y\sim p\star q$ with $X,Y$ independent is unimodal for
every unimodal $q$ on $\mathbb{R}$) if and only if $X$ is log-concave.

\begin{proof}
(a) From Proposition~\ref{EquivalencesLogCon} log-concavity of $p$ and $q$
is equivalent to monotonicity of their score functions $\varphi _{p}^{\prime
}=(-\log p)^{\prime }=-p^{\prime }/p$ $a.e.$ and $\varphi _{q}^{\prime
}=(-\log q)^{\prime }=-q^{\prime }/q$ $a.e.$ respectively. From the
approximation scheme described in 
Proposition~\ref{Lemma_convol_Gauss_multidim} above, 
we can assume that both $p$ and $q$ are absolutely
continuous. Indeed, Efron's theorem applied to formula (\ref{conv_gauss_score}) 
of Proposition~\ref{Lemma_convol_Gauss_multidim} with $m=2$ and 
$\Phi (x,y)=\rho _{\sigma G}(x)$, gives that the convolution with a
Gaussian variable preserves log-concavity. Then, from Lemma 3.1 of 
\cite{MR2128239}, 
\begin{equation*}
E\left\{ \rho _{X}(X)\bigg |X+Y=z\right\} =\rho _{X+Y}(z).
\end{equation*}
%
%
%
%
%
%
%
%
%
%
%%%
%Since convolution is symmetric in $X$ and $Y$ we also have
%\begin{eqnarray}
%E \left \{ \rho_Y (Y) \bigg | X + Y = z \right \} = \rho_{X+Y} (z) .
%\label{BarronJohnsonProjectionFormulaConvolutionVersionTwo}
%\end{eqnarray}
%Thus for any $\lambda \in [0,1]$ and  expressing everything in terms of $\varphi_{\#}^{\prime}$ 
%for $\# \in \{ p, q , p\star q \}$ it follows that 
%\begin{eqnarray*}
%E \{ \lambda \varphi_p^{\prime} (X) + (1-\lambda ) \varphi_q^{\prime} (Y) \big | X+Y = z \} 
%= \varphi_{p\star q}^{\prime} (z) .
%\end{eqnarray*}
%Thus by Efron's theorem with $m=2$ and 
%$$
%\Phi (x,y) = \lambda \varphi_p^{\prime} (x) + (1-\lambda ) \varphi_q^{\prime} (y) ,
%$$
%we see that $E \{ \Phi (X,Y) | X+Y = z \} = \varphi_{p\star q}^{\prime} (z) $ is a monotone function of $z$,
%%%
Thus by Efron's theorem with $m=2$ and 
\begin{equation*}
\Phi (x,y)=\rho _{Y}(y),
\end{equation*}%
we see that $E\{\Phi (X,Y)|X+Y=z\}=\varphi _{p\star q}^{\prime }(z)$ is a
monotone function of $z$, and hence by Proposition~\ref{EquivalencesLogCon},
(a) if and only if (b), log-concavity of the convolution $p\star q=p_{X+Y}$
holds. \medskip

(b) %%%
%The proof of preservation of strong log-concavity under convolution for $p$ and $q$ 
%strong log-concave on $\RR$ is similar to the proof of (a), but with scores replaced by relative scores
%%%
The proof of preservation of strong log-concavity under convolution for $p$
and $q$ strong log-concave on $\mathbb{R}$ is similar to the proof of (a),
but with scores replaced by relative scores, but it is interesting to note
that a symmetry argument is needed. From Proposition~\ref%
{EquivalencesStrongLogConTry3} strong log-concavity of $p$ and $q$ is
equivalent to monotonicity of their relative score functions $\rho
_{p}(x)\equiv \varphi _{p}^{\prime }(x)-x/\sigma ^{2}$ and $\rho
_{q}(x)\equiv \varphi _{q}^{\prime }(x)-x/\tau ^{2}$ respectively. Now we
take $m=2$, $\lambda \equiv \sigma ^{2}/(\sigma ^{2}+\tau ^{2})$, and define 
\begin{equation*}
\Phi (x,y)=\lambda \rho _{p}(x)+(1-\lambda )\rho _{q}(y).
\end{equation*}%
Thus $\Phi $ is coordinate-wise monotone and by using Lemma \ref%
{lem:ScoreProjectionMultivariate} with $d=1$ we find that 
\begin{equation*}
E\{\Phi (X,Y)|X+Y=z\}=\varphi _{p\star q}^{\prime }(z)-\frac{z}{\sigma
^{2}+\tau ^{2}}=\rho _{p\star q}(z).
\end{equation*}%
Hence it follows from Efron's theorem that the relative score $\rho _{p\star
q}$ of the convolution $p\star q$, is a monotone function of $z$. By
Proposition~\ref{EquivalencesStrongLogConTry3}(b) it follows that $p\star
q\in SLC_{1}(\sigma ^{2}+\tau ^{2},1)$.
\end{proof}

\subsection{A special case of Efron's theorem via symmetrization}

We now consider a particular case of Efron's theorem. Our motivation is as
follows: in order to prove that strong log-concavity is preserved under
convolution, recall that we need to show monotonicity in $z$ of 
%We can notice that, in order to follow the proof of stability under
%convolution for strong-log-concave distributions based on scores and the use
%of Efron's theorem, we only need a particular case of the latter theorem.
%Indeed, as we are interested by the monotonicity in $z$ of%
\begin{equation*}
\rho _{X+Y}(z)=E\left\{ \frac{\sigma ^{2}}{\sigma ^{2}+\tau ^{2}}\rho
_{X}(X)+\frac{\tau ^{2}}{\sigma ^{2}+\tau ^{2}}\rho _{Y}(Y)\bigg | %
X+Y=z\right\} \text{.}
\end{equation*}
Thus we only need to consider functions $\Phi $ of the form 
\begin{equation*}
\Phi (X,Y)=\Psi \left( X\right) +\Gamma \left( Y\right),
\end{equation*}
where $\Psi $ and $\Gamma $ are non-decreasing, and show the monotonicity of 
\begin{equation*}
E\left\{ \Phi (X,Y)\bigg |X+Y=z\right\} 
\end{equation*}
for functions $\Phi$ of this special form.
By symmetry between $X$ and $Y$, this reduces to the study of the
monotonicity of 
\begin{equation*}
E\left\{ \Psi \left( X\right) \bigg |X+Y=z\right\} \text{.}
\end{equation*}
We now give a simple proof of this monotonicity in dimension $1$.

\begin{proposition}
\label{prop_efron_spe_dim_1} Let $\Psi :\mathbb{R\rightarrow R}$ be
non-decreasing and suppose that $X \sim f_X$, $Y \sim f_Y$ are independent
and that $f_X, f_Y$ are log-concave. If the function $\eta : \mathbb{R}
\rightarrow \mathbb{R}$ given by%
\begin{equation*}
\eta(z) \equiv E\left\{ \Psi \left( X\right) \bigg |X+Y=z\right\}
\end{equation*}%
is well-defined ($\Psi $ integrable with respect to the conditional law of $%
X+Y$), %and if the function 
%$x\mapsto \left\vert \Psi \left( x\right) f_{Y}^{\prime }\left( z-x\right) f_{X}\left( x\right) \right\vert $ is
%integrable (w.r.t. the Lebesgue measure) 
then it is non-decreasing.
\end{proposition}

\begin{proof}
First notice that by truncating the values of $\Psi $ and using the monotone
convergence theorem, we assume that $\Psi $ is bounded. Moreover, by 
Proposition~\ref{Lemma_convol_Gauss_multidim}, we may assume that $f_{Y}$ is 
$\mathcal{C}^{1}$ with finite Fisher information, thus justifying the
following computations. We write 
\begin{equation*}
E\left\{ \Psi \left( X\right) \bigg |X+Y=z\right\} =\int_{\mathbb{R}}\Psi
\left( x\right) \frac{f_{X}\left( x\right) f_{Y}\left( z-x\right) }{F_{z}}dx%
\text{ ,}
\end{equation*}%
where 
\begin{equation*}
F_{z}=\int_{\mathbb{R}}f_{X}\left( x\right) f_{Y}\left( z-x\right) dx>0\text{
.}
\end{equation*}%
Moreover, with $f_{X}=\exp (-\varphi _{X})$ and $f_{Y}=\exp (-\varphi _{Y})$%
, 
\begin{eqnarray*}
\lefteqn{\frac{\partial }{\partial z}\left( \Psi \left( x\right) \frac{%
f_{X}\left( x\right) f_{Y}\left( z-x\right) }{F_{z}}\right) } \\
&=&-\Psi \left( x\right) \varphi _{Y}^{\prime }\left( z-x\right) \frac{%
f_{X}\left( x\right) f_{Y}\left( z-x\right) }{F_{z}} \\
&&\ \ +\ \Psi \left( x\right) \frac{f_{X}\left( x\right) f_{Y}\left(
z-x\right) }{F_{z}}\int_{\mathbb{R}}\varphi _{Y}^{\prime }\left( z-x\right) 
\frac{f_{X}\left( x\right) f_{Y}\left( z-x\right) }{F_{z}}dx\text{ ,}
\end{eqnarray*}%
where $\varphi _{Y}^{\prime }\left( y\right) =-f_{Y}^{\prime }\left(
y\right) /f_{Y}\left( y\right) $. As $f_{X}$ is bounded (see Section \ref%
{section_reg}) and $Y$ has finite Fisher information, we deduce that $\int_{%
\mathbb{R}}\left\vert \varphi _{Y}^{\prime }\left( z-x\right) \right\vert
f_{X}\left( x\right) f_{Y}\left( z-x\right) dx$ is finite. Then, 
\begin{eqnarray*}
\lefteqn{\frac{\partial }{\partial z}\left( E\left\{ \Psi \left( X\right)
\bigg|X+Y=z\right\} \right) } \\
&=&-E\left\{ \Psi \left( X\right) \cdot \varphi _{Y}^{\prime }\left(
Y\right) \bigg |X+Y=z\right\} +E\left\{ \Psi \left( X\right) \bigg |%
X+Y=z\right\} E\left\{ \varphi _{Y}^{\prime }\left( Y\right) \bigg |%
X+Y=z\right\} \\
&=&-\cov\left\{ \Psi \left( X\right) ,\varphi _{Y}^{\prime }\left( Y\right) %
\bigg |X+Y=z\right\} \text{ .}
\end{eqnarray*}%
If we show that the latter covariance is negative, the result will follow.
Let $\left( \tilde{X},\tilde{Y}\right) $ be an independent copy of 
$\left( X,Y\right) $. Then 
\begin{eqnarray*}
\lefteqn{E\left\{ \left( \Psi \left( X\right) -\Psi \left( \tilde{X}\right)
                         \right) \left( \varphi _{Y}^{\prime }\left( Y\right) 
                         -\varphi _{Y}^{\prime }
                \left( \tilde{Y}\right) \right) \bigg|\tilde{X}+\tilde{Y}=z,X+Y=z\right\} } \\
& = & 2\cov\left\{ \Psi \left( X\right) ,\varphi _{Y}^{\prime }\left( Y\right) \bigg |X+Y=z\right\} \text{ .}
\end{eqnarray*}
Furthermore, since $X\geq \tilde{X}$ implies $Y\leq \tilde{Y}$ under the
given condition $[X+Y=z, \tilde{X}+\tilde{Y} = z]$, 
\begin{eqnarray*}
\lefteqn{E\left\{ \left( \Psi \left( X\right) -\Psi \left( \tilde{X}\right)
             \right) \left( \varphi _{Y}^{\prime }\left( Y\right) -\varphi _{Y}^{\prime} \left( 
              \tilde{Y}\right) \right) \bigg|\tilde{X}+\tilde{Y}=z,X+Y=z\right\} } \\
&=&2E\left\{ \underset{\geq 0}{\underbrace{\left( \Psi \left( X\right) -\Psi
               \left( \tilde{X}\right) \right) }}\underset{\leq 0}{\underbrace{\left(
               \varphi _{Y}^{\prime }\left( Y\right) -\varphi _{Y}^{\prime }\left( \tilde{Y}%
               \right) \right) }}\mathbf{1}_{\left\{ X\geq \tilde{X}\right\} }
               \bigg |X+Y=z,\tilde{X}+\tilde{Y}=z\right\} \\
&\leq &0.
\end{eqnarray*}
This proves Proposition~\ref{prop_efron_spe_dim_1}.
\end{proof}

\subsection{Alternative proof of Efron's theorem via asymmetric
Brascamp-Lieb inequalities}

\label{subsec:AlternativePrfOfEfronViaAsymmBrascampLieb}

Now our goal is to give a new proof of Efron's 
Theorem~\ref{EfronMonotonicityThm} in the case $m=2$ using results related to recent
asymmetric Brascamp-Lieb inequalities and covariance formulas due to 
\cite{Menz-Otto:2013}.

\begin{theorem}[Efron]
Suppose that %Consider two functions $g:\mathbb{R\rightarrow R}$ and 
$\Phi :\mathbb{R}^{2}\rightarrow \mathbb{R}$, such that $\Phi $ is
coordinatewise non-decreasing and let 
\begin{equation*}
g(z)\equiv E\left\{ \Phi (X,Y)\bigg |X+Y=z\right\} \text{ },
\end{equation*}%
where $X$ and $Y$ are independent and log-concave. Then $g$ is
non-decreasing.
\end{theorem}

\begin{proof}
Notice that by truncating the values of $\Phi $ and using the monotone
convergence theorem, we may assume that $\Phi $ is bounded. Moreover, by
convolving $\Phi $ with a positive kernel, we preserve coordinatewise
monotonicity of $\Phi $ and we may assume that $\Phi $ is $\mathcal{C}^{1}$.
As $\Phi $ is taken to be bounded, choosing for instance a Gaussian kernel,
it is easily seen that we can ensure that $\nabla \Phi $ is uniformly
bounded on $\mathbb{R}^{2}$. Indeed, if 
\begin{equation*}
\Psi _{\sigma ^{2}}\left( a,b\right) =\int_{\mathbb{R}^{2}}\Phi (x,y)\frac{1%
}{2\pi \sigma ^{2}}e^{-\left\Vert \left( a,b\right) -\left( x,y\right)
\right\Vert ^{2}/2\sigma ^{2}}dxdy\text{ ,}
\end{equation*}%
then 
\begin{equation*}
\nabla \Psi _{\sigma ^{2}}\left( a,b\right) =-\int_{\mathbb{R}^{2}}\Phi (x,y)%
\frac{\left\Vert \left( a,b\right) -\left( x,y\right) \right\Vert }{2\pi
\sigma ^{4}}e^{-\left\Vert \left( a,b\right) -\left( x,y\right) \right\Vert
^{2}/2\sigma ^{2}}dxdy\text{ ,}
\end{equation*}%
which is uniformly bounded in $\left( a,b\right) $ whenever $\Phi $ is
bounded. Notice also that by Lemma \ref{lemma_approx},\ it suffices to prove
the result for strictly (or strongly) log-concave variables that have $%
\mathcal{C}^{\infty }$ densities and finite Fisher information. We write%
\begin{equation*}
N\left( z\right) =\int_{\mathbb{R}}f_{X}\left( z-y\right) f_{Y}\left(
y\right) dy
\end{equation*}%
and 
\begin{equation*}
g(z)=\int_{\mathbb{R}}\Phi \left( z-y,y\right) \frac{f_{X}\left( z-y\right)
f_{Y}\left( y\right) }{N\left( z\right) }dy\text{ ,}
\end{equation*}%
with $f_{X}=\exp \left( -\varphi _{X}\right) $ and $f_{Y}=\exp \left(
-\varphi _{Y}\right) $ the respective strictly log-concave densities of $X$
and $Y$. We note $\mu _{X}$ and $\mu _{Y}$ respectively the distributions of 
$X$ and $Y$. Since $\varphi _{X}^{\prime }$ is $L_{2}\left( \mu _{X}\right) $
(which means that $\mu _{X}$ has finite Fisher information) and $f_{Y}$ is
bounded (see Theorem \ref{theorem_exp_tails}), we get that $f_{X}^{\prime
}\left( z-y\right) f_{Y}\left( y\right) =-\varphi _{X}^{\prime }\left(
z-y\right) f_{X}\left( z-y\right) f_{Y}\left( y\right) $ is integrable and
so $N$ is differentiable with gradient given by%
\begin{equation*}
N^{\prime }\left( z\right) =-\int_{\mathbb{R}}\varphi _{X}^{\prime }\left(
z-y\right) f_{X}\left( z-y\right) f_{Y}\left( y\right) dy\text{ .}
\end{equation*}%
By differentiating with respect to $z$ inside the integral defining $g$ we
get 
\begin{eqnarray}
\lefteqn{\frac{d}{dz}\left( \Phi \left( z-y,y\right) \frac{f_{X}\left( z-y\right)
              f_{Y}\left( y\right) }{\int_{\mathbb{R}}f_{X}\left( z-y^{\prime }\right)
              f_{Y}\left( y^{\prime }\right) dy^{\prime }}\right)}
             \label{first-derivative_inside} \\
&=&  \left( \partial _{1}\Phi \right) 
         \left( z-y,y\right) \frac{f_{X}\left( z-y\right) f_{Y}\left( y\right) }
                                          {N\left( z\right) }-\Phi \left( z-y,y\right) 
                                          \varphi _{X}^{\prime }\left( z-y\right) 
                              \frac{f_{X}\left( z-y\right) f_{Y}\left( y\right) }{N\left( z\right) }  \notag \\
&& \ + \ \Phi \left( z-y,y\right) \frac{f_{X}\left( z-y\right) f_{Y}\left( y\right) }
                                                     {N\left( z\right) }\int_{\mathbb{R}}\varphi _{X}^{\prime }
                    \left( z-y\right) \frac{f_{X}\left( z-y\right) f_{Y}\left( y\right) }
                                                  {N\left( z\right) }dy\text{ .}  \notag
\end{eqnarray}%
We thus see that the quantity in (\ref{first-derivative_inside}) is
integrable (with respect to Lebesgue measure) and we get
\begin{equation}
g^{\prime }\left( z\right) =\mathbb{E}\left[ \left( \partial _{1}\Phi
\right) \left( X,Y\right) \left\vert X+Y=z\right. \right] -\cov\left[ \Phi
\left( X,Y\right) ,\varphi _{X}^{\prime }\left( X\right) \left\vert
X+Y=z\right. \right] \text{ .}  \label{derivative_I}
\end{equation}%
Now, by symmetrization we have 
\begin{eqnarray*}
\lefteqn{\cov\left[ \Phi \left( X,Y\right) ,\varphi _{X}^{\prime }\left(
X\right) \left\vert X+Y=z\right. \right] } \\
&=&\mathbb{E}\left[ \left( \Phi \left( X,Y\right) -\Phi \left( \tilde{X},%
\tilde{Y}\right) \right) \left( \varphi _{X}^{\prime }\left( X\right)
-\varphi _{X}^{\prime }\left( \tilde{X}\right) \right) \mathbf{1}_{\left\{
X\geq \tilde{X}\right\} }\left\vert X+Y=z,\tilde{X}+\tilde{Y}=z\right. %
\right] \\
&=&\mathbb{E}\left[ \left( \int_{\tilde{X}}^{X}\left( \partial _{1}\Phi
-\partial _{2}\Phi \right) \left( u,z-u\right) du\right) \underset{\geq 0}{%
\underbrace{\left( \varphi _{X}^{\prime }\left( X\right) -\varphi
_{X}^{\prime }\left( \tilde{X}\right) \right) }}\mathbf{1}_{\left\{ X\geq 
\tilde{X}\right\} }\left\vert X+Y=z,\tilde{X}+\tilde{Y}=z\right. \right] \\
&\leq &\mathbb{E}\left[ \left( \int_{\tilde{X}}^{X}\left( \partial _{1}\Phi
\right) \left( u,z-u\right) du\right) \left( \varphi _{X}^{\prime }\left(
X\right) -\varphi _{X}^{\prime }\left( \tilde{X}\right) \right) \mathbf{1}%
_{\left\{ X\geq \tilde{X}\right\} }\left\vert X+Y=z,\tilde{X}+\tilde{Y}%
=z\right. \right] \\
&=&\cov\left[ \Phi _{1}\left( X\right) ,\varphi _{X}^{\prime }\left(
X\right) \left\vert X+Y=z\right. \right] \text{ ,}
\end{eqnarray*}%
where $\Phi _{1}\left( x\right) =\int_{0}^{x}\left( \partial _{1}\Phi
\right) \left( u,z-u\right) du$. We denote $\eta $ the distribution of $X$
given $X+Y=z$. The measure $\eta $ has density $h_{z}\left( x\right)
=N^{-1}\left( z\right) f_{X}\left( x\right) f_{Y}\left( z-x\right) ,$ $y\in 
\mathbb{R}$. Notice that $h_{z}$ is strictly log-concave on $\mathbb{R}$ and
that for all $x\in \mathbb{R}$,
\begin{equation*}
\left( -\log h_{z}\right) ^{\prime \prime }\left( x\right) 
= \varphi_{X}^{\prime \prime }\left( x\right) +\varphi _{Y}^{\prime \prime }\left( z-x\right) \text{ .}
\end{equation*}
Now we are able to use the asymmetric Brascamp and Lieb inequality
of \cite{Menz-Otto:2013} (Lemma 2.11, page 2190, with their $\delta \psi
\equiv 0$ so their $\psi =\psi _{c}$ with $\psi ^{\prime \prime }>0$) or 
\cite{CarlenCordero-ErausquinLieb} ((1.2), page 2); 
see Proposition~\ref{prop:MenzOtto} below. This yields
\begin{eqnarray*}
\lefteqn{\cov\left[ \Phi _{1}\left( X\right) ,\varphi _{X}^{\prime }\left(
X\right) \left\vert X+Y=z\right. \right] } \\
&=&\int_{\mathbb{R}}\left( \Phi _{1}\left( x\right) -\mathbb{E}\left[ \Phi
_{1}\left( X,Y\right) \left\vert X+Y=z\right. \right] \right) \left( \varphi
_{X}^{\prime }\left( x\right) -\mathbb{E}\left[ \varphi _{X}^{\prime }\left(
X\right) \left\vert X+Y=z\right. \right] \right) h_{z}\left( x\right) dx \\
&\leq &\sup_{x\in \mathbb{R}}\left\{ \frac{\varphi _{X}^{\prime \prime
}\left( x\right) }{\left( -\log h_{z}\right) ^{\prime \prime }\left(
x\right) }\right\} \int_{\mathbb{R}}\Phi _{1}^{\prime }\left( x\right)
h_{z}\left( x\right) dx \\
&=&\sup_{x\in \mathbb{R}}\left\{ \frac{\varphi _{X}^{\prime \prime }\left(
x\right) }{\varphi _{X}^{\prime \prime }\left( x\right) +\varphi
_{Y}^{\prime \prime }\left( z-x\right) }\right\} \mathbb{E}\left[ \left(
\partial _{1}\Phi \right) \left( X,Y\right) \left\vert X+Y=z\right. \right]
\\
&\leq &\mathbb{E}\left[ \left( \partial _{1}\Phi \right) \left( X,Y\right)
\left\vert X+Y=z\right. \right] \text{ .}
\end{eqnarray*}%
Using the latter bound in (\ref{derivative_I}) then gives the result.
\end{proof}

%----------------------------------------------------------

\section{Preservation of log-concavity and strong log-concavity under
convolution in $\mathbb{R}^d$ via Brascamp-Lieb inequalities and towards a
proof via scores}

\label{sec:StrongLogConcavityPreservMultivariateCase}

%----------------------------------------------------------

In Sections~\ref{sec:EfronsTheoremOneDimension} and~\ref{Sec:DiscreteLogConc},
we used Efron's monotonicity theorem~\ref{EfronMonotonicityThm} to give
alternative proofs of the preservation of log-concavity and strong
log-concavity under convolution in the cases of continuous or discrete
random variables on $\mathbb{R}$ or $\mathbb{Z}$ respectively. In the former
case, we also used asymmetric Brascamp-Lieb inequalities to give a new proof
of Efron's monotonicity theorem. In this section we look at preservation of
log-concavity and strong log-concavity under convolution in $\mathbb{R}^{d}$
via:\newline
(a) the variance inequality due to \cite{MR0450480};\newline
(b) scores and potential (partial) generalizations of Efron's monotonicity
theorem to $\mathbb{R}^{d}$.\newline

While point (a) gives a complete answer (Section 
\ref{ssec:BrascampLiebPfSLCpreservedByConv}), the aim of point (b) is to give an
interesting link between preservation of (strong) log-concavity in 
$\mathbb{R}^{d}$ and a (guessed) monotonicity property in $\mathbb{R}^{d}$ 
(Section \ref{ssection:SLC_Efron_multidim}). This latter property would be a partial
generalization of Efron's monotonicity theorem to the multi-dimensional case
and further investigations remain to be accomplished in order to prove such
a result.

We refer to Section~\ref{sec:AppendixA}\  (Appendix A) for further comments about the
Brascamp-Lieb inequalities and related issues, as well as a recall of
various functional inequalities.

\subsection{Strong log-concavity is preserved by convolution (again): proof
via second derivatives and a Brascamp-Lieb inequality}

\label{ssec:BrascampLiebPfSLCpreservedByConv}

We begin with a different proof of the version of Theorem~\ref%
{StrongLogConPreservOne}(b) corresponding to our first definition of strong
log-concavity, Definition~\ref{strongLogConcaveDefn1}, which proceeds via
the Brascamp-Lieb variance inequality as given in part (a) of Proposition~%
\ref{BLInequalityPlus}:

\begin{proposition}
\label{prop:MultSLCpreservedByConvViaBL} If $X \sim p \in SLC_1 (\sigma^2,
d) $ and $Y\sim q \in SLC_1 (\tau^2,d)$ are independent, then $Z=X+Y \sim
p\star q \in SLC_1( \sigma^2 + \tau^2, d)$.
\end{proposition}

\begin{proof}
Now $p_{Z}=p_{X+Y}=p\star q$ is given by 
\begin{equation}
p_{Z}(z)=\int p(x)q(z-x)dx=\int p(z-y)q(y)dy.  \label{form_conv}
\end{equation}%
Now $p=\exp (-\varphi _{p})$ and $q=\exp (-\varphi _{q})$ where we may
assume (by (b) of Proposition \ref{Lemma_convol_Gauss_multidim})
that $\varphi _{p},\varphi _{q}\in C^{2}$ and that $p$ and $q$ have finite
Fisher information. Then, by Proposition~\ref{EquivalencesStrongLogConTry3}, 
\begin{equation*}
{\nabla ^{2}\mathstrut }(\varphi _{p})(x)\geq \frac{1}{\sigma ^{2}}I,\ \ %
\mbox{and}\ \ {\nabla ^{2}\mathstrut }(\varphi _{q})(x)\geq \frac{1}{\tau
^{2}}I.
\end{equation*}%
As we can interchange differentiation and integration in (\ref{form_conv})
(see for instance the detailed arguments for a similar situation in the
proof of Proposition \ref{prop_efron_spe_dim_1}), we find that 
\begin{equation*}
\nabla (-\log p_{Z})(z)=-\frac{\nabla p_{Z}}{p_{Z}}\left( z\right)
=E\{\nabla \varphi _{q}(Y)|X+Y=z\}=E\{\nabla \varphi _{p}(X)|X+Y=z\}.
\end{equation*}%
Then 
\begin{eqnarray*}
\lefteqn{\MyHess(-\log p_{Z})(z)} \\
&=&\nabla \left\{ E[q(z-X)\nabla (-\log q)(z-X)]\cdot \frac{1}{p_{Z}(z)}%
\right\}  \\
&=&E\{-\nabla q(Y)(\nabla \log q(Y))^{T}|X+Y=z\} \\
&&\qquad +\ E\{{\nabla ^{2}\mathstrut }(-\log q)(Y)|X+Y=z\}+\left( E\{\nabla
\log q(Y)|X+Y=z\}\right) ^{\otimes 2} \\
&=&-Var(\nabla \varphi _{q}(Y)|X+Y=z)+E\{{\nabla ^{2}\mathstrut }\,\varphi
_{q}(Y)|X+Y=z\} \\
&=&-Var(\nabla \varphi _{p}(X)|X+Y=z)+E\{{\nabla ^{2}\mathstrut }\,\varphi
_{p}(X)|X+Y=z\}.
\end{eqnarray*}%
Now we apply \cite{MR0450480} Theorem 4.1 (see Proposition~\ref%
{BLInequalityPlus}(a)) with 
\begin{eqnarray}
&&h(x)=\nabla _{z}\varphi _{q}(z-x), \\
&&F(x)=p(x)q(z-x),  \label{NotationConvSpecialCase}
\end{eqnarray}%
to obtain 
\begin{eqnarray*}
\lefteqn{Var(\nabla _{z}\varphi _{q}(Y)|Z+Y=z)} \\
&\leq &\int_{\mathbb{R}^{d}}{\nabla ^{2}\mathstrut }\,\varphi
_{g}(z-x)\left\{ {\nabla ^{2}\mathstrut }\,\varphi _{p}(x)+{\nabla
^{2}\mathstrut }\,\varphi _{q}(z-x)\right\} ^{-1}\cdot {\nabla
^{2}\mathstrut }\,\varphi _{q}(z-x)\frac{F(x)}{\int_{\mathbb{R}%
^{d}}F(x^{\prime })dx^{\prime }}dx.
\end{eqnarray*}%
This in turn yields 
\begin{eqnarray}
\lefteqn{\MyHess(-\log p_{Z})(z)}  \label{BCLspecialCase1} \\
&\geq &E\left\{ {\nabla ^{2}\mathstrut }\,\varphi _{q}(Y)-{\nabla
^{2}\mathstrut }\,\varphi _{q}(Y)\big [{\nabla ^{2}\mathstrut }\,\varphi
_{p}(X)+{\nabla ^{2}\mathstrut }\,\varphi _{q}(Y)\big ]^{-1}{\nabla
^{2}\mathstrut }\,\varphi _{q}(Y)|X+Y=z\right\} .  \notag
\end{eqnarray}%
By symmetry between $X$ and $Y$ we also have 
\begin{eqnarray}
\lefteqn{\MyHess(-\log p_{Z})(z)}  \label{BCLspecialCase2} \\
&\geq &E\left\{ {\nabla ^{2}\mathstrut }\,\varphi _{p}(X)-{\nabla
^{2}\mathstrut }\,\varphi _{p}(X)\big [{\nabla ^{2}\mathstrut }\,\varphi
_{p}(X)+{\nabla ^{2}\mathstrut }\,\varphi _{q}(Y)\big ]^{-1}{\nabla
^{2}\mathstrut }\,\varphi _{p}(X)|X+Y=z\right\} .  \notag
\end{eqnarray}%
In proving the inequalities in the last two displays we have in fact
reproved Theorem 4.2 of \cite{MR0450480} in our special case given by (\ref%
{NotationConvSpecialCase}). Indeed, Inequality (4.7) of Theorem 4.2 in \cite%
{MR0450480} applied to our special case is the first of the two inequalities
displayed above.

Now we combine (\ref{BCLspecialCase1}) and (\ref{BCLspecialCase2}). We set 
\begin{eqnarray*}
&& \alpha \equiv \frac{\sigma^2}{\sigma^2+\tau^2}, \ \ \ \beta \equiv
1-\alpha = \frac{\tau^2}{\sigma^2 + \tau^2} , \\
&& A \equiv \big [ {\nabla^2 \mathstrut} \, \varphi_p (X) + {\nabla^2
\mathstrut} \, \varphi_q (Y) \big ]^{-1} , \\
&& s = s(X) \equiv {\nabla^2 \mathstrut} \, \varphi_p (X), \ \ \ t = t(X)
\equiv {\nabla^2 \mathstrut} \, \varphi_q (Y) .
\end{eqnarray*}
We get from (\ref{BCLspecialCase1}) and (\ref{BCLspecialCase2}): 
\begin{eqnarray*}
\lefteqn{\MyHess (-\log p_{Z}) (z) }  \label{BCLspecialCaseSymmetricOne} \\
& \ge & E \left \{ \alpha s + \beta t - \alpha s A s - \beta t A t \big | %
X+Y = z \right \} \\
& = & E \left \{ (\alpha s + \beta t) A (s+t) - \alpha s A s - \beta t A t %
\big | X+Y = z \right \} \\
&& \qquad \mbox{since} \ A (s+t) = I \equiv \mbox{identity} \\
& = & E \left \{ \alpha s A t + \beta t A s \big | X+Y = z \right \} .
\end{eqnarray*}
Now 
\begin{eqnarray*}
\alpha s A t & = & \frac{\sigma^2}{\sigma^2 + \tau^2} {\nabla^2 \mathstrut}
\, \varphi_p \big [ {\nabla^2 \mathstrut} \, \varphi_p (X) + {\nabla^2
\mathstrut} \, \varphi_q (Y) \big ]^{-1} {\nabla^2 \mathstrut} \, \varphi_q
(Y) \\
& = & \frac{\sigma^2}{\sigma^2 + \tau^2} \big [ ({\nabla^2 \mathstrut} \,
\varphi_p )^{-1} (X) + ( {\nabla^2 \mathstrut} \, \varphi_q)^{-1} (Y) \big ]%
^{-1} .
\end{eqnarray*}
By symmetry 
\begin{eqnarray*}
\beta t A s & = & \frac{\tau^2}{\sigma^2+\tau^2} \big [ ({\nabla^2 \mathstrut%
} \, \varphi_p )^{-1} (X) + ( {\nabla^2 \mathstrut} \, \varphi_q)^{-1} (Y) %
\big ]^{-1}
\end{eqnarray*}
and we therefore conclude that 
\begin{eqnarray*}
\lefteqn{\MyHess (-\log p_{Z}) (z) }  \label{BCLspecialCaseSymmetricTwo} \\
& \ge & \frac{\sigma^2+\tau^2}{\sigma^2 + \tau^2} E \left \{ \big [ ({%
\nabla^2 \mathstrut} \, \varphi_p )^{-1} (X) + ( {\nabla^2 \mathstrut} \,
\varphi_q)^{-1} (Y) \big ]^{-1} \big | X+Y = z \right \} \\
& \ge & \frac{1}{\sigma^2 + \tau^2} I .
\end{eqnarray*}
Note that the resulting inequality 
\begin{eqnarray*}
\MyHess (-\log p_{Z}) (z)  % \label{BCLspecialCaseSymmetricThree} 
 \ge  E \left \{ \big [ ({\nabla^2 \mathstrut} \, \varphi_p )^{-1} (X) + ( 
{\nabla^2 \mathstrut} \, \varphi_q)^{-1} (Y) \big ]^{-1} \big | X+Y = z \right \}
\end{eqnarray*}
also gives the right lower bound for convolution of strongly log-concave
densities in the definition of $SLC_2( \mu, \Sigma,d)$, namely 
\begin{eqnarray*}
{\nabla^2 \mathstrut} (-\log p_{Z}) (z) \ge ( \Sigma_X + \Sigma_Y )^{-1} .
\label{BCLspecialCaseSymmetricThree}
\end{eqnarray*}
\end{proof}

\subsection{Strong log-concavity is preserved by convolution (again):
towards a proof via scores and a multivariate Efron inequality\label%
{ssection:SLC_Efron_multidim}}

We saw in the previous sections that Efron's monotonicity theorem allows to
prove stability under convolution for (strongly) log-concave measures on $%
\mathbb{R}$. However, the stability holds also in $\mathbb{R}^{d}$, $d>1$.
This gives rise to the two following natural questions: does a
generalization of Efron's theorem in higher dimensions exist? 
Does it allow recovery stability under convolution for 
log-concave measures in $\mathbb{R}^{d}$?

Let us begin with a projection formula for scores in dimension $d$.

\begin{lemma}
\label{lem:ScoreProjectionMultivariate} %\noindent \textbf{Lemma 1.} 
(Projection) Suppose that $X$ and $Y$ are $d-$dimensional independent random
vectors with log-concave densities $p_{X}$ and $q_{Y}$ respectively on $%
\mathbb{R}^{d}$. Then $\nabla \varphi_{X+Y}$ and $\rho _{X+Y}:\mathbb{R}%
^{d}\rightarrow \mathbb{R}^{d}$ are given by 
\begin{eqnarray*}
\nabla \varphi_{X+Y} (z) = E \left \{ \lambda \nabla \varphi_X (X) +
(1-\lambda) \nabla \varphi_{Y} (Y) | X+Y = z \right \}
\end{eqnarray*}
for each $\lambda \in [0,1]$, and, if $p_X \in SLC_1 (\sigma^2, d)$ and $p_Y
\in SLC_1 (\tau^2, d)$, then 
\begin{eqnarray*}
\rho _{X+Y}(z) = E\left\{ \frac{\sigma ^{2}}{\sigma ^{2}+\tau ^{2}}%
\rho_{X}(X) +\frac{\tau ^{2}}{\sigma ^{2}+\tau ^{2}}\rho _{Y}(Y)\bigg | %
X+Y=z\right\} .
\end{eqnarray*}
\end{lemma}

\smallskip

\begin{proof}
This can be proved just as in the one-dimensional case, much as in \cite%
{MR659464}, but proceeding coordinate by coordinate. %\hfill $\Box$
\end{proof}

Since we know from Propositions~\ref{EquivalencesLogCon} and~\ref%
{EquivalencesStrongLogConTry3} that the scores $\nabla \varphi_X$ and $%
\nabla \varphi_Y$ and the relative scores $\rho_X$ and $\rho_Y$ are
multivariate monotone, the projection Lemma~\ref%
{lem:ScoreProjectionMultivariate} suggests that proofs of preservation of
multivariate log-concavity and strong log-concavity might %(possibly) 
be possible via a multivariate generalization of Efron's monotonicity Theorem~
\ref{EfronMonotonicityThm} to $d\ge 2$ along the following lines: 
%\noindent \textbf{Lemma 2.} (multivariate version of Efron's theorem).
Suppose that $\Phi :(\mathbb{R}^{d})^{n}\rightarrow \mathbb{R}^{d}$ is
coordinatewise multivariate monotone: for each fixed $j\in \{1,\ldots ,n\}$
the function $\Phi _{j}:\mathbb{R}^{d}\rightarrow \mathbb{R}^{d}$ defined by 
\begin{equation*}
\Phi _{j}(x)=\Phi (x_{1},\ldots ,x_{j-1},x,x_{j+1},\ldots ,x_{n})
\end{equation*}%
is multivariate monotone: that is 
\begin{equation*}
\langle \Phi _{j}(x_{1})-\Phi _{j}(x_{2}),x_{1}-x_{2}\rangle \geq 0\ \ %
\mbox{for all}\ \ x_{1},x_{2}\in \mathbb{R}^{d}.
\end{equation*}%
If $X_{1},\ldots ,X_{m}$ are independent with $X_{j}\sim f_{j}$ log-concave
on $\mathbb{R}^{d}$, then it might seem natural to conjecture that the
function $g$ defined by 
\begin{equation*}
g(z)\equiv E\left\{ \Phi (X_{1},\ldots ,X_{n})\bigg |X_{1}+\cdots
+X_{m}=z\right\}
\end{equation*}%
is a monotone function of $z\in \mathbb{R}^{d}$: 
\begin{equation*}
\langle g(z_{1})-g(z_{2}),z_{1}-z_{2}\rangle \geq 0\ \ \mbox{for all}\ \
z_{1},z_{2}\in \mathbb{R}^{d}.
\end{equation*}
Unfortunately, this seemingly natural generalization of Efron's theorem does
not hold without further assumptions. In fact, it fails for $m=2$ and $%
\mathbf{X}_1, \mathbf{X}_2$ Gaussian with covariances $\Sigma_1$ and $%
\Sigma_2$ sufficiently different. For an explicit example see \cite%
{Saumard-Wellner:13}. \smallskip

Again, the result holds for $m$ random vectors if it holds for 
$2,\ldots ,m-1 $ random vectors. It suffices to prove the theorem for $m=2$ random
vectors. Since everything reduces to the case where $\Phi $ is a function of
two variables (either for Efron's theorem or for a multivariate
generalization), we will restrict ourselves to this situation.

Thus if we define 
\begin{equation*}
g(s)\equiv E\left\{ \Phi (X_{1},X_{2})\bigg |X_{1}+X_{2}=s\right\} \text{ },
\end{equation*}%
then we want to show that 
\begin{equation*}
\langle g(s_{1})-g(s_{2}),s_{1}-s_{2}\rangle \geq 0\ \ \mbox{for all}\ \
s_{1},\ s_{2}\in \mathbb{R}^{d}\text{ }.
\end{equation*}

Finally, our approach to Efron's monotonicity theorem in dimension $d\geq 2$
is based on the following remark.

\begin{remark}
\label{remark_link_cov_Efron_multidim}For suitable regularity of 
$\Phi :(\mathbb{R}^{d})^{2}\rightarrow \mathbb{R}^{d}$ 
and 
$\rho _{X}:\mathbb{R}^{d}\rightarrow \mathbb{R}$, we have 
\begin{eqnarray*}
\left( \nabla g\right) \left( z\right) 
& = &E\left\{ \left( \nabla _{1}\Phi \right) (X,Y)\bigg |X+Y=z\right\} \\
&& \ \ -\ \cov\left\{ \Phi (X,Y),\rho _{X}\left( X\right) \bigg |X+Y=z\right\} \ \ \left( \in \mathbb{R}^{d\times d}\right) 
\text{ .}
\end{eqnarray*}
Recall that 
$\nabla _{1}\Phi \equiv \nabla \Phi _{1}:(\mathbb{R}^{d})^{2}\rightarrow \mathbb{R}^{d\times d}$. 
Furthermore, the matrix $\left( \nabla g\right) \left( z\right) $ is positive semi-definite if for
all $a\in \mathbb{R}^{d}$, $a^{T}\nabla g(z)a^{T}\geq 0$, which leads to
leads to asking if the following covariance inequality holds: 
\begin{equation}
\cov\left\{ a^{T}\Phi (X,Y),\rho _{X}^{T}(X)a\bigg |X+Y=z\right\} \leq
E\left\{ a^{T}\nabla _{1}\Phi (X,Y)a\bigg |X+Y=z\right\} \text{?}
\label{cov_ineq_Efron_multidim}
\end{equation}%
Covariance inequality (\ref{cov_ineq_Efron_multidim}) would imply a
multivariate generalization of Efron's theorem (under sufficient regularity).
\end{remark}

%\subsection{Around Efron's monotonicity theorem for some particular functions}

%----------------------------------------------------------
\section{Peakedness and log-concavity}
\label{sec:peakedness}
%----------------------------------------------------------

Here is a summary of the results of \cite{MR0187269}, %Proschan 
\cite{MR927147}, %Olkin and Tong (1988) and 
\cite{MR2095937}, %Harg\'e (2004)
and \cite{MR994278}. %Kelley ?

First \cite{MR2095937}. %Harg\'e (2004).  
Let $f$ be log-concave, and let $g$ be convex. Then if $X \sim N_d (\mu,
\Sigma) \equiv \gamma$, 
\begin{eqnarray}
E \{ g(X+\mu-\nu) f(X) \} \le E f(X) \cdot Eg(X)
\label{HargeInequalityConvexGfirstForm}
\end{eqnarray}
where $\mu = E(X)$, $\nu = E( X f(X) )/E ( f(X))$. Assuming that $f \ge 0$,
and writing $\tilde{f} d \gamma \equiv f d \gamma / \int f d \gamma $, $%
\tilde{g}(x - \mu) \equiv g(x)$, and $\tilde{X} \sim \tilde{f} d \gamma$ so
that $\tilde{X}$ is strongly log-concave, this can be rewritten as 
\begin{eqnarray}
E \tilde{g} (\tilde{X} - E(\tilde{X} )) \le E \tilde{g}(X- \mu) .
\label{HargeInequalityConvexGsecondForm}
\end{eqnarray}
In particular, for $\tilde{g}(x) = | x |^r$ with $r \ge 1$, 
\begin{equation*}
E | \tilde{X} - \tilde{\mu} |^r \le E | X - \mu |^r ,
\end{equation*}
and for $\tilde{g}(x) = | a^T x |^r$ with $a \in \mathbb{R}^d$, $r\ge 1$, 
\begin{equation*}
E | a^T(\tilde{X} - \tilde{\mu} ) |^r \le E | a^T ( X- \mu ) |^r,
\end{equation*}
which is Theorem 5.1 of \cite{MR0450480}. Writing (\ref%
{HargeInequalityConvexGfirstForm}) as (\ref{HargeInequalityConvexGsecondForm}%
) makes it seem more related to the ``peakedness'' results of \cite{MR927147}
to which we now turn. %Olkin and Tong (1988).
%Note that if $g = -1_A$ with $A \in {\cal A}_n$ as defined below, then 
%$g$ is convex
\medskip

An $n$-dimensional random vector $Y$ is said to be \textsl{more peaked} than
a vector $X$ if they have densities and if 
\begin{equation*}
P(Y \in A) \ge P(X \in A)
\end{equation*}
holds for all $A \in \mathcal{A}_n$, the class of compact, convex, symmetric
(about the origin) Borel sets in $\mathbb{R}^n$. When this holds we will
write $Y \overset{p}{\ge} X$. A vector $a $ \textsl{majorizes} the vector $b$
(and we write $a \succ b$) if $\sum_{i=1}^k b_{[i]} \le \sum_{i=1}^k a_{[i]}$
for $k=1, \ldots , n-1$ and $\sum_{i=1}^n b_{[i]} = \sum_{i=1}^n a_{[i]}$
where $a_{[1]} \ge a_{[2]} \ge \cdots \ge a_{[n]}$ and similarly for $b$.
(In particular $b = (1, \ldots , 1)/n \prec (1,0, \ldots , 0) = a$.) \medskip

%\par\noindent

\begin{proposition}
(Sherman, 1955; see \cite{MR927147}) \label%
{PeakednessPlusSymmetryPreservedByConvolution} %{\bf Proposition 1:}
Suppose that $f_1, f_2, g_1, g_2$ are log-concave densities on $\mathbb{R}^n$
which are symmetric about $0$. Suppose that $X_j \sim f_j$ and $Y_j \sim g_j 
$ for $j=1,2$ are independent. Suppose that $Y_1 \overset{p}{\ge} X_1$ and $%
Y_2 \overset{p}{\ge} X_2$. Then $Y_1 + Y_2 \overset{p}{\ge} X_1 + X_2$.
\end{proposition}

%\medskip 

%\par\no indent

\begin{proposition}
\label{PeakednessPreservedWeightedSumsOnRealLine} %{\bf Proposition 2:}  
If $X_1, \ldots X_n$ are independent random variables with log-concave
densities symmetric about $0$, and $Y_1, \ldots , Y_n$ are independent with
log-concave densities symmetric about $0$, and $Y_j \overset{p}{>} X_j$ for $%
j = 1, \ldots , n$, then 
\begin{equation*}
\sum_{j=1}^n c_j Y_j \overset{p}{\ge} \sum_{j=1}^n c_j X_j
\end{equation*}
for all real numbers $\{ c_j \}$.
\end{proposition}

%\par\noindent

\begin{proposition}
\label{PeakednessPreservedUnderConvergenceInDistr} %{\bf Proposition 3:}  
If $\{ X_m \}$ and $\{ Y_m \}$ are two sequences of $n-$dimensional random
vectors with $Y_m \overset{p}{>} X_m$ for each $m$ and $X_m \rightarrow_d X$%
, $Y_m \rightarrow_d Y$, then $Y \overset{p}{>} X$.
\end{proposition}

%\medskip

% \par\noindent

\begin{proposition}
$Y \overset{p}{\ge} X$ if and only if $C Y \overset{p}{\ge } CX$ for all $%
k\times n$ matrices $C$ with $k\le n$.
\end{proposition}

\medskip

\begin{proposition}
(\cite{MR0187269}) \label{ProschanMajorizationAndIIDLogConcaveOnR} Suppose
that $Z_1, \ldots , Z_n $ are i.i.d. random variables with log-concave
density symmetric about zero. Then if $a , b \in \mathbb{R}_+^n$ with $a
\succ b$ ($a$ majorizes $b$), then 
\begin{equation*}
\sum_{j=1}^n b_j Z_j \overset{p}{\ge } \sum_{j=1}^n a_j Z_j \ \ \ \mbox{in}
\ \ \mathbb{R}
\end{equation*}
\end{proposition}

\begin{proposition}
(\cite{MR927147}) \label{OlkinTongMajorizationAndIIDLogConcaveOnRd} Suppose
that $Z_1, \ldots , Z_n $ are i.i.d. $d-$dimensional random vectors with
log-concave density symmetric about zero. Then 
if $a_j , b_j \in \mathbb{R}^1 $ with $a \succ b$ ($a$ majorizes $b$), then 
\begin{equation*}
\sum_{j=1}^n b_j Z_j \overset{p}{>} \sum_{j=1}^n a_j Z_j \ \ \ \mbox{in} \ \ 
\mathbb{R}^d .
\end{equation*}
\end{proposition}

Now let 
$\mathcal{K}_n \equiv \{ x \in \mathbb{R}^n : \ x_1 \le x_2 \le \ldots \le x_n \}$. 
For any 
$y \in \mathbb{R}^n$, let $\hat{y} = (\hat{y}_1, \ldots, \hat{y}_n )$ 
denote the projection of $y$ onto $\mathcal{K}_n$. 
Thus $| y - \hat{y} |^2 = \mbox{min}_{x \in \mathcal{K}} | y - x |^2$.

\begin{proposition}
\label{KellyOrderedMeansPeakednessForGaussian} (\cite{MR994278}). Suppose
that $\underline{Y} = (Y_1, \ldots , Y_n)$ where $Y_j \sim N(\mu_j, \sigma^2) $ are independent and 
$\mu_1\le \mu_2 \le \ldots \le \mu_n$. Thus $\underline{\mu} \in {\mathcal K}_n$ 
and $\underline{\hat{\mu}} \equiv \underline{\hat{Y}} \in {\mathcal K}_n$.
Then $\hat{\mu}_k - \mu_k \overset{p}{\ge} Y_k - \mu_k$ for each $k \in \{ 1, \ldots , n \}$; i.e. 
\begin{eqnarray*}
P( | \hat{\mu}_k - \mu_k | \le t ) \ge P( | Y_k - \mu_k | \le t ) \ \ \ 
\mbox{for all} \ \ t>0, \ \ k \in \{ 1, \ldots, n \}.
\end{eqnarray*}
\end{proposition}

% \bigskip

%\noindent \textbf{Question 3:} \ \ Does Kelly's Proposition~\ref%
%{KellyOrderedMeansPeakednessForGaussian} continue to hold if the normal
%distributions of the $Y_i$'s is replaced some other centrally-symmetric
%log-concave distribution, for example Chernoff's distribution (see \cite%
%{Balabdaoui-Wellner:13}).

%----------------------------------------------------------
\section{Some open problems and further connections with log-concavity}
\label{sec:OpenProblemsAndFurtherConnections}

%----------------------------------------------------------

\subsection{Two questions}

\noindent \textbf{Question 1:} \ \ 
Does Kelly's Proposition~\ref{KellyOrderedMeansPeakednessForGaussian} 
continue to hold if the normal
distributions of the $Y_i$'s is replaced some other centrally-symmetric
log-concave distribution, for example Chernoff's distribution 
(see \cite{Balabdaoui-Wellner:13})?
\medskip 

\noindent \textbf{Question 2:} \ \ 
\cite{Balabdaoui-Wellner:13} show that Chernoff's distribution is
log-concave. Is it strongly log-concave? A proof would probably give a way
of proving strong log-concavity for %showing that 
a large class of functions of the form $f(x)=g(x)g(-x)$ where $g\in
PF_{\infty }$ is the density of the sum $\sum_{j=1}^{\infty }(Y_{j}-\mu
_{j}) $ where $Y_{j}$'s are independent exponential random variables with
means $\mu _{j}$ satisfying $\sum_{j=1}^{\infty }\mu _{j}=\infty $ and $%
\sum_{j=1}^{\infty }\mu _{j}^{2}<\infty $.

\subsection{Cross-connections with the families of hyperbolically monotone
densities}

A theory of hyperbolically monotone and completely monotone densities has
been developed by \cite{MR1224674}, %Bondesson 1992
\cite{MR1481175}.

\begin{definition}
A density $f$ on $\mathbb{R}^+$ is hyperbolically completely monotone if $%
H(w) \equiv f(uv) f(u/v)$ is a completely monotone function of $w =
(v+1/v)/2 $. A density $f$ on $\mathbb{R}^+$ is hyperbolically monotone of
order $k$, or $f \in HM_k$ if the function $H$ satisfies $(-1)^j H^{(j)} (w)
\ge 0$ for $j = 0, \ldots , k-1$ and $(-1)^{k-1} H^{(k-1)} (w) $ is
right-continuous and decreasing.
\end{definition}

For example, the exponential density $f(x) = e^{-x} 1_{(0,\infty)} (x)$ is
hyperbolically completely monotone, while the half-normal density $f(x) = 
\sqrt{2/\pi} \exp (- x^2/2) 1_{(0,\infty)(x) }$ is $HM_1$ but not $HM_2$.

\cite{MR1481175} page 305 shows that if $X\sim f\in HM_{1}$, then $\log
X\sim e^{x}f(e^{x})$ is log-concave. Thus $HM_{1}$ is closed under the
formation of products: if $X_{1},\ldots ,X_{n}\in HM_{1}$, then $Y\equiv
X_{1}\cdots X_{n}\in HM_{1}$.

\subsection{Suprema of Gaussian processes\label{ssec_suprema}}

\cite{MR2375638} use log-concavity of Gaussian measures to show that the
supremum of an arbitrary non degenerate Gaussian process has a continuous
and strictly increasing distribution function. This is useful for bootstrap
theory in statistics. The methods used by \cite{MR2375638} originate in \cite%
{MR0388475} % C. Borell
and \cite{MR745081}; %  A Erhard
see \cite{MR1642391} chapters 1 and 4 for an exposition.

Furthermore, in relation to Example \ref{SupremumBrownianBridge}\ above, one
can wonder what is the form of the density of the maximum of a Gaussian
process in general? \cite{MR2415134}  actually gives a complete
characterization of the distribution of suprema of Gaussian processes.
Indeed, the author proves that $F$ is the distribution of the supremum of a\
general Gaussian process if and only if $\Phi ^{-1}\left( F\right) $ is
concave, where $\Phi ^{-1}$ is the inverse of the standard normal
distribution function on the real line. Interestingly, the \textquotedblleft
only if\textquotedblright\ part is a direct consequence the Brunn-Minkowski
type inequality for the standard Gaussian measure $\gamma _{d}$ on $\mathbb{R%
}^{d}$ due to \cite{MR745081}: for any $A$ and $B\in \mathbb{R}^{d}$ of
positive measure and for all $\lambda \in \left( 0,1\right) $,%
\begin{equation*}
\Phi ^{-1}\left( \gamma _{d}\left( \lambda A+\left( 1-\lambda \right)
B\right) \right) \geq \lambda \Phi ^{-1}\left( \gamma _{d}\left( A\right)
\right) +\left( 1-\lambda \right) \Phi ^{-1}\left( \gamma _{d}\left(
B\right) \right) \text{ .}
\end{equation*}

\subsection{Gaussian correlation conjecture}

The Gaussian correlation conjecture, first stated by \cite{MR0413364}, is as
follows. Let $A$ and $B$ be two symmetric convex sets. If $\mu $ is a
centered, Gaussian measure on $\mathbb{R}^{n}$, then%
\begin{equation}
\mu \left( A\cap B\right) \geq \mu \left( A\right) \mu \left( B\right) \text{
.}  \label{Gaussian_corr}
\end{equation}%
In other words, the correlation between the sets $A$ and $B$ under the
Gaussian measure $\mu $ is conjectured to be nonnegative. As the indicator
of a convex set is log-concave, the Gaussian correlation conjecture
intimately related to log-concavity.

In \cite{MR1742895}, the author gives an elegant partial answer to 
Problem (\ref{Gaussian_corr}), using semigroup techniques. The Gaussian correlation
conjecture has indeed been proved to hold when $d=2$ by \cite{MR0448705} 
%Pitt (reference??) 
and by \cite{MR1742895} when one of the sets is a symmetric ellipsoid and
the other is convex symmetric. \cite{MR1894593} gave another proof of
Harg\'{e}'s result, as a consequence of Caffarelli's Contraction Theorem
(for more on the latter theorem, see Section \ref{ssection_opt_transp}\
below). Extending Caffarelli's Contraction Theorem, \cite{MR2983070} also
extended the result of Harg\'{e} and Cordero-Erausquin, but without proving
the full Gaussian correlation conjecture.

\cite{MR1742895} gives some hints towards a complete solution of Problem (%
\ref{Gaussian_corr}). Interestingly, a sufficient property would be the
preservation of log-concavity along a particular family of semigroups. More
precisely, let $A(x)$ be a positive definite matrix for each 
$x\in \mathbb{R}^{d}$ and define 
\begin{equation*}
Lf(x)=(1/2)(\mbox{div}(A(x)^{-1}\nabla f)-(\nabla f(x)^{T}A^{-1}(x)x).
\end{equation*}%
The operator $L$ is the infinitesimal generator of an associated semigroup.
The question is: does $L$ preserve log-concavity? See \cite{MR1742895} 
%Harge 1999
and \cite{MR1863297}. % Kolesnikov (2001).  
For further connections involving the semi-group approach to correlation
inequalities, see \cite{MR1307413}, %  Bakry (1992)
\cite{MR1334608}, %Ledoux (1995)
\cite{MR2376572}, % Harge (2008)
and \cite{CattiauxGuillin:13}.

Further connections in this direction involve the theory of parabolic and
heat-type partial differential equations; see e.g. \cite{MR1033179}, 
%Keady (1990)
\cite{MR1863297}, %Kolesnikov (2001)
\cite{MR2462697}, % Andreu, F. and Caselles, V. and Maz{\'o}n (2008)
\cite{MR684757}, %Korevaaar  
\cite{MR692284}. % Korevaar (1983(b))

\subsection{Further connections with Poincar\'e, Sobolev, and log-Sobolev
inequalities}

For a very nice paper with interesting historical and expository passages,
see \cite{MR1800062}. %Bobkov and Ledoux
Among other things, these authors establish an entropic or log-Sobolev
version of the Brascamp-Lieb type inequality under a concavity assumption on 
$h^{T}\varphi ^{\prime \prime }(x)h$ for every $h$. The methods in the
latter paper build on \cite{Maurey:91}. See \cite{bakry2014analysis}\ for a
general introduction to these analytic inequalities from a Markov diffusion
operator viewpoint.%Maurey

\subsection{Further connections with entropic central limit theorems}

This subject has its beginnings in the work of \cite{Linnik:59}, %Linnik,
\cite{MR659464}, %Brown
and \cite{MR815975}, %Barron
but has interesting cross-connections with log-concavity in the more recent
papers of \cite{MR2128239}, % Barron and Johnson
\cite{MR1124273}, %Carlen and Soffer
\cite{MR1991646}, %Ball, Barthe, and Naor
\cite{MR2083473}, %Artstein, Ball, Barthe, and Naor I
and \cite{MR2128238}. %Artstein, Ball, Barthe, and Naor II
More recently still, further results have been obtained by: \cite{MR2077162}%
, %Carlen,  Lieb, and Loss 2004
\cite{MR2545242}, %Carlen, and Cordero-Erausquin  2009
and \cite{MR2644890}. %Cordero-Erausquin and Ledoux 2010 

\subsection{Connections with optimal transport and Caffarelli's contraction
theorem\label{ssection_opt_transp}}

\cite{MR2895086} give a nice survey about advances in transport
inequalities, with Section 7 devoted to strongly log-concave measures
(called measures with \textquotedblleft uniform convex
potentials\textquotedblright\ there). The theory of optimal transport is
developed in \cite{MR1964483} and \cite{MR2459454}. See also \cite{MR1127042}%
, %Cafferilli , 
\cite{MR1124980}, %Caffarelli 1992
\cite{MR1800860}, %Caffarelli 2000
and \cite{MR2983070} for results on (strongly) log-concave measures. The
latter authors extend the results of \cite{MR1800860} under a third
derivative hypothesis on the \textquotedblleft potential\textquotedblright\ $%
\varphi $.

In the following, we state the celebrated Caffarelli's Contraction Theorem (%
\cite{MR1800860}). Let us recall some related notions. A Borel map $T$ is
said to \textit{push-forward} $\mu $ onto $\nu $, for two Borel probability
measures $\mu $ and $\nu $, denoted $T_{\ast }\left( \mu \right) =\nu $, if
for all Borel sets $A$, $\nu \left( A\right) =\mu \left( T^{-1}\left(
A\right) \right) $. Then the Monge-Kantorovich problem (with respect to the
quadratic cost) is to find a map $T_{opt}$ such that%
\begin{equation*}
T_{opt}\in \arg \min_{T\text{ s.t.}T_{\ast }\left( \mu \right) =\nu }\left\{
\int_{\mathbb{R}^{d}}\left\vert T\left( x\right) -x\right\vert ^{2}d\mu
\left( x\right) \right\} \text{ .}
\end{equation*}%
The map $T_{opt}$ (when it exists) is called the Brenier map and it is $\mu $%
-a.e. unique. Moreover, \cite{MR1100809} showed that Brenier maps are
characterized to be gradients of convex functions (see also \cite{MR1369395}%
). See \cite{MR2087149} for a very nice elementary introduction to monotone
transportation. We are now able to state Caffarelli's Contraction Theorem.

\begin{theorem}[\protect\cite{MR1800860}]
Let $b\in \mathbb{R}^{d}$, $c\in \mathbb{R}$ and $V$ a convex function on 
$\mathbb{R}^{d}$. Let $A$ be a positive definite matrix in 
$\mathbb{R}^{d\times d}$ and $Q$ be the following quadratic function,
\begin{equation*}
Q\left( x\right) =\left\langle Ax,x\right\rangle +\left\langle
b,x\right\rangle +c,\text{ \ }x\in \mathbb{R}^{d}\text{ .}
\end{equation*}
Let  $\mu $ and $\nu $ denote two probability measures on $\mathbb{R}^{d}$ with
respective densities $\exp \left( -Q\right) $ and 
$\exp \left( -\left( Q+V\right) \right) $ 
with respect to Lebesgue measure. Then the Brenier map 
$T_{opt}$ pushing $\mu$ forward onto $\nu $ is a contraction:
\begin{equation*}
\left\vert T\left( x\right) -T\left( y\right) \right\vert \leq \left\vert
x-y\right\vert \text{ \ \ for all }x,y\in \mathbb{R}^{d}\text{ .}
\end{equation*}
\end{theorem}

Notice that Caffarelli's Contraction Theorem is in particular valid when 
$\mu $ is a Gaussian measure and that case, $\nu $ is a strongly log-concave
measure.

\subsection{Concentration and convex geometry}

\cite{Guedon:12} %Guedon (2012) 
gives a nice survey, explaining the connections between the Hyperplane
conjecture, the KLS conjecture, the Thin Shell conjecture, the Variance
conjecture and the Weak and Strong moments conjecture. Related papers
include \cite{MR2846382} and %Guedon and Milman (2011)
\cite{FradelizaGuedonPajor:13}. %Bobkov and Madiman (2011).  

It is well-known that concentration properties are linked the behavior of moments.
\cite{MR2857249} prove that if $\eta >0$ is log-concave
then the function 
\begin{equation*}
\bar{\lambda}_{p}=\frac{1}{\Gamma \left( p+1\right) }\mathbb{E}\left[ \eta
^{p}\right] ,\text{ \ \ \ }p\geq 0,
\end{equation*}%
is also log-concave, where $\Gamma $ is the classical Gamma function. This
is equivalent to having a so-called "reverse Lyapunov's inequality",%
\begin{equation*}
\bar{\lambda}_{a}^{b-c}\bar{\lambda}_{c}^{a-b}\leq \bar{\lambda}_{b}^{a-c},%
\text{ \ \ \ }a\geq b\geq c\geq 0.
\end{equation*}%
Also, \cite{MR2083386} proves that log-concavity of $\tilde{\lambda}%
_{p}=\mathbb{E}\left[ \left( \eta /p \right) ^{p}\right] $ holds
(this is a consequence of the Pr\'{e}kopa-Leindler inequality).These results
allow for instance \cite{MR2857249} %BobkovMadiman:11a} 
to prove sharp concentration
results for the information of a log-concave vector.

\subsection{Sampling from log concave distributions; convergence of Markov
chain Monte Carlo algorithms}

Sampling from log-concave distributions has been studied by \cite{MR773927}, 
%devroye (1984). 
\cite{MR2910053} %devroye (2012)
for log-concave densities on $\mathbb{R}$, and by \cite{MR1284987, MR1304785}%
, % Frieze, Kannan, Polson 1994
\cite{MR1682608}, % Frieze, Kannan (1999)
and \cite{MR2309621} %Lov{\'a}sz, L{\'a}szl{\'o} and Vempala
for log-concave densities on $\mathbb{R}^d$; see also \cite{lovaszVempala:03},
\cite{LovaszVempala:06}, \cite{MR1318794}, and %MR1318794
\cite{MR1608200}.

Several different types of algorithms have been proposed: the rejection
sampling algorithm of \cite{MR773927} %Devroye (1984) 
requires knowledge of the mode; see % while a recent improvement
\cite{MR2910053} for some improvements. 
%is straightforward rejection sampling, 
The algorithms proposed by \cite{Gilk:Wild:adap:1992} %Gilks and Wild 
are based on adaptive rejection sampling. The algorithms of \cite{MR1994729} 
%Neal (2003)
and \cite{MR1904830} %Roberts and rosenthal (2002)
involve ``slice sampling''; and the algorithms of \cite{lovaszVempala:03}, 
%Lovasz and Vempala (2003)
\cite{LovaszVempala:06}, % Lovasz and Vempala (2006)
\cite{MR2309621} %Lovasz and vempala (2007)
are based on random walk methods.

Log-concavity and bounds for log-concave densities play an important role in
the convergence properties of MCMC algorithms. For entry points to this
literature, see \cite{Gilk:Wild:adap:1992}, % Gilks and Wild (1992)
\cite{MR1425412}, %Polson 1996
\cite{MR1608152}, %Brooks 1998
\cite{MR1904830}, %Roberts and Rosenthal 2002
\cite{MR1953771}, %Fort, Moulines, roberts, (2003)
\cite{MR2877599}, 
%Jyl{\"a}nki, Pasi and Vanhatalo, Jarno and Vehtari, Aki (2011)
and \cite{MR2977521}. %Rudolf 2012.

\subsection{Laplace approximations}

Let $X_{1},...,X_{n}$ be i.i.d. real-valued random variables with density $q$
and Laplace transform 
\begin{equation*}
\phi \left( s\right) =\mathbb{E}\left[ \exp \left( \theta X_{1}\right) %
\right] \text{ .}
\end{equation*}%
Let $x^{\ast }$ be the upper limit of the support of $q$ and let $\tau >0$
be the upper limit of finiteness of $\phi $. Let us assume that $q$ is
almost log-concave (see \cite{MR1354837} p155) on 
$\left( x_{0},x^{\ast}\right) $ for some $x_{0}<x^{\ast}$. 
This means that there exist two constants $c_{1}>c_{2}>0$ and two functions 
$c$ and $h$ on $\mathbb{R}$ such that 
\begin{equation*}
q\left( x\right) =c\left( x\right) \exp \left( -h\left( x\right) \right) 
\text{, \ \ }x<x^{\ast }\text{ ,}
\end{equation*}%
where $c_{2}<c\left( x\right) <c_{1}$ whenever $x>x_{0}$ and $h$ is convex.
In particular, log-concave functions are almost log-concave for $%
x_{0}=-\infty $. Now, fix $y\in \mathbb{R}$. The saddlepoint $s$ associated
to $y$ is defined by
\begin{equation*}
\left( \frac{d}{dt}\log \phi \right) \left( s\right) =y
\end{equation*}%
and the variance $\sigma ^{2}\left( s\right) $ is defined to be%
\begin{equation*}
\sigma ^{2}\left( s\right) =\left( \frac{d^{2}}{dt^{2}}\log \phi \right)
\left( s\right) \text{ .}
\end{equation*}%
Let us write $f_{n}$ the density of the empirical mean $\overline{X}%
=\sum_{i=1}^{n}X_{i}/n$. By Theorem 1\ of \cite{MR1094278}, if $q\in
L^{\zeta }\left( \lambda \right) $ for $1<\zeta <2$, then the following
saddlepoint approximations hold uniformly for $s_{0}<s<\tau $ for any $%
s_{0}>0$: 
\begin{equation*}
f_{n}\left( y\right) =\sqrt{\frac{n}{2\pi \sigma ^{2}\left( s\right) }}\phi
\left( s\right) ^{n}\exp \left( -nsy\right) \left\{ 1+O\left( \frac{1}{n}%
\right) \right\}
\end{equation*}%
and 
\begin{equation*}
\mathbb{P}\left( \overline{X}>y\right) =\frac{\phi \left( s\right) ^{n}\exp
\left( -nsy\right) }{s\sigma \left( s\right) \sqrt{n}}\left\{ B_{0}\left(
s\sigma (s)\sqrt{n}\right) +O\left( \frac{1}{n}\right) \right\}
\end{equation*}%
where $B_{0}(z)=z\exp \left( z^{2}/2\right) \left( 1-\Phi \left( z\right)
\right) $ with $\Phi $ the standard normal distribution function. According
to \cite{MR1094278}, this result extends to the multidimensional setting
where almost log-concavity is required on the entire space (and not just on
some directional tails). As detailed in \cite{MR1354837}, saddlepoint
approximations have many applications in statistics, such as in testing or
Markov chain related estimation problems.

As Bayesian methods are usually expensive in practice, approximations of
quantities linked to the prior/posterior densities are usually needed. In
connection with Laplace's method, log-quadratic approximation of densities
are especially suited when considering log-concave functions, see \cite%
{MR1354837}, %Jensen 1995
\cite{BarberWilliams:97}, \cite{Minka:01}, and references therein.

\subsection{Machine learning algorithms and Gaussian process methods}

%Boughorbel et al. 
\cite{BougTarBouj:05} used the radius margin bound of \cite{MR1719582}\ on
the performance of a Support Vector Machine (SVM) in order to tune
hyper-parameters of the kernel. More precisely they proved that for a
weighted $L^{1}$ -distance kernel the radius is log-convex while the margin
is log-concave. Then they used this fact to efficiently tune the
multi-parameter of the kernel through a direct application of the Convex
ConCave Procedure (or CCCP) due to \cite{yuilRanga:03}. In contrast to the
gradient descent technique (\cite{ChapVapBouMukh:02}), \cite{BougTarBouj:05}
show that a variant of the CCCP which they call Log Convex ConCave Procedure
(or LCCP) ensures that the radius margin bound decreases monotonically and
converges to a local minimum without a search for the size step.

Bayesian methods based on Gaussian process priors have become popular in
statistics and machine learning: see, for example \cite{Seeger:04}, \cite%
{MR2773550}, \cite{MR2418663}, and \cite{MR2819028}. These methods require
efficient computational techniques in order to be scalable, or even
tractable in practice. Thus, log-concavity of the quantities of interest
becomes important in this area, since it allows efficient optimization
schemes.

In this context, \cite{paninski:04} shows that the predictive densities
corresponding to either classification, regression, density estimation or
point process intensity estimation models, are log-concave given any
observed data. Furthermore, in the density and point process intensity
estimation, the likelihood is log-concave in the hyperparameters controlling
the mean function of the Gaussian prior. In the classification and
regression settings, the mean, covariance and observation noise parameters
are log-concave. As noted in \cite{paninski:04}, the results still hold for
much more general prior distributions than Gaussian: it suffices that the prior
and the noise (in models where a noise appears) are jointly log-concave. The
proofs are based on preservation properties for log-concave functions such
as pointwise limit or preservation by marginalization.

\subsection{Compressed sensing and random matrices}

Compressed sensing, aiming at reconstructing sparse signals from incomplete
measurement, is extensively studied since the seminal works of 
\cite{MR2241189}, \cite{MR2230846}\ and \cite{MR2300700}. As detailed
in \cite{MR3113826}, %chafai2011interactions},\ 
compressed sensing is intimately linked
to the theory of random matrices. The matrices ensembles that are most
frequently used and studied are those linked to Gaussian matrices, Bernoulli
matrices and Fourier (sub-)matrices.

By analogy with the Wishart Ensemble, the Log-concave Ensemble is defined in 
\cite{MR2601042} to be the set of squared $n\times n$
matrices equipped with the distribution of $AA^{\ast }$, where $A$ is a $%
n\times N$ matrix with i.i.d. columns that have an isotropic log-concave
distribution. \cite{MR2601042} show\ that the Log-concave
Ensemble satisfies a sharp Restricted Isometry Property (RIP), see also 
\cite{MR3113826} Chapter 2.

\subsection{Log-concave and s-concave as nonparametric function classes in
statistics}

\label{subsec:NonparametricStatistics}

Nonparametric estimation of log-concave densities was initiated by 
\cite{MR1941467} %Walther 
in the context of testing for unimodality. For log-concave densities on
$\mathbb{R}$ it has been explored in more detail by \cite{MR2546798}, 
%Duembgen and Rufibach 
\cite{MR2509075}, %Balabdoui Rufibach and Wellner
and recent results for estimation of log-concave densities on $\mathbb{R}%
^{d} $ have been obtained by \cite{MR2645484}, %Cule and Samworth (2010
\cite{MR2758237}, %  Cule Sawmworth & Stewart 2010
\cite{MR2816336}. \cite{MR2758237} formulate the problem of computing the
maximum likelihood estimator of a multidimensional log-concave density as a
non-differentiable convex optimization problem and propose an algorithm that
combines techniques of computational geometry with Shor's r-algorithm to
produce a sequence that converges to the estimator. An R version of the
algorithm is available in the package LogConcDEAD: \textsl{Log-Concave
Density Estimation in Arbitrary Dimensions}, with further description of the 
algorithm given in \cite{CuleGramacySamworth:2008}. 
Nonparametric estimation of $s-$%
concave densities has been studied by \cite{MR2766867}. 
%Seregin and Wellner
They show that the MLE exists and is Hellinger consistent. 
\cite{Doss-Wellner:2013} have obtained Hellinger rates of convergence for the
maximum likelihood estimators of log-concave and $s-$concave densities on 
$\mathbb{R}$, while 
\cite{Kim-Samworth:14} study Hellinger rates of convergence for the MLEs of log-concave 
densities on $\RR^d$.   
\cite{hen+ast06} %Henn and ast ??  
consider replacement of Gaussian errors by log-concave error distributions
in the context of the Kalman filter.

\cite{MR2757433} gives a review of some of the recent progress.

%----------------------------------------------------------
\section{Appendix A: Brascamp-Lieb inequalities and more}
\label{sec:AppendixA}
%----------------------------------------------------------

Let $\mathbf{X}$ have distribution $P$ with density $p=\exp (-\varphi )$ on $%
\mathbb{R}^{d}$ where $\varphi $ is strictly convex and $\varphi \in C^{2}(%
\mathbb{R}^{d});$ thus ${\nabla ^{2}\mathstrut }(\varphi )\left( x\right)
=\varphi ^{\prime \prime }(x)>0,$ $x\in \mathbb{R}^{d}$ as symmetric
matrices. Let $G,H$ be real-valued functions on $\mathbb{R}^{d}$ with 
$G,H\in C^{1}(\mathbb{R}^{d})\cap L_{2}(P)$. 
We let $H_{1} (P) $ denote
the set of functions $f$ in $L_{2}\left( P\right) $ such that $\nabla f$ (in
the distribution sense) is in $L_{2}\left( P\right) $.

Let $\mathbf{Y}$ have distribution $Q$ with density $q=\psi ^{-\beta }$ on
an open, convex set $\Omega \subset \mathbb{R}^{d}$ where $\beta >d$ and $%
\psi $ is a positive, strictly convex and twice continuously differentiable
function on $\Omega $. In particular, $Q$ is $s=-1/\left( \beta -d\right) $%
-concave (see Definition \ref{defn:Sconcave} and \cite{MR0388475}, \cite%
{MR0404559}). Let $T$ be a real-valued function on $\mathbb{R}^{d}$ with $%
T\in C^{1}(\Omega )\cap L_{2}(Q)$. The following Proposition summarizes a
number of analytic inequalities related to a Poincar\'{e}-type inequality
from \cite{MR0450480}. Such inequalities are deeply connected to
concentration of measure and isoperimetry, as exposed in \cite%
{bakry2014analysis}. Concerning log-concave measures, these inequalities are
also intimately linked to the geometry of convex bodies. Indeed, as noted by 
\cite{CarlenCordero-ErausquinLieb} page 9,

\begin{quote}
\textquotedblleft The Brascamp-Lieb inequality (1.3), as well as inequality
(1.8), have connections with the geometry of convex bodies. It was observed
in [2] (\cite{MR1800062}) %Bobkov and Ledoux
that (1.3) (\textsl{see Proposition~\ref{BLInequalityPlus}, (a)}) can be
deduced from the Pr\'{e}kopa-Leindler inequality (which is a functional form
of the Brunn-Minkowski inequality). But the converse is also true: the Pr%
\'{e}kopa theorem follows, by a local computation, from the Brascamp-Lieb
inequality (see [5] (\cite{MR2115450}) where the procedure is explained in
the more general complex setting). To sum up, the Brascamp-Lieb inequality
(1.3) can be seen as the \textsl{local} form of the Brunn-Minkowski
inequality for convex bodies.\textquotedblright
\end{quote}

\begin{proposition}
\label{BLInequalityPlus} $\phantom{bla}$\newline
(a) \cite{MR0450480}: If $p$ is strictly log-concave, then 
\begin{equation*}
\var(G(\mathbf{X}))\leq E\left[ \nabla G(\mathbf{X})^{T}(\varphi ^{\prime \prime
}(\mathbf{X}))^{-1}\nabla G(\mathbf{X})\right] \text{ }.
\end{equation*}%
(b) If $p=\exp (-\varphi )$ where $\varphi ^{\prime \prime }\geq cI$ with $c>0$, then 
\begin{equation*}
\var (G(\mathbf{X}))\leq \frac{1}{c}E|\nabla G(\mathbf{X})|^{2}.
\end{equation*}%
(c) \cite{MR2376572}: If $\varphi \in L_{2}\left( P\right) $, then for all $f\in H_{1} (P)$, 
\begin{equation*}
\var
\left( f\left( \mathbf{X}\right) \right) 
\leq E\left[ \nabla f(\mathbf{X})^{T}(\varphi ^{\prime \prime }(\mathbf{X}))^{-1}\nabla f(\mathbf{X})\right]
         -\frac{1+a/b}{d}\left( \cov
          \left( \varphi \left( \mathbf{X}\right) ,f\left( \mathbf{X}\right) \right) \right) ^{2}\text{ ,}
\end{equation*}
where 
\begin{equation*}
a=\inf_{x\in \mathbb{R}^{d}}\min \left\{ \lambda \text{ eigenvalue of }%
\varphi ^{\prime \prime }\left( x\right) \right\}
\end{equation*}
and
\begin{equation*}
b=\sup_{x\in \mathbb{R}^{d}}\max \left\{ \lambda \text{ eigenvalue of }%
\varphi ^{\prime \prime }\left( x\right) \right\} \text{ .}
\end{equation*}
Notice that $0\leq a\leq b\leq +\infty $ and $b>0$.

(d) \cite{MR2510011}: If $U=\psi T$, then 
\begin{equation*}
\left( \beta +1\right) 
%TCIMACRO{\TeXButton{var}{\var}}%
%BeginExpansion
\var%
%EndExpansion
(T(\mathbf{Y}))\leq E\left[ \frac{1}{V\left( \mathbf{Y}\right) }\nabla U(%
\mathbf{Y})^{T}(\varphi ^{\prime \prime }(\mathbf{Y}))^{-1}\nabla U(\mathbf{Y%
})\right] +\frac{n}{\beta -n}E\left[ T(\mathbf{Y})\right] ^{2}\text{ .}
\end{equation*}%
Taking $\psi =\exp \left( \varphi /\beta \right) $ and setting $R_{\varphi
,\beta }\equiv \varphi ^{\prime \prime }+\beta ^{-1}\nabla \varphi \otimes
\nabla \varphi $, this implies that for any $\beta \geq d$,%
\begin{equation*}
%TCIMACRO{\TeXButton{var}{\var}}%
%BeginExpansion
\var%
%EndExpansion
(G(\mathbf{X}))\leq C_{\beta }E\left[ \nabla G(\mathbf{X})^{T}(R_{\varphi
,\beta }(\mathbf{X}))^{-1}\nabla G(\mathbf{X})\right] \text{ ,}
\end{equation*}%
where $C_{\beta }=\left( 1+\sqrt{\beta +1}\right) ^{2}\left/ \beta \right. $%
. Notice that $1\leq C_{\beta }\leq 6$.

(e) \cite{MR1307413}: If $p=\exp (-\varphi )$ where $\varphi ^{\prime \prime
}\geq cI$ with $c>0$, then 
\begin{equation*}
\mbox{Ent}_{P}\left( G^{2}(\mathbf{X})\right) \leq \frac{1}{c}E|\nabla G(%
\mathbf{X})|^{2}.
\end{equation*}%
where 
\begin{equation*}
\mbox{Ent}_{P}(Y^{2})=E\left[ Y^{2}\log (Y^{2})\right] -E\left[ Y^{2}\right]
\log (E\left[ Y^{2}\right] ).
\end{equation*}%
(f) \cite{MR1600888}, \cite{MR1849347}: % Ledoux 
If the conclusion of (e) holds for all smooth $G$, then $E_{P}\exp (\alpha |%
\mathbf{X}|^{2})<\infty $ for every $\alpha <1/(2c)$.\newline
(g) \cite{MR1742893}: If $E_{P}\exp (\alpha |\mathbf{X}|^{2})<\infty $ for a
log-concave measure $P$ and some $\alpha >0$, then the conclusion of (e)
holds for some $c=c_{d}$.\newline
(h) \cite{MR1800062}: If $\varphi $ is strongly convex with respect to a
norm $\Vert \cdot \Vert $ (so $p$ is strongly log-concave with respect to $%
\Vert \cdot \Vert $), then 
\begin{equation*}
\mbox{Ent}_{P}(G^{2}(\mathbf{X}))\leq \frac{2}{c}E_{P}\Vert \nabla G(\mathbf{%
X})\Vert _{\ast }^{2}
\end{equation*}%
for the dual norm $\Vert \cdot \Vert _{\ast }$.
\end{proposition}

Inequality (a) originated in \cite{MR0450480} and the original proof of the
authors is based on a dimensional induction. For more details about the
induction argument used by \cite{MR0450480}, see \cite%
{CarlenCordero-ErausquinLieb}. Building on \cite{Maurey:91}, \cite{MR1800062}
give a non-inductive proof of (a) based on the Pr\'{e}kopa-Leindler theorem 
\cite{MR0315079}, \cite{MR0404557}, \cite{MR2199372} which is the
functional form of the celebrated Brunn-Minkowski inequality. The converse
is also true in the sense that the Brascamp-Lieb inequality (a) implies the
Pr\'{e}kopa-Leindler inequality, see \cite{MR2115450}. Inequality (b) is an
easy consequence of (a) and is referred to as a Poincar\'{e} inequality for
strongly log-concave measures.

Inequality (c) is a reinforcement of the Brascamp-Lieb inequality (a) due to 
\cite{MR2376572}. The proof is based on (Marvovian) semi-group techniques,
see \cite{bakry2014analysis} for a comprehensive introduction to these
tools. In particular, \cite{MR2376572}, Lemma 7, gives a variance
representation for strictly log-concave measures that directly implies the
Brascamp-Lieb inequality (a).

The first inequality in (d) is referred in \cite{MR2510011} 
as a ``weighted Poincar\'{e}-type inequality" for convex (or $s-$concave with negative
parameter $s$) measures. 
It implies the second inequality of (d) which is a
quantitative refinement of the Brascamp-Lieb inequality (a). Indeed,
Inequality (a) may be viewed as the limiting case in the second inequality
of (d) for $\beta \rightarrow +\infty $ (as in this case 
$C_{\beta }\rightarrow 1$ and 
$R_{\varphi ,\beta }\rightarrow \varphi ^{\prime \prime } $). 
As noted in \cite{MR2510011}, for finite $\beta $ the second inequality of
(d) may improve the Brascamp-Lieb inequality in terms of the decay of the
weight. For example, when $Y$ is a random variable with exponential
distribution with parameter $\lambda >0$ 
($q\left( y\right) =\lambda e^{-\lambda y}$ on $\Omega =\left( 0,\infty \right) $), 
the second inequality in (d) gives the usual Poincar\'{e}-type inequality,%
\begin{equation*}
\var (G(Y))\leq \frac{6}{\lambda ^{2}}E\left[ \left( G^{\prime }(Y)\right) ^{2}
\right] \text{,}
\end{equation*}
which cannot be proved as an direct application of the Brascamp-Lieb
inequality (a).  
\cite{MR799431} %Klaassen (1985) 
shows that the inequality in the last display holds (in the exponential case) with 
$6$ replaced by $4$, and establishes similar results for other distributions.  
The exponential and two-sided exponential (or Laplace) distributions are also 
treated by \cite{MR1440138}.

Points (e) to (h) deal, in the case of (strongly) log-concave measures, with
the so-called logarithmic-Sobolev inequality, which is known to strengthen
the Poincar\'{e} inequality (also called spectral gap inequality) (see for
instance Chapter 5 of \cite{bakry2014analysis}). Particularly, 
\cite{MR1800062} proved their result in point (d), \textit{via} 
the use of the Pr\'{e}kopa-Leindler inequality. In their survey on optimal transport, 
\cite{MR2895086} show how to obtain the result of Bobkov and Ledoux from some
transport inequalities.

We give now a simple application of the Brascamp-Lieb inequality (a), that
exhibits its relation with the Fisher information for location.

\begin{example}
\label{LinearG} Let $G(x)=a^{T}x$ for $a\in \mathbb{R}^{d}$. Then the
inequality in (a) becomes 
\begin{equation}
a^{T}\cov(\mathbf{X})a\leq a^{T}E\{[\varphi ^{\prime \prime }(\mathbf{X}%
)]^{-1}\}a  \label{CovMatrixBound}
\end{equation}%
or equivalently 
\begin{equation*}
\cov(\mathbf{X})\leq E\{[\varphi ^{\prime \prime }(\mathbf{X})]^{-1}\}
\end{equation*}%
with equality if $\ \mathbf{X}\sim N_{d}(\mu ,\Sigma )$ with $\Sigma $
positive definite. When $d=1$ (\ref{CovMatrixBound}) becomes 
\begin{equation}
\cov(\mathbf{X})\leq E[(\varphi ^{\prime \prime })^{-1}(\mathbf{X}%
)]=E[((-\log p)^{\prime \prime })^{-1}(\mathbf{X})]  \label{VarianceBound}
\end{equation}%
while on the other hand 
\begin{equation}
\cov(\mathbf{X})\geq \lbrack E(\varphi ^{\prime \prime })(\mathbf{X}%
)]^{-1}\equiv I_{loc}^{-1}(\mathbf{X})  \label{CRVarianceLowerBound}
\end{equation}%
where $I_{loc}(X)=E(\varphi ^{\prime \prime })$ denotes the Fisher
information for location (in $X$ or $p$); in fact for $d\geq 1$ 
\begin{equation*}
\cov(\mathbf{X})\geq \lbrack E(\varphi ^{\prime \prime })(\mathbf{X}%
)]^{-1}\equiv I_{loc}^{-1}(\mathbf{X})
\end{equation*}%
where $I_{loc}(X)\equiv E(\varphi ^{\prime \prime })$ is the Fisher
information matrix (for location). If $X\sim N_{d}(\mu ,\Sigma )$ then
equality holds (again). On the other hand, when $d=1$ and $p$ is the
logistic density given in Example~\ref{LogisticDensity}, then $\varphi
^{\prime \prime }=2p$ so the right side in (\ref{VarianceBound}) becomes $%
E\{(2p(\mathbf{X}))^{-1}\}=\int_{\mathbb{R}}(1/2)dx=\infty $ while $Var(%
\mathbf{X})=\pi ^{2}/3$ and $I_{loc}(\mathbf{X})=1/3$ so the inequality (\ref%
{VarianceBound}) holds trivially, while the inequality (\ref%
{CRVarianceLowerBound}) holds with strict inequality: 
\begin{equation*}
3=I_{loc}^{-1}(\mathbf{X})<\frac{\pi ^{2}}{3}=Var(\mathbf{X})<E[(\varphi
^{\prime \prime })^{-1}(\mathbf{X})]=\infty .
\end{equation*}%
(Thus while $X$ is slightly inefficient as an estimator of location for $p$,
it is not drastically inefficient.)
\end{example}

Now we briefly summarize the asymmetric Brascamp - Lieb inequalities of \cite%
{Menz-Otto:2013} and \cite{CarlenCordero-ErausquinLieb}.

\begin{proposition}
\label{prop:MenzOtto} $\phantom{bla}$\newline
(a) \cite{Menz-Otto:2013}: Suppose that $d=1$ and $G, H \in C^{1} (\mathbb{R}%
) \cap L^2 ( P)$. If $p$ is strictly log-concave and $1/r + 1/s = 1$ with $%
r\ge 2$, then \newline
\begin{eqnarray*}
| \cov ( G(X), H(X)) | \le \sup_x \left \{ \frac{ | H^{\prime}(x)|}{%
\varphi^{\prime \prime} (X) } \right \} E \{ | G^{\prime}(X) | \} .
\end{eqnarray*}
(b) \cite{CarlenCordero-ErausquinLieb}: If $p$ is strictly log-concave on $%
\mathbb{R}^d$ and $\lambda_{min}(x)$ denotes the smallest eigenvalue of $%
\varphi^{\prime \prime}$, then 
\begin{eqnarray*}
| \cov (G(X), H(X) ) | \le \| ( \varphi^{\prime \prime} )^{-1/r} \nabla G
\|_s \cdot \| \lambda_{min}^{(2-r)/r} (\varphi^{\prime \prime})^{-1/r}
\nabla H \|_r .
\end{eqnarray*}
\end{proposition}

\begin{remark}
(i) When $r=2$, the inequality in (b) yields 
\begin{equation*}
\left ( \cov (G(X), H(X) ) \right )^2 \le E \{ \nabla G^T (
\varphi^{\prime\prime})^{-1} \nabla G \} \cdot E \{ \nabla G^T (
\varphi^{\prime\prime})^{-1} \nabla G \}
\end{equation*}
which can also be obtained from the Cauchy-Schwarz inequality and the
Brascamp-Lieb inequality (a) of Proposition~\ref{BLInequalityPlus}.\newline
(ii) The inequality (b) also implies that 
\begin{eqnarray*}
| \cov (G(X), H(X) ) | \le \| ( \lambda_{min}^{-1/r} \nabla G \|_s \cdot \|
\lambda_{min}^{-1/s} \nabla H \|_r ;
\end{eqnarray*}
taking $r=\infty$ and $s=1$ yields 
\begin{equation*}
| \cov (G(X), H(X) ) | \le \| \nabla G \|_1 \cdot \| \lambda_{min}^{-1}
\nabla H \|_\infty
\end{equation*}
which reduces to the inequality in (a) when $d=1$.
\end{remark}

%----------------------------------------------------------
\section{Appendix B: some further proofs}
\label{sec:Proofs}
%---------------------------------------------------------- 

\begin{proof}
Proposition~\ref{LogConcaveEqualsPF2}: \textbf{(b):} $p_{\theta} (x) =
f(x-\theta)$ has MLR if and only if 
\begin{eqnarray*}
\frac{f(x-\theta^{\prime})}{f(x-\theta)} \le \frac{f(x^{\prime}-\theta^{%
\prime})}{f(x^{\prime}-\theta)} \ \ \mbox{for all} \ \ x < x^{\prime}, \
\theta < \theta^{\prime}
\end{eqnarray*}
This holds if and only if 
\begin{eqnarray}
\log f(x-\theta^{\prime}) + \log f(x^{\prime}-\theta) \le \log f(x^{\prime}-
\theta^{\prime}) + \log f (x-\theta) .  \label{MLR1stIFF}
\end{eqnarray}
Let $t = (x^{\prime}-x)/(x^{\prime}-x + \theta^{\prime}-\theta)$ and note
that %\pagebreak
% \vs{-1}
\begin{eqnarray*}
&& x-\theta = t(x-\theta^{\prime}) + (1-t) (x^{\prime}- \theta), \\
&& x^{\prime}- \theta^{\prime}= (1-t) (x - \theta^{\prime}) + t (x^{\prime}-
\theta)
\end{eqnarray*}
Hence log-concavity of $f$ implies that 
\begin{eqnarray*}
&& \log f(x-\theta) \ge t \,\log f(x-\theta^{\prime}) + (1-t) \log
f(x^{\prime}- \theta) , \\
&& \log f(x^{\prime}- \theta^{\prime}) \ge (1-t) \log f(x-\theta^{\prime}) +
t\, \log f(x^{\prime}-\theta) .
\end{eqnarray*}
Adding these yields (\ref{MLR1stIFF}); i.e. $f$ log-concave implies $%
p_{\theta} (x)$ has MLR in $x$.

Now suppose that $p_{\theta} (x)$ has MLR so that (\ref{MLR1stIFF}) holds.
In particular that holds if $x,x^{\prime},\theta, \theta^{\prime}$ satisfy $%
x-\theta^{\prime}=a < b = x^{\prime}-\theta$ and $t =
(x^{\prime}-x)/(x^{\prime}-x+\theta^{\prime}-\theta) =1/2$, so that $%
x-\theta = (a+b)/2 = x^{\prime}-\theta^{\prime}$. Then (\ref{MLR1stIFF})
becomes 
\begin{equation*}
\log f(a) + \log f(b) \le 2 \log f ((a+b)/2) .
\end{equation*}
This together with measurability of $f$ implies that $f$ is log-concave. 
%\pagebreak
%------------------------------------------------------------------------------------

\noindent \textbf{(a):} \ Suppose $f$ is $PF_2$. Then for $x< x^{\prime}$, $%
y< y^{\prime}$, 
\begin{eqnarray*}
\lefteqn{\det \left ( \begin{array}{l l} f(x-y) & f(x- y') \\ f(x' - y ) &
f(x' - y') \end{array} \right ) } \\
&& = f(x-y) f(x^{\prime}- y^{\prime}) - f(x - y^{\prime}) f(x^{\prime}- y)
\ge 0
\end{eqnarray*}
if and only if 
\begin{eqnarray*}
&& f(x-y^{\prime}) f(x^{\prime}-y) \le f(x-y) f(x^{\prime}-y^{\prime}) ,
\end{eqnarray*}
or, if and only if 
\begin{eqnarray*}
&& \frac{f(x - y^{\prime})}{f(x-y)} \le \frac{f(x^{\prime}-y^{\prime})}{%
f(x^{\prime}-y)} .
\end{eqnarray*}
That is, $p_{y} (x)$ has MLR in $x$. By (b) this is equivalent to $f$
log-concave.
\end{proof}

\begin{proof}
Proposition~\ref{EquivalencesStrongLogConTry3}:  
To prove Proposition~\ref{EquivalencesStrongLogConTry3} it suffices to note
the log-concavity of $g(x)=p(x)/\prod_{j=1}^{d}\phi (x_{j}/\sigma )$ and to
apply Proposition~\ref{EquivalencesLogCon} (which holds as well for
log-concave functions). The claims then follow by basic calculations.

Here are the details. Under the assumption that $\varphi \in C^{2}$ (and
even more generally) the equivalence between (a) and (b) follows from \cite%
{MR1491362}, Exercise 12.59, page 565. The equivalence between (a) and (c)
follows from the corresponding proof concerning the equivalence of (a) and
(c) in Proposition~\ref{EquivalencesLogCon}; see e.g. \cite{MR2061575}, page
71.

(a) implies (d): this follows from the corresponding implication in
Proposition~\ref{EquivalencesLogCon}. Also note that 
%The second claim follows by direct calculation:  
for $x_{1},x_{2}\in \mathbb{R}^{d}$ we have 
\begin{eqnarray*}
\lefteqn{\langle \nabla \varphi _{J_{a}}(x_{2})-x_{2}/c-\left( \nabla
\varphi _{J_{a}}(x_{1})-x_{1}/c\right) ,x_{2}-x_{1}\rangle } \\
&=&\langle \nabla \varphi (a+x_{2})-\nabla \varphi (a-x_{2})-x_{2}/c-\left(
\nabla \varphi (a+x_{1})-\nabla \varphi (a-x_{1})-x_{1}/c\right)
,x_{2}-x_{1}\rangle \\
&=&\langle \nabla \varphi (a+x_{2})-\nabla \varphi
(a+x_{1})-(x_{2}-x_{1})/(2c),x_{2}-x_{1}\rangle \\
&&\ \ -\ \langle \nabla \varphi (a-x_{2})-\nabla \varphi
(a-x_{1})+(x_{2}-x_{1})/(2c),x_{2}-x_{1}\rangle \\
&=&\langle \nabla \varphi (a+x_{2})-\nabla \varphi
(a+x_{1})-(a+x_{2}-(a+x_{1}))/(2c),x_{2}-x_{1}\rangle \\
&&\ \ +\ \langle \nabla \varphi (a-x_{1})-\nabla \varphi
(a-x_{2})-(a-x_{1}-(a-x_{2}))/(2c),a-x_{1}-(a-x_{2})\rangle \\
&\geq &0\ \ \ 
\end{eqnarray*}%
if $c=\sigma ^{2}/2$.

(d) implies (e): this also follows from the corresponding implication in
Proposition~\ref{EquivalencesLogCon}. Also note that when $\varphi \in C^2$
so that $\nabla^2 \varphi $ exists, 
\begin{eqnarray*}
\nabla^2 \varphi_{J_a} (x) -2 I/\sigma^2 & = & \nabla^2 \varphi (a+x) +
\nabla^2 \varphi (a-x) - 2I/\sigma^2 \\
& = & \nabla^2 \varphi (a+x) - I/\sigma^2 + \nabla^2 \varphi (a-x) -
I/\sigma^2 \\
& \ge & 0 + 0 = 0.
\end{eqnarray*}

To complete the proof when $\varphi \in C^2$ we show that (e) implies (c).
Choosing $a=x_0 $ and $x=0$ yields 
\begin{eqnarray*}
0 & \le & \nabla^2 \varphi_{J_a} (0) -2 I/\sigma^2 \\
& = & \nabla^2 \varphi (x_0) + \nabla^2 \varphi (x_0) - 2I/\sigma^2 \\
& = & 2 \left ( \nabla^2 \varphi (x_0) - I/\sigma^2 \right ) ,
\end{eqnarray*}
and hence (c) holds.

To complete the proof more generally, we proceed as in \cite{MR2814377},
page 199: to see that (e) implies (f), let $a = (x_1 + x_2)/2$, $x =
(x_1-x_2)/2$. Since $J_a (\cdot ; g)$ is even and radially monotone, $J_a
(0; g)^{1/2} \ge J_a (x; g)^{1/2}$; that is, 
\begin{equation*}
\{ g(a+0) g(a-0) \}^{1/2} \ge \{ g(a+x) g(a-x) \}^{1/2},
\end{equation*}
or 
\begin{equation*}
g((x_1 +x_2)/2) \ge g(x_1)^{1/2} g(x_2)^{1/2} .
\end{equation*}
Finally (f) implies (a): as in \cite{MR2814377}, page 199 (with ``convex''
changed to ``concave'' three times in the last three lines there): midpoint
log-concavity of $g$ together with lower semicontinuity implies that $g$ is
log-concave, and hence $p$ is strongly log-concave, so (a) holds.
\end{proof}

\begin{proof}
Proposition~\ref{prop_I_SLC}:  
First notice that by Proposition \ref{Lemma_convol_Gauss_multidim}, we may
assume that $f$ is $\mathcal{C}^{\infty }$ 
(so $\varphi $ is also $\mathcal{C}^{\infty }$).

\textbf{(i) }As $I$ is $\mathcal{C}^{\infty }$, we differentiate $I$ twice. We
have $I^{\prime }\left( p\right) =f^{\prime }\left( F^{-1}\left( p\right)
\right) /I\left( p\right) =-\varphi ^{\prime }\left( F^{-1}\left( p\right)
\right) $ and%
\begin{equation}
I^{\prime \prime }\left( p\right) =-\varphi ^{\prime \prime }\left(
F^{-1}\left( p\right) \right) /I\left( p\right) \leq -c^{-1}\left\Vert
f\right\Vert _{\infty }^{-1}\text{ .}  \label{I_second}
\end{equation}%
This gives the first part of \textbf{(i). }The second part comes from the
fact that 
$\left\Vert f\right\Vert _{\infty }^{-1}\geq \sqrt{ \var \left( X\right) }$ 
by Proposition \ref{prop_pointwise_bounds}\ below.

\textbf{(ii)} It suffices to exhibit an example. We take $X\geq 0$, with
density
\begin{equation*}
f\left( x\right) =xe^{-x}\mathbf{1}_{\left( 0, \infty \right) }\left( x\right) \text{ }.
\end{equation*}
Then $f=e^{-\varphi }$ is log-concave (in fact, $f$ log-concave of order $2$%
, see Definition \ref{def_LC_order_alpha}) and not strongly log-concave as,
on the support of $f$, $\varphi ^{\prime \prime }\left( x\right)
=x^{-2}\rightarrow  0 $ as ${x\rightarrow \infty }0$. 
By the equality in (\ref{I_second}) we have
\begin{equation*}
I^{\prime \prime }\left( p\right) =-\frac{\varphi ^{\prime \prime }}{f}%
\left( F^{-1}\left( p\right) \right) \text{ .}
\end{equation*}%
Hence, to conclude it suffices to show that 
$\inf_{x>0}\left\{ \varphi^{\prime \prime }/f\right\} >0$. 
By simple calculations, we have
\begin{equation*}
\varphi ^{\prime \prime }\left( x\right) /f\left( x\right) =x^{-3}e^{x}%
\mathbf{1}_{\left( 0, \infty \right) }\left( x\right) \geq e^3/27 >0\text{ ,}
\end{equation*}
so \textbf{(ii)} is  proved.

\textbf{(iii) }We take 
$f\left( x\right) =\exp \left( -\varphi \right) 
=\alpha ^{-1}\exp \left( -\exp \left( x\right) \right) 1_{\left\{ x>0\right\} }$ 
where $\alpha =\int_{0}^{\infty }\exp \left( -\exp \left( x\right) \right) dx$. 
Then the function $R_{h}$ is $\mathcal{C}^{\infty }$
on $\left( 0,1\right) $ and we have by basic calculations, 
for any $p\in \left( 0,1\right) $,
\begin{equation*}
R_{h}^{\prime }\left( p\right) =f\left( F^{-1}\left( p\right) +h\right)
/f\left( F^{-1}\left( p\right) \right)
\end{equation*}
and%
\begin{equation*}
R_{h}^{\prime \prime }\left( p\right) =\frac{f\left( F^{-1}\left( p\right)
+h\right) }{f\left( F^{-1}\left( p\right) \right) ^{2}}\left( \varphi
^{\prime }\left( F^{-1}\left( p\right) \right) -\varphi ^{\prime }\left(
F^{-1}\left( p\right) +h\right) \right) \text{ .}
\end{equation*}%
Now, for any $x>0$, taking $p=F\left( x\right) $ in the previous identity
gives%
\begin{eqnarray}
R_{h}^{\prime \prime }\left( F\left( x\right) \right) &=&\frac{f\left(
x+h\right) }{f\left( x\right) ^{2}}\left( \varphi ^{\prime }\left( x\right)
-\varphi ^{\prime }\left( x+h\right) \right)  \label{Rh_second} \\
&=&\alpha ^{-1}\exp \left( \exp \left( x\right) \left( 2-\exp \left(
h\right) \right) \right) \cdot \exp \left( x\right) \left( 1-\exp \left(
h\right) \right) \text{ .}  \notag
\end{eqnarray}%
We deduce that if $h>\log 2$ then $R_{h}^{\prime \prime }\left( F\left(
x\right) \right) \rightarrow 0$ whenever $x\rightarrow +\infty $. Taking $%
h_{0}=1$ gives point \textbf{(iii)}.

\textbf{(iv)} For $X$ of density $f\left( x\right) =xe^{-x}\mathbf{1}%
_{\left( 0,+\infty \right) }\left( x\right) $, we have 
$\inf R_{h}^{\prime \prime }\leq -he^{-h}<0$. 
Our proof of the previous fact is based on
identity (\ref{Rh_second})\ and left to the reader.
\end{proof}

\begin{proof}
Proposition~\ref{EquivalencesStrongLogConDef2Try2}:  
Here are the details.  
Under the assumption that $\varphi \in C^2$ (and even more generally) the equivalence of (a) and (b) 
follows from \cite{MR1491362}, Exercise 12.59, page 565.
The equivalence of (a) and (c) follows from the corresponding proof concerning the equivalence of (a) and (c) 
 in Proposition~\ref{EquivalencesLogCon};
see e.g. \cite{MR2061575}, page 71.

That (a) implies (d):  this follows from the 
corresponding implication in Proposition~\ref{EquivalencesLogCon}.
Also note that %The second claim follows by direct calculation:  
for $x_1, x_2 \in \RR^d$ we have
\begin{eqnarray*}
\lefteqn{\langle \nabla \varphi_{J_a} (x_2) - x_2/c - \left ( \nabla \varphi_{J_a} (x_1) - x_1/c\right ) , x_2-x_1 \rangle } \\
& = & \langle \nabla \varphi (a+x_2) - \nabla \varphi(a-x_2) - x_2/c 
                 - \left (\nabla \varphi (a+x_1) - \nabla \varphi(a-x_1) - x_1/c \right ) , x_2 - x_1 \rangle \\
& = & \langle \nabla \varphi (a+x_2) - \nabla \varphi (a+x_1) - (x_2-x_1)/(2c), x_2 - x_1 \rangle \\
&& \ \ - \ \langle \nabla \varphi (a-x_2) - \nabla \varphi (a-x_1) + (x_2-x_1)/(2c), x_2 - x_1 \rangle \\
& = & \langle \nabla \varphi (a+x_2) - \nabla \varphi (a+x_1) - (a+x_2-(a+x_1))/(2c), x_2 - x_1 \rangle \\
&& \ \ + \ \langle \nabla \varphi (a-x_1) - \nabla \varphi (a-x_2) -  (a-x_1- (a-x_2))/(2c), a - x_1 - (a-x_2)  \rangle \\
& \ge & 0 \ \ \ 
\end{eqnarray*}
if $c = \sigma^2/2$.  

(d) implies (e):  this also follows from the corresponding 
implication in Proposition~\ref{EquivalencesLogCon}.  
Also note that when $\varphi \in C^2$ so that $\nabla^2 \varphi $ exists,
\begin{eqnarray*}
\nabla^2 \varphi_{J_a} (x) -2 I/\sigma^2 
& = & \nabla^2 \varphi (a+x) + \nabla^2 \varphi (a-x) - 2I/\sigma^2 \\
& = & \nabla^2 \varphi (a+x) - I/\sigma^2 + \nabla^2 \varphi (a-x) - I/\sigma^2 \\
& \ge & 0 + 0 = 0.
\end{eqnarray*}

To complete the proof when $\varphi \in C^2$ we show that 
(e) implies (c).  Choosing $a=x_0 $ and $x=0$ yields 
\begin{eqnarray*}
0 & \le & \nabla^2 \varphi_{J_a} (0) -2 I/\sigma^2 \\
& = & \nabla^2 \varphi (x_0) + \nabla^2 \varphi (x_0) - 2I/\sigma^2 \\
& = & 2 \left ( \nabla^2 \varphi (x_0) - I/\sigma^2 \right ) ,
\end{eqnarray*}
and hence (c) holds.  

To complete the proof more generally, we proceed as in 
\cite{MR2814377}, page 199:  to see that (e) implies (f), let $a = (x_1 + x_2)/2$, $x = (x_1-x_2)/2$.
Since $J_a (\cdot ; g)$ is even and radially monotone, 
$J_a (0; g)^{1/2} \ge J_a (x; g)^{1/2}$; that is,
$$
\{ g(a+0) g(a-0) \}^{1/2} \ge \{ g(a+x) g(a-x) \}^{1/2},
$$
or
$$
g((x_1 +x_2)/2) \ge g(x_1)^{1/2} g(x_2)^{1/2} .
$$
Finally (f) implies (a):  as in \cite{MR2814377}, page 199 
(with ``convex'' changed to ``concave'' three times in the last three lines there):  
midpoint log-concavity of $g$ together with lower semicontinuity implies that 
$g$ is log-concave, and hence $p$ is strongly log-concave, so (a) holds.
\end{proof}

\begin{proof}
Proposition~\ref{Lemma_convol_Gauss_multidim}: 
\textbf{(i)} This is given by the stability of log-concavity through
convolution.

\textbf{(ii)}This is point \textbf{(b)} of Theorem \ref{StrongLogConPreservOne}.

\textbf{(iii)} We have 
\begin{equation*}
\varphi _{Z}\left( z\right) =-\log \int_{y\in \mathbb{R}^{d}}p\left(
y\right) q\left( z-y\right) dy
\end{equation*}%
and%
\begin{equation*}
\int_{y\in \mathbb{R}^{d}}\left\Vert \nabla q\left( z-y\right) \right\Vert
p\left( y\right) dy=\int_{y\in \mathbb{R}^{d}}\left\Vert z-y\right\Vert
q\left( z-y\right) p\left( y\right) dy<\infty
\end{equation*}%
since $y\mapsto \left\Vert z-y\right\Vert q\left( z-y\right) $ is bounded.
This implies that $p_{Z}>0$ on $\mathbb{R}^{d}$ and%
\begin{eqnarray*}
\nabla \varphi _{Z}\left( z\right) &=&\int_{y\in \mathbb{R}^{d}}\frac{z-y}{%
\sigma ^{2}}\frac{p\left( y\right) q\left( z-y\right) }{\int_{y\in \mathbb{R}%
^{d}}p\left( u\right) q\left( z-u\right) du}dy \\
&=&\sigma ^{-2}\mathbb{E}\left[ \sigma G\left\vert X+\sigma G=z\right. %
\right] \\
&=&\mathbb{E}\left[ \rho _{\sigma G}\left( \sigma G\right) \left\vert
X+\sigma G=z\right. \right] \text{ .}
\end{eqnarray*}%
In the same manner, successive differentiation inside the integral shows
that $\varphi _{Z}$ is $\mathcal{C}^{\infty }$, which gives (\textbf{iii}).

\textbf{(iv)} Notice that%
\begin{eqnarray*}
\lefteqn{\left\vert \left\Vert \int_{z,y\in \mathbb{R}^{d}}\sigma
^{-4}\left( z-y\right) \left( z-y\right) ^{T}p\left( y\right) q\left(
z-y\right) dydz\right\Vert \right\vert } \\
&\leq &\sigma ^{-4}\int_{y\in \mathbb{R}^{d}}\int_{z\in \mathbb{R}%
^{d}}\left\Vert z-y\right\Vert ^{2}q\left( z-y\right) p\left( y\right)
dy<\infty
\end{eqnarray*}%
as $y\mapsto \left\Vert z-y\right\Vert ^{2}q\left( z-y\right) $ is bounded.
Hence the Fisher information $J(Z)$ of $Z$ is finite and we have%
\begin{eqnarray*}
J\left( Z\right) &=&\sigma ^{-4}\int_{z,y\in \mathbb{R}^{d}}\mathbb{E}\left[
\sigma G\left\vert X+\sigma G=z\right. \right] \mathbb{E}\left[ \left(
\sigma G\right) ^{T}\left\vert X+\sigma G=z\right. \right] p\left( y\right)
q\left( z-y\right) dydz \\
&\leq &\sigma ^{-4}\int_{z,y\in \mathbb{R}^{d}}\mathbb{E}\left[ \sigma
G\left( \sigma G\right) ^{T}\left\vert X+\sigma G=z\right. \right] p\left(
y\right) q\left( z-y\right) dydz \\
&=&\sigma ^{-4}\int_{z,y\in \mathbb{R}^{d}}\left( \int_{u\in \mathbb{R}%
^{d}}\left( z-u\right) \left( z-u\right) ^{T}\frac{p\left( u\right) q\left(
z-u\right) }{\int_{y\in \mathbb{R}^{d}}p\left( v\right) q\left( z-v\right) dv%
}du\right) p\left( y\right) q\left( z-y\right) dydz \\
&=&\sigma ^{-4}\int_{z,y\in \mathbb{R}^{d}}\left( \int_{u\in \mathbb{R}%
^{d}}\left( z-u\right) \left( z-u\right) ^{T}p\left( u\right) q\left(
z-u\right) du\right) \frac{p\left( y\right) q\left( z-y\right) }{\int_{v\in 
\mathbb{R}^{d}}p\left( v\right) q\left( z-v\right) dv}dydz \\
&=&\sigma ^{-4}\int_{z\in \mathbb{R}^{d}}\int_{u\in \mathbb{R}^{d}}\left(
z-u\right) \left( z-u\right) ^{T}p\left( u\right) q\left( z-u\right) dudz \\
&=&\int_{u\in \mathbb{R}^{d}}\left( \underset{J\left( \sigma G\right) }{%
\underbrace{\sigma ^{-4}\int_{z\in \mathbb{R}^{d}}\left( z-u\right) \left(
z-u\right) ^{T}q\left( z-u\right) dz}}\right) p\left( u\right) du \\
&=&J\left( \sigma G\right) \text{ ,}
\end{eqnarray*}%
which is (\textbf{iv}).
\end{proof}

\begin{proof}  Proposition~\ref{prop_approx_SLC}:  
The fact that $h_{c}\in SLC_{1}\left( c^{-1},d\right) $ is obvious due to
Definition \ref{strongLogConcaveDefn1}. By Theorem \ref{theorem_exp_tails}\
above, there exist $a>0$ and $b\in \mathbb{R}$ such that%
\begin{equation*}
f\left( x\right) \leq e^{-a\left\Vert x\right\Vert +b},\text{ \ \ }x\in 
\mathbb{R}^{d}.
\end{equation*}%
We deduce that if $X$ is a random vector with density $f$ on $\mathbb{R}^{d}$, 
then $\mathbb{E}\left[ e^{\left( a/2\right) \left\Vert X\right\Vert } \right] < \infty $ and so, for any $\beta >0$, 
\begin{equation*}
\mathbb{P}\left( \left\Vert X\right\Vert >2\beta \right) \leq Ae^{-a\beta }%
\text{ ,}
\end{equation*}%
where $A=\mathbb{E}\left[ e^{\left( a/2\right) \left\Vert X\right\Vert }%
\right] >0$. Take $\varepsilon \in \left( 0,1\right) $. We have$\medskip $

\noindent $%
\begin{array}{l}
\left\vert \int_{\mathbb{R}^{d}}f\left( v\right) e^{-c\left\Vert
v\right\Vert ^{2}/2}dv-1\right\vert  \\ 
=\int_{\mathbb{R}^{d}}f\left( v\right) \left( 1-e^{-c\left\Vert v\right\Vert
^{2}/2}\right) dv \\ 
=\int_{\mathbb{R}^{d}}f\left( v\right) \left( 1-e^{-c\left\Vert v\right\Vert
^{2}/2}\right) \mathbf{1}_{\left\{ \left\Vert v\right\Vert \leq
2c^{-\varepsilon /2}\right\} }dv \\ 
\qquad +\int_{\mathbb{R}^{d}}f\left( v\right) \left( 1-e^{-c\left\Vert
v\right\Vert ^{2}/2}\right) \mathbf{1}_{\left\{ \left\Vert v\right\Vert
>2c^{-\varepsilon /2}\right\} }dv \\ 
\leq \left( 1-e^{-2c^{1-\varepsilon }}\right) \int_{\mathbb{R}^{d}}f\left(
v\right) \mathbf{1}_{\left\{ \left\Vert v\right\Vert \leq \sqrt{2}%
c^{-\varepsilon /2}\right\} }dv+\mathbb{P}\left( \left\Vert X\right\Vert
>2c^{-\varepsilon /2}\right)  \\ 
\leq \left( 1-e^{-2c^{1-\varepsilon }}\right) +Ae^{-ac^{-\varepsilon /2}}%
\text{ .}%
\end{array}%
\medskip $

\noindent We set $B_{\alpha }=\left( 1-e^{-2\alpha ^{1-\varepsilon }}\right)
+Ae^{-a\alpha ^{-\varepsilon /2}}$ and we then have 
\begin{equation*}
\left\vert \int_{\mathbb{R}^{d}}f\left( v\right) e^{-c\left\Vert
v\right\Vert ^{2}/2}dv-1\right\vert \leq B_{c}=O_{c\rightarrow 0}\left(
c^{1-\varepsilon }\right) \rightarrow _{c\rightarrow 0}0\text{ .}
\end{equation*}%
Now, for $x\in \mathbb{R}^{d}$, we have, for all $c>0$ such that $B_{c}<1$,
\begin{eqnarray*}
\lefteqn{ \left\vert h_{c}\left( x\right) -f\left( x\right) \right\vert } \\ 
& = & \left\vert \frac{f\left( x\right) e^{-c\left\Vert x\right\Vert ^{2}/2}}
                              {\int_{\mathbb{R}^{d}}f\left( v\right) e^{-c\left\Vert v\right\Vert ^{2}/2}dv}
              - f\left( x\right) \right\vert  \\ 
& \leq & \left\vert \frac{f\left( x\right) e^{-c\left\Vert x\right\Vert ^{2}/2}}
                                  {\int_{\mathbb{R}^{d}}f\left( v\right) e^{-c\left\Vert v\right\Vert ^{2}/2}dv}
                -f\left( x\right) e^{-c\left\Vert x\right\Vert ^{2}/2}\right\vert
                +\left\vert f\left( x\right) e^{-c\left\Vert x\right\Vert ^{2}/2}-f\left(
                     x\right) \right\vert  \\ 
& \leq & f\left( x\right) \left\vert \left( \int_{\mathbb{R}^{d}}f\left(v\right) 
                e^{-c\left\Vert v\right\Vert ^{2}/2}dv\right) ^{-1}-1\right\vert
               +    f\left( x\right) \left( 1-e^{-c\left\Vert x\right\Vert ^{2}/2}\right)  \\ 
& \leq & f\left( x\right) \left( \frac{B_{c}}{1-B_{c}}+1-e^{-c\left\Vert x\right\Vert ^{2}/2}\right) \text{ .}%
\end{eqnarray*}
Hence, for all $c>0$ such that $B_{c}<1$,%
\begin{eqnarray*}
\lefteqn{\sup_{x\in \mathbb{R}^{d}}\left\vert h_{c}\left( x\right) -f\left( x\right) \right\vert } \\
& \leq &     \sup_{\left\{ x;\left\Vert x\right\Vert \leq 2c^{-\varepsilon/2}\right\} }
                \left\vert h_{c}\left( x\right) -f\left( x\right) \right\vert 
                + \sup_{\left\{ x;\left\Vert x\right\Vert >2c^{-\varepsilon /2}\right\} }\left\vert h_{c}\left( x\right) -f\left( x\right) \right\vert  \\
&\leq & e^{b}\left( \frac{B_{c}}{1-B_{c}}+1-e^{-2c^{1-\varepsilon }}\right)
             +e^{-2ac^{-\varepsilon /2}+b}     \left( \frac{B_{c}}{1-B_{c}}+1\right)  \\
&=& O\left( c^{1-\varepsilon }\right) \text{ as }c\rightarrow 0\text{ .}
\end{eqnarray*}
Furthermore, for $p\in \left[ 1, \infty \right) $,
\begin{eqnarray*}
\lefteqn{ \int_{\mathbb{R}^{d}}\left\vert h_{c}\left( x\right) -f\left( x\right)
           \right\vert ^{p}dx }\\
& = & \int_{\mathbb{R}^{d}}\left\vert h_{c}\left( x\right) -f\left( x\right) \right\vert ^{p}
            \mathbf{1}_{\left\{ \left\Vert x\right\Vert \leq 2c^{-\varepsilon /2}\right\} }dx
            +\int_{\mathbb{R}^{d}}\left\vert h_{c}\left( x\right) -f\left( x\right) \right\vert ^{p}\mathbf{1}_{\left\{ \left\Vert
              x\right\Vert >2c^{-\varepsilon /2}\right\} }dx \\ 
& \leq & \sup_{\left\{ x;\left\Vert x\right\Vert \leq 2c^{-\varepsilon/2}\right\} }\left\vert h_{c}\left( x\right) 
             -f\left( x\right) \right\vert^{p}
             +\int_{\mathbb{R}^{d}}f\left( x\right) ^{p}\left( \frac{B_{c}}{1-B_{c}} +1\right) 
                   \mathbf{1}_{\left\{ \left\Vert x\right\Vert >2c^{-\varepsilon/2}\right\} }dx \\ 
& \leq & e^{pb}\left( \frac{B_{c}}{1-B_{c}}+1-e^{-2c^{1-\varepsilon }}\right)^{p}
             +\left( \frac{B_{c}}{1-B_{c}}+1\right) e^{\left( p-1\right) b}\mathbb{P}
               \left( \left\Vert X\right\Vert >2c^{-\varepsilon /2}\right)  \\ 
& \leq & e^{pb}\left( \frac{B_{c}}{1-B_{c}}+1-e^{-2c^{1-\varepsilon }}\right)^{p}
            +A\left( \frac{B_{c}}{1-B_{c}}+1\right) e^{\left( p-1\right) b}e^{-ac^{-\varepsilon /2}} \\ 
& = & O\left( c^{p\left( 1-\varepsilon \right) }\right) \text{ as }c\rightarrow 0 \text{.}%
\end{eqnarray*}
The proof is now complete.
\end{proof}

%----------------------------------------------------------
\section*{Acknowledgments}
%----------------------------------------------------------

We owe thanks to Michel Ledoux for a number of pointers to the literature and 
for sending us a pre-print version of \cite{BobkovLedoux:14}.  

\bigskip %\nocite{*}
\def\cprime{$'$} \def\cprime{$'$}

%\begin{thebibliography}{24}
%\begin{bibliography}
%\bibliographystyle{imsart-nameyear}
%\bibliography{chern}

\end{document}